\documentclass[reqno, 11pt, oneside]{amsart}
\usepackage{amsmath, mathrsfs, url, color}
\usepackage[margin=2cm, a4paper]{geometry}
\usepackage{enumitem}
\usepackage{bbm}
\usepackage{graphicx}
\usepackage{subcaption}
\allowdisplaybreaks

\newtheorem{thm}{Theorem}[section]

\newtheorem{cor}{Corollary}[section]
\newtheorem{lem}{Lemma}[section]
\newtheorem{prop}{Proposition}[section]
\newtheorem{rem}{Remark}[section]

\newtheorem{assumption}{Assumption}[section]

\title[]{Well-posedness and tamed schemes for McKean--Vlasov Equations with Common Noise}

\author[]{Chaman Kumar, Neelima, Christoph Reisinger, and Wolfgang Stockinger}

\thanks{}

\begin{document}
\begin{abstract}
\noindent
In this paper, we first establish well-posedness of McKean--Vlasov stochastic differential equations (McKean--Vlasov SDEs) with common noise, possibly with coefficients having super-linear growth in the state variable.
Second, we present stable time-stepping schemes for this class of McKean--Vlasov SDEs.
Specifically, we propose an explicit tamed Euler and tamed Milstein scheme for an interacting particle system associated with the McKean--Vlasov equation. We prove stability and strong convergence of order $1/2$ and $1$, respectively. To obtain our main results, we employ techniques from calculus on the Wasserstein space. The proof for the strong convergence of the tamed Milstein scheme only requires the coefficients to be once continuously differentiable in the state and measure component. To demonstrate our theoretical findings, we present several numerical examples, including mean-field versions of the stochastic $3/2$ volatility model and the stochastic double well dynamics with multiplicative noise. 
\end{abstract}
\maketitle
\section{Introduction}
A McKean--Vlasov equation for an $\mathbb{R}^{d}$ valued process $X$ (introduced by H.\ McKean \cite{mckean}) is an SDE where the drift and diffusion coefficients depend on the state of the process and, additionally, on the marginal laws of $X$, i.e., 
\begin{equation}\label{McKeanLimit}
    X_t = X_0 + \int_0^t b_s(X_s, \mathcal{L}_{X_s}) ds + \int_0^t \sigma_s(X_s,\mathcal{L}_{X_s}) dW_s
\end{equation}
almost surely, for any $t \in[0,T]$, with a given $T >0$, where $W$ is an $m$-dimensional Wiener process, $X_0$ is an $\mathbb{R}^d$-valued random variable and $\left( \mathcal{L}_{X_t} \right)_{0 \leq t \leq T}$ denotes the flow of deterministic marginal distributions of $X$. The theory concerning existence and uniqueness results for strong solutions of McKean--Vlasov SDEs with linearly-growing coefficients satisfying Lipschitz type conditions (in the state and measure component) is well-established (see, e.g., \cite{AS}). Additional results related to the existence and uniqueness for weak/strong solutions of McKean--Vlasov SDEs can be found in \cite{BBP, HSS, MVA} and references cited therein. For super-linearly growing drift and diffusion with linear growth, it is known that a McKean--Vlasov SDE admits a unique strong solution, if the drift term satisfies a one-sided Lipschitz condition (see \cite{reis2019b}).

The main interest of the present article are McKean--Vlasov equations with common noise $W^{0}$, an $m^{0}$-dimensional Wiener process, i.e., SDEs with a solution $(X_t)_{0 \leq t \leq T}$ that satisfies 
\begin{align} \label{eq:MVSDE}
X_t=X_0 +\int_0^t b_s\big(X_s, \mathcal{L}^1(X_s)\big) ds &+ \int_0^t \sigma_s\big(X_s,\mathcal{L}^1(X_s)\big)dW_s  +  \int_0^t \sigma_s^0\big(X_s,\mathcal{L}^1(X_s)\big)dW_s^0
\end{align}
almost surely, and $\left( \mathcal{L}^1(X_t)\right)_{0 \leq t \leq T}$ denotes the flow of marginal conditional distributions of $X$ given the common source of noise, (see below for a rigorous probabilistic set-up for this equation). Compared to McKean--Vlasov SDEs without common noise (\ref{McKeanLimit}), the marginal laws are not deterministic anymore. Intuitively speaking, the classical notion of propagation of chaos refers to the fact that in a large network of $N$ interacting particles, these particles become asymptotically independent as $N \to \infty$. If the particles have a common source of randomness, we do not expect them to become asymptotically independent in the limit; however, conditioned on the information generated by the common noise, they are asymptotically independent. To put this intuition in other words, the empirical distribution of the particles is expected to converge towards the common conditional distribution of each particle given $W^{0}$. For more details on this topic, we refer to \cite{carmona2018a, carmona2018b}. 

The existence and uniqueness of strong solutions for McKean--Vlasov SDEs with common noise under Lipschitz type conditions is well-studied, see \cite{carmona2018b}. Also, in the aforementioned reference, the notion of strong solution is made rigorous. Here, we go beyond the classical Lipschitz framework and allow all coefficients of (\ref{eq:MVSDE}) to be locally Lipschitz continuous, as long as a certain coercivity assumption is satisfied. The main difficulty to prove well-posedness of (\ref{eq:MVSDE}), is to identify an appropriate space on which one can define a contraction map, which allows a fixed-point argument. Note that the space considered in \cite{carmona2018b} seems not to be applicable in our model set-up. An interesting study on weak solution for McKean--Vlasov SDEs with common noise can be found in \cite{hammersley2019}. 

Note that our results also provide extensions to known well-posedness results for McKean--Vlasov equations with no common noise, see, e.g., \cite{carmona2018a, reis2019b}. The precise formulation of the assumptions is given in Section 2. The contributions of this paper are the following: 
\begin{itemize}
\item We study the well-posedness of (\ref{eq:MVSDE}) under milder assumptions than those existing in the literature, in particular allowing coefficients which are not globally Lipschitz (see Theorem \ref{thm:eu}).  
\item 
We construct a novel tamed Euler and a tamed Milstein scheme to simulate the interacting particle system associated with (\ref{eq:MVSDE}) and prove strong convergence  of order 1/2 and 1, respectively (see Theorems \ref{thm:rate:euler} and \ref{thm:rate:milstein}, respectively). 
We do not rely on an application of It\^{o}'s formula to derive our results and only require all coefficients to be once continuously differentiable in the state and measure components. 
\end{itemize}

McKean--Vlasov equations with common noise arise as limiting equations of, e.g., $N$-player games, where each individual
player is exposed to an idiosyncratic noise and random shocks common to all the players \cite{carmona2018b}. Recently, McKean--Vlasov equations with common noise have received significant attention in the optimal control literature, see, e.g., \cite{pham2016}, where an optimal control problem for a linear conditional McKean--Vlasov equation with quadratic cost functional was studied. An application of McKean--Vlasov SDEs with common noise in systemic risk was considered in \cite{ledger2019}. As pointed out by the authors, their model could serve as a mean-field model for the interplay between common exposures and contagion in large financial systems. Further literature which motivates the consideration of a common noise source is concerned with the mathematical analysis of a large interacting network of neurons, where McKean--Vlasov equations are employed to describe the voltage level of a typical neuron in a network, see e.g., \cite{ullner2018}.

The simulation of McKean--Vlasov SDEs with common noise will involve two steps: At each time $t$, the unknown measure $\mathcal{L}^1_{X_t}$ is approximated by an empirical measure 
\begin{equation*}
  \mu_s^{X,N}(\mathrm{d}x) := \frac{1}{N}\sum_{j=1}^{N} \delta_{X_t^{j,N}}(\mathrm{d}x),
\end{equation*}    
where $\delta_{x}$ is the Dirac measure at point $x$ and $(X^{i,N})_{1 \leq i \leq N}$ (so-called interacting particles) solves the $\mathbb{R}^{dN}$-dimensional SDE,  
\begin{equation*}
X_t^{i,N}=X_0^{i}+ \int_0^t b_s\big(X_s^{i,N}, \mu_s^{X,N}\big) ds + \int_0^t \sigma_s\big(X_s^{i,N}, \mu_s^{X,N}\big) dW_s^{i} + \int_0^t \sigma^0_s\big(X_s^{i,N}, \mu_s^{X,N}\big) dW_s^{0}.
\end{equation*}
Here, $W^{i}$, $W^{0}$ and $X_{0}^{i}$, $i \in \{1,\ldots, N \}$, are independent Wiener processes (also independent of $W$) and i.~i.~d.\ random initial values with $\text{Law}(X_0^{i}) = \text{Law}(X_0)$, respectively.
In a second step, one needs to introduce a stable time-stepping method to approximate the particle system $(X^{i,N})_{1 \leq i \leq N}$ over some finite time horizon $[0,T]$.

The study of stable time-stepping schemes for interacting particle systems (without common noise) with a drift that grows super-linearly in the state component and with globally Lipschitz continuous diffusion term was initiated in \cite{reis2019a}. In this reference, a tamed Euler scheme in the spirit of \cite{sabanis2013, hutzenthaler2012} was proposed. Here, one requires both coefficients to be globally Lipschitz continuous in the measure component. These ideas were independently extended to higher-order numerical schemes in \cite{bao2020} and \cite{kumar2020}. This paper further relaxes the assumptions imposed on the diffusion, by also allowing this term to grow super-linearly in the state component, similar to the works for classical SDEs, \cite{sabanis2016, kumar2019}. To give a motivation for this extension, we mention that our framework allows to consider mean-field versions of popular models appearing in mathematical finance such as the $3/2$-model that is often used for pricing VIX options and modelling certain stochastic volatility processes, see \cite{goard}. This model will be discussed in more detail in Section \ref{Sec:4}.

To recover the strong convergence of order $1$ for the tamed Milstein scheme a term involving the Lions' derivative of the diffusion term will appear. These additional terms are of theoretical interest and highlight the inherent difference of McKean--Vlasov SDEs to classical SDEs regarding higher order time-stepping schemes. This notion of differentiability for functions on $\mathcal{P}_2(\mathbb{R}^d)$ was introduced by P.-L.\ Lions', see \cite{cardaliaguet2013} and the below subsection for a short introduction to this notion. 
This is the first paper dealing with a Milstein type scheme for such equations; to our knowledge, even for linearly growing coefficients, no such scheme is available in the literature. 

The remainder of this article is organised as follows: In Section 1, we prove existence/uniqueness of strong solutions for (\ref{eq:MVSDE}), merely requiring a coercivity and monotonicity condition on the coefficients. To complete this section, we give a quantatitive strong conditional propagation of chaos result, see e.g., \cite{AS, carmona2018a, carmona2018b}. Section 2 and Section 3 are devoted to introducing a tamed Euler and Milstein scheme for the interacting particle system associated with (\ref{eq:MVSDE}), respectively. In both sections, we demonstrate stability of the schemes and prove the expected strong convergence rates of order $1/2$ and $1$, respectively. In Section \ref{Sec:4}, we present several numerical examples to support our theoretical findings. We note that in each section, we give a list all model assumptions which are needed to derive the main results of the respective section.

In the following subsections, we present several notions and auxiliary results, which will be needed throughout this article. Also, we describe in detail the probabilistic framework we work in to analyse the McKean--Vlasov SDE with common noise (\ref{eq:MVSDE}).

\subsection*{Probabilistic framework}\hfill\\
In the sequel, we introduce the precise probabilistic set-up for (\ref{eq:MVSDE}) and follow closely the description presented in \cite{carmona2018b}.
To distinguish between the two underlying sources of randomness, we introduce the complete probability spaces, $\left(\Omega^{0}, \mathcal{F}^{0}, P^{0} \right)$ and $\left(\Omega^{1}, \mathcal{F}^{1}, P^{1} \right)$, equipped with the filtrations $\mathbb{F}^{0} := (\mathcal{F}^{0}_t)_{t \geq 0}$ and $\mathbb{F}^{1} := (\mathcal{F}^{1}_t)_{t \geq 0}$, satisfying the usual conditions. Here, the Wiener process $W^{0}$ will be supported on $\left(\Omega^{0}, \mathcal{F}^{0}, P^{0} \right)$ and $W$ (and the copies thereof used to define the interacting particle system) will be supported on $\left(\Omega^{1}, \mathcal{F}^{1}, P^{1} \right)$. Consequently this allows to define a product space $\left(\Omega, \mathcal{F}, P \right)$, where $\Omega = \Omega^{0} \times \Omega^{1}$, $\left( \mathcal{F}, P \right)$ is the completion of $\left(\mathcal{F}^{0} \otimes \mathcal{F}^{1}, P^{0} \otimes P^{1} \right)$ and $\mathbb{F}:= (\mathcal{F}_t)_{t \geq 0}$ is the complete and right-continuous augmentation of $\left( \mathcal{F}_t^{0} \otimes \mathcal{F}_t^{1} \right)_{t \geq 0}$. It is known from Lemma 2.4 in \cite{carmona2018b} that for a given random variable $X: \Omega \to \mathbb{R}^d$, the mapping $\mathcal{L}^1(X): \Omega^{0} \ni \omega^{0} \mapsto \mathcal{L}(X(\omega^{0},\cdot))$, is $P^{0}$-almost surely well-defined and is a random variable from $\left(\Omega^{0}, \mathcal{F}^{0}, P^{0} \right)$ into $\mathcal{P}_2(\mathbb{R}^d)$, which can also be seen as a conditional law of $X$ given $\mathcal{F}^{0}$.    

As pointed out in \cite{carmona2018b}, it is a-priori not guaranteed that $\left( \mathcal{L}^1(X_t) \right)_{0 \leq t \leq T}$ is adapted to $\mathbb{F}^{0}$. However, if the $\mathbb{F}$-adapted unique solution $(X_t)_{0 \leq t \leq T}$ of (\ref{eq:MVSDE}) has continuous paths and has uniformly bounded second moment, then one can find a version of $\mathcal{L}^1(X_t)$, for every $t \geq 0$, such that $\left( \mathcal{L}^1(X_t) \right)_{t \geq 0}$ has continuous paths and is $\mathbb{F}^{0}$-adapted, see Lemma 2.5 in \cite{carmona2018b}. 

Further, we remark that the initial value $X_0$ of (\ref{eq:MVSDE}) is assumed to be defined on $\left(\Omega^{1}, \mathcal{F}_0^{1}, P^{1} \right)$, which means that only $W^{0}$ plays the role of the common noise. In light of Proposition 2.9 in \cite{carmona2018b}, $\mathcal{L}^1(X_t)$ is a version of the conditional law of $X_t$ given $W^{0}$. For alternative choices of the initial data, we refer to Remark 2.10 in \cite{carmona2018b}.  

The coefficients $b_{\cdot}$, $\sigma_{\cdot}$ and $\sigma_{\cdot}^0$ appearing in (\ref{eq:MVSDE}) are measurable functions defined on $[0,T] \times \mathbb{R}^d \times \mathcal{P}_2(\mathbb{R}^d)$ taking values in $\mathbb{R}^d$, $\mathbb{R}^{d \times m}$ and $\mathbb{R}^{d \times m^{0}}$, respectively.

\subsection*{Notations}\hfill\\
The Euclidean norm of a $d$-dimensional vector and the Hilbert-Schmidt norm of a $d \times m$-matrix are both denoted by $| \cdot |$. 
For the inner product of two vectors $x, y \in \mathbb{R}^{d}$, we write $xy$. 
Also, $Ax$ stands for  matrix product of $A\in\mathbb{R}^{d \times m}$ and  $x\in\mathbb{R}^m$. 
The transpose of a matrix $A$ will be denoted by $A^{*}$. 
For a vector $x \in \mathbb{R}^d$, $x^{(l)}$ denotes its $l$-th element. For a matrix $A \in\mathbb{R}^{d \times m}$, $A^{(u,v)}$ and $A^{(l)}$ denote its $(u,v)$-th element and $l$-th column, respectively.   
The floor function is denoted by $\lfloor \cdot \rfloor$.
The gradient of a function $f:\mathbb{R}^d \to \mathbb{R}$ is denoted by $\partial_x f$.   
Further, we abbreviate by $\mathcal{P}(\mathbb{R}^d)$ the set of all probability measures on $\mathbb{R}^d$ and define the set of probability measures having finite second moment by,
\begin{equation*}
\mathcal{P}_2(\mathbb{R}^d):= \Big \{ \mu\in \mathcal
{P}(\mathbb{R}^d) : \ \int_{\mathbb{R}^d} |x|^2 \mu(\mathrm{d} x)<\infty \Big \}.
\end{equation*} 
Also, for any $\mu, \nu \in \mathcal{P}_2(\mathbb{R}^d)$, the $L^2$-Wasserstein distance is defined by,
\begin{equation*}
\mathcal{W}_2(\mu, \nu):= \left(\inf_{\pi \in \Pi(\mu,\nu)} \int_{\mathbb{R}^d \times \mathbb{R}^d} |x-y |^2 \pi(\mathrm{d}x,\mathrm{d}y) \right)^{1/2},
\end{equation*}
where $\Pi(\mu,\nu)$ is the set of couplings of $\mu$ and
$\nu$. Clearly, $\mathcal{P}_2(\mathbb{R}^d)$ is a Polish space under the $L^2$-Wasserstein metric. 
Throughout this article, $K>0$ is a generic constant that might change its value from line to line. 

\subsection*{Differentiability of functions of measures}\hfill\\
For functions of measures, there are different notions for differentiability, see, e.g., \cite{ambrosio2008, villani2009}.  
In this article, we use the notion of Lions' derivative.
A real-valued function on $\mathcal{P}_2(\mathbb{R}^d)$ is said to be differentiable at $\nu_0 \in \mathcal{P}_2(\mathbb{R}^d)$ if there exists an $\mathbb{R}^d$-valued square integrable random variable $Y_0$ on some atomless, Polish probability space $(\tilde{\Omega}, \tilde{\mathcal{F}}, \tilde{P})$ such that its law satisfies $\mathcal{L}_{Y_0}:=\tilde{P}\circ Y_0^{-1}=\nu_0$ and the lifting of $f$ defined by $F(Z):=f(\mathcal{L}_Z)$ on $L^2(\tilde{\Omega};\mathbb{R}^d)$ has Fr\'echet derivative $F'[Y_0]$ at $Y_0\in L^2(\tilde{\Omega};\mathbb{R}^d)$. 
The function $f$ is in class $C^1$ if its lifting $F$ is in class $C^1$.   
Further on using the Riesz representation theorem, for the bounded linear operator $F'[Y_0]:L^{2}(\tilde{\Omega};\mathbb{R}^d) \to \mathbb{R}$, there exists a unique element $DF(Y_0)\in L^2(\tilde{\Omega};\mathbb{R}^d)$ satisfying $F'[Y_0](Z)=\tilde{E} \langle DF[Y_0], Z\rangle$ for all $Z\in L^2 (\tilde{\Omega};\mathbb{R}^d)$. Moreover, by Theorem 6.5 (structure of the gradient) in \cite{cardaliaguet2013}, there exits a measurable function $\partial_\mu f(\nu_0): \mathbb{R}^d\to\mathbb{R}^d$, independent of the choice of the probability space $(\tilde{\Omega}, \tilde{\mathcal{F}}, \tilde{P})$ and the random variable $Y_0$ used for the lifting, such that  $\int_{\mathbb{R}^d} |\partial_\mu f(\nu_0)(x)|^2 \nu_0(dx) <\infty$  and $DF(Y_0)=\partial_\mu f(\nu_0)(Y_0)$ holds. 
The function $\partial_\mu f(\nu_0)$ is defined as the \textit{Lions' derivative} (abbreviated by $L$-derivative) of $f$ at $\nu_0=\mathcal{L}_{Y_0}$ and $\partial_\mu f : \mathcal{P}_2(\mathbb{R}^d)\times \mathbb{R}^d\to\mathbb{R}^d$ is given by $\partial_\mu f(\nu,z)=\partial_\mu f(\nu)(z)$ for any $\nu \in\mathcal{P}_2(\mathbb{R}^d)$ and $z\in\mathbb{R}^d$.

\subsection*{Auxiliary Lemmas}\hfill\\
The following lemma is frequently used in this article:
\begin{lem}[Lemma 3.2 in \cite{gyongy2003}] \label{lem:gk}
Let $f:=(f_t)_{t\geq 0}$ and $g:=(g_t)_{t\geq 0}$ be non-negative continuous $\mathcal{F}_t$-adapted processes satisfying, 
\begin{align*}
Ef_\tau I_{g_0\leq c} \leq Eg_\tau I_{g_0\leq c}, 
\end{align*}
for any bounded stopping time $\tau\leq T$ where $c>0$ and $T\in[0,\infty]$ are any constants. Then, for  any bounded stopping time $\tau\leq T$, 
\begin{align*}
E\sup_{t\leq \tau} f_\tau^\gamma I_{g_0\leq c} \leq \frac{2-\gamma}{1-\gamma} Eg_\tau^\gamma I_{g_0\leq c},
\end{align*}
for any $\gamma\in(0,1)$. 
\end{lem}
The above lemma helps us in obtaining uniform second moment bound of certain SDEs when the diffusion coefficient grows super-linearly (see, e.g., the proof of Theorem 2.1) by imposing that the initial value of these SDEs is in $L^{2+\epsilon}$ for some $\epsilon>0$. However, these considerations and additional regularity on the initial value is not required in case the diffusion coefficient grows only linearly. In addition, it allows us to obtain a uniform rate of strong convergence, in Section 3 (see, Theorem 3.1) and Section 4 (see, Theorem 4.1).

The following lemma plays a key role in proving the existence and uniqueness result in Theorem~\ref{thm:eu}. We are grateful to D. \v{S}i\v{s}ka and L. Szpruch, both from University of Edinburgh, for allowing us to use the result from their working paper.
\begin{lem}[\cite{siska2020}] \label{lem:DL}
Let $(X,d)$ be a complete metric space and let $(\Omega,\mathcal{F},P)$ be a probability space. Let $p\geq 1$. Let $\Vert \cdot \Vert_p$ denote the norm in $L^p(\Omega;\mathbb{R})$. Let $\mu_0:\Omega\to X $ be a random variable. Let
\[
S:=\{\mu:\Omega\to X\,\,\, r.v. : \Vert d(\mu, \mu_0) \Vert_p < \infty\}.
\]
Let $\rho(\mu,\mu'):=\Vert d(\mu, \mu') \Vert_p$. Then $(S,\rho)$ is a complete metric space.
\end{lem}





\section{Existence, Uniqueness, Moment Bound and Propagation of Chaos}
In this section, the existence and uniqueness of the solution of \eqref{eq:MVSDE} is proved under more relaxed assumptions than those existing in the literature, see for example \cite{carmona2018b}.
Indeed, this is the first result on existence and uniqueness of \eqref{eq:MVSDE} where the coefficients are allowed to grow super-linearly. As a by product when $\sigma^0=0$ (wihout common noise), our result is an extension of \cite{reis2019b} in the sense that we also allow super-linear diffusion coefficient for equation \eqref{McKeanLimit} by assuming slightly more regular initial value, i.e., $X_0\in L^{2+\epsilon}$ for any $\epsilon>0$.
Further, moment boundedness of the solution is established. 
 
 The following assumptions are made in this section. 
\begin{assumption} \label{as:x0}
  $E|X_0|^{p_0}< \infty $ for a fixed constant $p_0>2$. 
 \end{assumption}
  \begin{assumption} \label{as:coercivity}
  There exists a constant $L>0$ such that 
  \begin{align*}
  2 x b_t(x,\mu)  & +  (p_0-1) \big|\sigma_t(x,\mu)\big|^2 + (p_0-1)  \big|\sigma_t^0(x,\mu)\big|^2 \leq L\big\{\big(1+|x|\big)^2 +  \mathcal{W}_2^{2}\big(\mu,\delta_0\big) \big\},
  \end{align*}
   for all $t\in[0,T]$, $x \in\mathbb{R}^d$ and $\mu\in\mathcal{P}_2(\mathbb{R}^d)$. 
   \end{assumption}
 \begin{assumption} \label{as:monotonicity}
 There exists a constant $L>0$ such that 
\begin{align*}
2 \big( x-\bar{x}\big) \big(b_t(x,\mu)- b_t(\bar{x},\bar{\mu})\big) & +  \big|\sigma_t(x,\mu)- \sigma_t(\bar{x},\bar{\mu})\big|^2 +  \big|\sigma_t^0(x,\mu)- \sigma_t^0(\bar{x},\bar{\mu})\big|^2  
\\
& \leq L\big\{|x-\bar{x}|^2 +\mathcal{W}_2^2(\mu,\bar{\mu})\big\},
 \end{align*} 
 for all $t\in[0,T]$, $x,\bar{x} \in\mathbb{R}^d$ and $\mu, \bar{\mu}\in\mathcal{P}_2(\mathbb{R}^d)$. 
 \end{assumption}
  \begin{assumption} \label{as:continuity}
  For every $t\in[0,T]$ and $\mu\in\mathcal{P}_2\big(\mathbb{R}^d\big)$, $b_t(x, \mu\big)$ is a continuous function of $x\in\mathbb{R}^d$. 
    \end{assumption}
The main result of this section is given in the following theorem.
\begin{thm}[\textbf{Existence, Uniqueness and Moment Bound}] \label{thm:eu}
Let Assumptions \ref{as:x0}, \ref{as:coercivity}, \ref{as:monotonicity} and \ref{as:continuity} be satisfied. Then, there exists a unique strong solution of \eqref{eq:MVSDE} and the following holds, 
\begin{align*} 
\sup_{0\leq t\leq T} E|X_t|^{p_0} \leq K,  
\end{align*}
where $K:=K(L, E|X_0|^{p_0}, d, m, m_0)>0$ is a constant. Moreover, 
\begin{align*} 
 E\sup_{0\leq t\leq T}|X_t|^{q} \leq K,  
\end{align*}
for any $q<p_0$. 
\end{thm}
\begin{proof}
Consider the space 
$$
\chi:=\Big\{\nu:\Omega^0 \to P_2(\mathbb{R}^d) \,\,\,r.v. : E^0\int_{\mathbb{R}^d} |x|^2 \nu(\mathrm{d}x)<\infty \Big\}.
$$
It follows from Lemma~\ref{lem:DL} that $(\chi,d)$ is a complete metric space, where the metric $d$ is given by 
\[
d(\nu_1,\nu_2):=\big(E^0 \mathcal{W}_2^2(\nu_1,\nu_2)\big)^{1/2},
\] 
and hence the space $C\big([0,T]; \chi\big)$ is also a complete metric space.
 Define an operator,  
 $$
 \Phi:C\big([0,T]; \chi\big) \to C\big([0,T]; \chi\big),
 $$ 
 by 
$$
 \Phi(\mu)=\big(\mathcal{L}^1(Y_t^\mu)\big)_{0\leq t\leq T},
$$
where $(Y_t^\mu)_{0\leq t\leq T}$ is the unique solution of the following stochastic differential equation with random coefficients, 
\begin{align*}
Y_t^\mu=X_0+\int_0^t b_s(Y_s^\mu, \mu_s) ds +\int_0^t \sigma_s(Y_s^\mu, \mu_s) dW_s+\int_0^t \sigma_s^0(Y_s^\mu, \mu_s) dW_s^0,
\end{align*}
almost surely for any $t\in[0,T]$. 
From \cite{gyongy1980}, the above SDE has a unique continuous solution in a strong sense. Further, it is known that under the assumptions of this article (see, e.g., \cite{kumar2017b}), 
\begin{align}
\sup_{t\in[0,T]} E |Y_t^\mu|^{p_0} \leq K \mbox{ and } E\sup_{t\in[0,T]} |Y_t^\mu|^{p} \leq K, \label{eq:mb:prop}
\end{align}
for any $p<p_0$ where $K:=K(L, E|X_0|^{p_0},  d, m, m_0)>0$. 

Notice that the map $\Phi$ is well-defined. Indeed,  for any $\mu\in C\big([0,T]; \chi\big)$, 
\begin{align*}
 |\Phi(\mu)|_{C([0,T]; \chi)} &  = \sup_{0\leq t \leq T}  E^0 \int_{\mathbb{R}^d} |x|^2 \mathcal{L}^1(Y^\mu_t)(dx)= \sup_{0\leq t\leq T} E^0  E^1\big(|Y_t^\mu|^2\big|\mathcal{F}^{W^0}\big)=\sup_{0\leq t\leq T}  E|Y_t^\mu|^2<\infty,
\end{align*}
which implies that $\Phi(\mu) \in C\big([0,T]; \chi \big)$. 
In view of Lemma 2.5 in \cite{carmona2018b} and equation \eqref{eq:mb:prop}, the flow $\big(\mathcal{L}^1(Y^\mu_t)\big)_{t\geq 0}$ has $P_0$-almost surely continuous paths in $\mathcal{P}_2\big(\mathbb{R}^d\big)$ and is $\mathbb{F}^0$-adapted.

By It\^o's formula, 
\begin{align*}
|Y^\mu_t & -Y^\nu_t|^{2} = 2 \int_0^t  \big( Y^\mu_s-Y^\nu_s\big)\big( b_s(Y_s^\mu, \mu_s)- b_s(Y_s^\nu, \nu_s) \big)  ds 
\\
&+2 \int_0^t  \big( Y^\mu_s-Y^\nu_s) \big(\sigma_s(Y_s^\mu, \mu_s)- \sigma_s(Y_s^\nu, \nu_s)\big)dW_s   
\\
&+2 \int_0^t  \big( Y^\mu_s-Y^\nu_s\big) \big(\sigma_s^0(Y_s^\mu, \mu_s)- \sigma_s^0(Y_s^\nu, \nu_s)\big) dW_s^0   
\\
&+\int_0^t \big| \sigma_s(Y_s^\mu, \mu_s)- \sigma_s(Y_s^\nu, \nu_s)\big|^2 ds  +\int_0^t  \big| \sigma_s^0(Y_s^\mu, \mu_s)- \sigma_s^0(Y_s^\nu, \nu_s)\big|^2 ds,  
\end{align*}
which on using Assumption \ref{as:monotonicity} gives, 
\begin{align*}
E\big|Y^\mu_t  -Y^\nu_t|^{2}=  & E\int_0^t \big\{2 \big( Y^\mu_s-Y^\nu_s\big)\big( b_s(Y_s^\mu, \mu_s)- b_s(Y_s^\nu, \nu_s) \big)  
\\
& +  \big| \sigma_s(Y_s^\mu, \mu_s)- \sigma_s(Y_s^\nu, \nu_s)\big|^2  +  \big| \sigma_s^0(Y_s^\mu, \mu_s)- \sigma_s^0(Y_s^\nu, \nu_s)\big|^2 \big\} ds
\\
 \leq & K \int_0^t  E|Y^\mu_s-Y^\nu_s|^{2} ds + K E\int_0^t  \mathcal{W}_2^{2}(\mu_s,\nu_s) ds
\end{align*}
for any $t\in[0,T]$. The application of Gronwall's inequality yields, 
\begin{align*}
\sup_{0\leq t \leq T}E|Y^\mu_t & -Y^\nu_t|^{2} \leq K \int_0^T \sup_{0\leq r \leq s} E^0\mathcal{W}_2^{2}(\mu_r,\nu_r) ds
\end{align*}
for any $t\in[0,T]$ which further implies, 
\begin{align*}
\big| \Phi (\mu) &-\Phi (\nu) \big|^2_{C\big([0,T]; \chi \big) }    \leq  \sup_{0\leq t \leq T} E |Y_t^\mu-Y_t^\nu|^2  
\\
& \leq  K  \int_0^T  \sup_{0\leq r \leq s} E^0 \mathcal{W}_2^{2}(\mu_r,\nu_r) ds= K  \int_0^T \big|\mu-\nu \big|_{C\big([0,s]; \chi \big)}^2 ds.  
\end{align*}
Using the above inequality, one obtains, 
\begin{align*}
\big| \Phi^2 (\mu) &-\Phi^2 (\nu) \big|^2_{C\big([0,T]; \chi\big) } \leq  K  \int_0^T \big|\Phi(\mu)-\Phi(\nu) \big|_{C\big([0,t_1]; \chi \big)}^2 dt_1
\\
& \leq  K^2  \int_0^T  \int_0^{t_1}  \big|\mu-\nu \big|_{C\big([0,t_2]; \chi \big)}^2 dt_2 dt_1
\end{align*}
and iterating further yields,
\begin{align*}
\big| \Phi^j (\mu) - & \Phi^j (\nu) \big|^2_{C\big([0,T]; \chi \big) }  \leq  K^j  \int_0^T  \int_0^{t_1} \cdots   \int_0^{t_{j-1}}  \big|\mu-\nu \big|_{C\big([0,t_j]; \chi \big)}^2 dt_j\cdots dt_1
\\
& \leq  K^j  \int_0^T  \frac{(T-t_j)^{j-1}}{(j-1)! } \big|\mu-\nu \big|_{C\big([0,t_j]; \chi \big)}^2 dt_j
\\
& \leq \frac{(KT)^j}{j!} \big|\mu-\nu \big|_{C\big([0,T]; \chi \big)}^2. 
\end{align*}
Since $\sum_{j=1}^\infty (KT)^j/j!=e^{KT}<\infty$, the mapping $\Phi$ has a unique fixed point which is the solution of \eqref{eq:MVSDE}.  
\end{proof}

Let us now introduce the \emph{interacting particle system} of \eqref{eq:MVSDE}, in order to study the \emph{propagation of chaos} property, see e.g., \cite{AS}. The state of the particle $i\in\{1,\ldots,N\}$ in the symmetric system of $N$ SDEs coupled in a mean-field sense is given by, 
\begin{align}
X_t^{i,N}=X_0^{i}+ \int_0^t b_s\big(X_s^{i,N}, \mu_s^{X,N}\big) ds + \int_0^t \sigma_s\big(X_s^{i,N}, \mu_s^{X,N}\big) dW_s^{i} + \int_0^t \sigma^0_s\big(X_s^{i,N}, \mu_s^{X,N}\big) dW_s^{0}, \label{eq:interacting}
\end{align}
almost surely, where 
\begin{align*}
\mu_t^{X,N}(\cdot):=\frac{1}{N} \sum_{i=1}^{N} \delta_{X_t^{i,N}} (\cdot)
\end{align*}
for any $t\in[0,T]$. Also, consider the following system of conditional non-interacting particles, 
\begin{align}
X_t^{i}=X_0^{i}+ \int_0^t b_s\big(X_s^{i}, \mathcal{L}^1(X^i_s)\big) ds +  \int_0^t \sigma_s\big(X_s^{i}, \mathcal{L}^1(X^i_s)\big) dW_s^{i} +  \int_0^t \sigma_s^0\big(X_s^{i}, \mathcal{L}^1(X^i_s)\big) dW_s^{0}, \label{eq:noninteracting} 
\end{align}
almost surely for any $t\in[0,T]$ and $i\in\{1,\ldots,N\}$. Moreover, by Proposition 2.11 in \cite{carmona2018b}, 
\begin{align*}
P^0\Big[\, \mathcal{L}^1(X_t^i)=\mathcal{L}^1(X_t^1) \mbox{ for all }  t\in [0,T] \, \Big]=1. 
\end{align*}
 The following proposition gives the propagation of chaos under the assumptions of this paper. 

\begin{prop}[\textbf{Propagation of Chaos}] \label{prop:chaos}
Let Assumptions  \ref{as:x0}, \ref{as:coercivity}, \ref{as:monotonicity} and \ref{as:continuity} be satisfied with $p_0>4$. Then, 
 \[
 \sup_{i \in \{1,\ldots,N\}} \sup_{t\in[0,T]} E \big|X_t^i-X_t^{i,N}\big|^2 \leq K
\begin{cases}
 N^{-1/2}, & \mbox{ if } d<4,
\\
  N^{-1/2} \ln(N), & \mbox{ if } d=4,
\\
  N^{-2/d},  & \mbox{ if } d>4,
\end{cases}
\]
where the constant $K>0$ does not depend on $N$. 
\end{prop}
\begin{proof}
From equations \eqref{eq:interacting}, \eqref{eq:noninteracting} and It\^o's formula, 
\begin{align*}
\big|X_t^i-X_t^{i,N}\big|^2=  &2\int_0^t \big(X_s^i-X_s^{i,N}\big)\big(b_s\big(X^i_s, \mathcal{L}^1({X}^i_s)\big)-b_s\big(X_s^{i,N}, \mu_s^{X,N}\big)  \big) ds
\\
&+2 \int_0^t \big(X^i_s-X_s^{i,N}\big)\big(\sigma_s\big(X^i_s, \mathcal{L}^1(X^i_s)\big)-\sigma_s\big(X_s^{i,N}, \mu_s^{X,N}\big)  \big) dW_s
\\
&+ 2 \int_0^t \big(X^i_s-X_s^{i,N}\big)\big(\sigma_s^0\big(X^i_s, \mathcal{L}^1({X}^i_s)\big)-\sigma_s^0\big(X_s^{i,N}, \mu_s^{X,N}\big)  \big) dW_s^0
\\
&+ \int_0^t \big|\sigma_s\big(X^i_s, \mathcal{L}^1(X^i_s)\big)-\sigma_s\big(X_s^{i,N}, \mu_s^{X,N}\big)\big|^2   ds
\\
&+ \int_0^t \big|\sigma_s^0\big(X^i_s, \mathcal{L}^1(X^i_s)\big)-\sigma_s^0\big(X_s^{i,N}, \mu_s^{X,N}\big)\big|^2   ds
\end{align*}
which due to Assumption \ref{as:monotonicity} yields, 
\begin{align}
E\big|X_t^i  -X_t^{i,N}\big|^2= &E\int_0^t \Big\{2\big(X^i_s-X_s^{i,N}\big)\big(b_s\big({X}^i_s, \mathcal{L}^1({X}^i_s)\big)-b_s\big(X_s^{i,N}, \mu_s^{X,N}\big) \notag 
\\
&+ \big|\sigma_s\big(X^i_s, \mathcal{L}^1(X^i_s)\big)-\sigma_s\big(X_s^{i,N}, \mu_s^{X,N}\big)\big|^2  +  \big|\sigma_s^0\big(X^i_s, \mathcal{L}^1(X^i_s)\big)-\sigma_s^0\big(X_s^{i,N}, \mu_s^{X,N}\big)\big|^2   \Big\} ds \notag
\\
\leq &  K E\int_0^t  \mathcal{W}_2^2\big(\mathcal{L}^1(X^i_s), \mu_s^{X,N}\big) ds + K \int_0^t  E\big|X^i_s-X_s^{i,N}\big|^2ds  \label{eq:prop:gron}
\end{align}
for any $t\in[0,T]$. Notice that
\begin{align*}
\mathcal{W}_2^2 \big(\mathcal{L}^1(X^i_s), \mu_s^{X,N} \big) \leq \frac{2}{N} \sum_{i=1}^N \big| X^i_s-X_s^{i,N} \big|^2 + 2 \mathcal{W}_2^2 \Big(\frac{1}{N} \sum_{i=1}^N \delta_{X^i_s}, \mathcal{L}^1\big( X_s^1\big) \Big),
\end{align*}
and then applying Theorem 5.8 and Remark 5.9 of \cite{carmona2018a}, one obtains, 
\begin{align} \label{eq:rate:w2}
E^1 \Big[ \mathcal{W}_2^2 \Big(\frac{1}{N} \sum_{i=1}^N \delta_{X^i_s}, \mathcal{L}^1\big( X_s^1\big) \Big) \Big]\leq K \big[E^1|X_t^1|^{p_0}\big]^{2/{p_0}}
\begin{cases}
 N^{-1/2}, & \mbox{ if } d <4, 
\\
 N^{-1/2} \ln (N), &   \mbox{ if } d =4, 
\\
 N^{-2/d}, &  \mbox{ if } d >4,  
\end{cases}
\end{align}
$P_0$--almost surely, where  $K:=K(d,p_0)>0$.  Taking expectation with respect to $P_0$ in equation \eqref{eq:rate:w2}, substituting this bound in equation \eqref{eq:prop:gron} and applying Gronwall's lemma completes the proof. 
\end{proof}
\section{Tamed Euler Scheme} 
In this section, we construct a tamed Euler scheme for the interacting particle system \eqref{eq:interacting} associated with \eqref{eq:MVSDE} when coefficients $b$, $\sigma$ and $\sigma^0$ are allowed to grow super-linearly. 
This is the first result on numerical approximation for interacting particles associated with equation \eqref{eq:MVSDE} with linearly and/or super-linearly growing coefficients. If $\sigma^0=0$, i.e., in the absence of common noise, a tamed Euler scheme has been studied in \cite{reis2019a}, for equations where only the drift coefficient is allowed to grow super-linearly. However, our results allow even the diffusion coefficient to grow super-linearly.
The proposed tamed Euler scheme is given below in equation \eqref{eq:euler}. We investigate the moment bounds and the rate of convergence of our tamed Euler scheme in Lemma \ref{lem:mb:euler} and Theorem \ref{thm:rate:euler}, respectively. 
Indeed, the rate of convergence of the scheme is shown to be equal to $1/2$, which is consistent with the classical Euler scheme for SDEs. 
For this purpose, we replace Assumption \ref{as:monotonicity} by a slightly stronger Assumption \ref{as:monotonicity:rate} and add more assumptions on the regularity of the coefficients as stated below. 
\begin{assumption} \label{as:monotonicity:rate}
 For some $p_1>2$, there exists a constant $L>0$ such that 
\begin{align*}
2 \big( x-\bar{x}\big) \big(b_t(x,\mu)- b_t(\bar{x},\bar{\mu})\big) & + (p_1-1) \big|\sigma_t(x,\mu)- \sigma_t(\bar{x},\bar{\mu})\big|^2 + (p_1-1)  \big|\sigma_t^0(x,\mu)- \sigma_t^0(\bar{x},\bar{\mu})\big|^2  
\\
& \leq L\big\{|x-\bar{x}|^2 +\mathcal{W}_2^2(\mu,\bar{\mu})\big\},
 \end{align*} 
 for all $t\in[0,T]$, $x,\bar{x} \in\mathbb{R}^d$ and $\mu, \bar{\mu}\in\mathcal{P}_2(\mathbb{R}^d)$. 
 \end{assumption}
 \begin{assumption} \label{as:polynomial:Lipschitz}
 There exist   constants $L>0$  and $\rho>0$ such that 
\begin{align*}
 |b_t(x,\mu)- b_t(\bar{x},\bar{\mu})| & \leq L \big\{ (1+|x|+|\bar{x}|)^{\rho/2}|x-\bar{x}|+  \mathcal{W}_2(\mu,\bar{\mu}) \big\},
\end{align*}
for all $t\in[0,T]$, $x,\bar{x} \in\mathbb{R}^d$ and $\mu, \bar{\mu}\in\mathcal{P}_2(\mathbb{R}^d)$.  
\end{assumption}
 \begin{assumption} \label{as:holder}
 There exists a constant $L>0$ such that
 \begin{align*}
 \big|b_t(x,\mu)-b_s(x,\mu)\big| + \big|\sigma_t(x,\mu)-\sigma_s(x,\mu)\big| + \big|\sigma_t^0(x,\mu)-\sigma_s^0(x,\mu)\big| \leq L |t-s|^{1/2},
 \end{align*}
 for all $t, s\in[0,T]$, $x\in\mathbb{R}^d$ and $\mu\in\mathcal{P}_2(\mathbb{R}^d)$. 
 \end{assumption}
 \begin{rem} \label{rem:poly:lipschitz}
 Due to Assumption \ref{as:monotonicity:rate} and \ref{as:polynomial:Lipschitz}, there exists a constant $K:=K(L)>0$ such that  
 \begin{align*}
  |\sigma_t(x,\mu)- \sigma_t(\bar{x},\bar{\mu})| +  |\sigma_t^0(x,\mu)- \sigma_t^0(\bar{x},\bar{\mu})| & \leq K \big\{ (1+|x|+|\bar{x}|)^{\rho/4}|x-\bar{x}|+  \mathcal{W}_2(\mu,\bar{\mu}) \big\},
 \end{align*}
 for all $t\in[0,T]$, $x,\bar{x} \in\mathbb{R}^d$ and $\mu, \bar{\mu}\in\mathcal{P}_2(\mathbb{R}^d)$.  
 \end{rem}
 \begin{rem} \label{rem:poly:growth}
 Due to Assumptions \ref{as:coercivity} and \ref{as:polynomial:Lipschitz},  there exists a constant $K:=K(L, T)>0$ such that  
 \begin{align*}
  |b_t(x,\mu)| &\leq K\big\{(1+|x|)^{\rho/2+1}+ \mathcal{W}_2(\mu,\delta_0) \big\},
  \\
    |\sigma_t(x,\mu)| + |\sigma_t^0(x,\mu)| &\leq K \big\{(1+|x|)^{\rho/4+1} + \mathcal{W}_2(\mu,\delta_0) \big\},  
 \end{align*}
 for all $t\in[0,T]$, $x \in\mathbb{R}^d$ and $\mu \in\mathcal{P}_2(\mathbb{R}^d)$.  
 \end{rem}
\begin{rem}
It is easy to see that Assumption \ref{as:continuity} follows from Assumption \ref{as:polynomial:Lipschitz}. 
\end{rem}
For introducing the tamed Euler scheme for the interacting particle system \eqref{eq:interacting} associated with \eqref{eq:MVSDE}, we partition $[0, T]$  into $n$ sub-intervals of size $h:=T/n$ and define $~\kappa_n(t):=\lfloor nt\rfloor/n$ for any $t\in[0,T]$ and $n \in\mathbb{N}$. Further, for every $t \in [0,T]$, $x \in \mathbb{R}^d$ and $\mu \in \mathcal{P}_2( \mathbb{R}^d)$, we define,
\begin{align}\label{eq:Taming}
b_{t}^n \big(x, \mu):= \frac{b_{t} \big(x, \mu)}{1+n^{-1/2}|x|^{\rho/2}}, \quad \sigma_{t}^n \big(x, \mu):= \frac{\sigma_{t}\big(x, \mu)}{1+n^{-1/2}|x|^{\rho/2}}, \quad \sigma_{t}^{0, n} \big(x, \mu):= \frac{\sigma_{t}^{0} \big(x, \mu)}{1+n^{-1/2}|x|^{\rho/2}},
\end{align}
and propose the tamed Euler scheme given by
\begin{align}
X_t^{i,N, n}=X_0^{i}+ \int_0^t b_{\kappa_n(s)}^n \big(X_{\kappa_n(s)}^{i,N, n}, & \mu_{\kappa_n(s)}^{X,N,n}\big) ds  + \int_0^t \sigma_{\kappa_n(s)}^n\big(X_{\kappa_n(s)}^{i,N,n}, \mu_{\kappa_n(s)}^{X,N,n}\big) dW_s^{i}  \notag
\\
& + \int_0^t \sigma^{0,n}_{\kappa_n(s)}\big(X_{\kappa_n(s)}^{i,N,n}, \mu_{\kappa_n(s)}^{X,N,n}\big) dW_s^{0}, \label{eq:euler}
\end{align}
for each $i\in\{1,\ldots,N\}$, where 
\begin{align*}
\mu_t^{X,N, n}(\cdot):=\frac{1}{N} \sum_{i=1}^{N} \delta_{X_t^{i,N,n}} (\cdot),
\end{align*}
almost surely for any $t\in[0,T]$ and $n\in\mathbb{N}$. 
\begin{rem} \label{rem:growth:taming}
Using equation \eqref{eq:Taming} and Remark  \ref{rem:poly:growth}, one obtains,
\begin{align*}
& |b_{t}^n \big(x, \mu)| \leq  K \min \Big\{n^{1/2} \left(1+|x| + \mathcal{W}_2(\mu,\delta_0) \right), |b_{t}\big(x, \mu)| \Big\}, \nonumber \\
& |\sigma_{t}^n \big(x, \mu)| \leq  K \min \Big\{ n^{1/4} \left(1+|x| + \mathcal{W}_2(\mu,\delta_0) \right), |\sigma_{t} \big(x, \mu)| \Big\}, \nonumber \\
& |\sigma_{t}^{0,n} \big(x, \mu)| \leq  K \min \Big\{ n^{1/4} \big(1+|x| + \mathcal{W}_2(\mu,\delta_0) \big), |\sigma_{t}^{0} \big(x, \mu)|  \Big\},
\end{align*}
for all $t \in [0,T]$, $x \in \mathbb{R}^d$ and $\mu \in \mathcal{P}_2( \mathbb{R}^d)$ and for some constant $K>0$ independent of $n$.
\end{rem}
\begin{rem} \label{rem:no:taming}
When $\rho=0$, then Assumption \ref{as:polynomial:Lipschitz} and Remark \ref{rem:poly:lipschitz} leads to Lipschitz continuity of $b$ and $\sigma$ in the state variable in which case no taming is needed in equation \eqref{eq:Taming}.   
\end{rem} 

Before proving the moment bound of the tamed Euler scheme \eqref{eq:euler}, we require the following lemmas. 
\begin{lem}\label{Lemma:onestep}
Let Assumptions \ref{as:x0}, \ref{as:coercivity}, \ref{as:monotonicity:rate} and \ref{as:polynomial:Lipschitz} hold. Then, 
\begin{align*}
E \big|X_{s}^{i,N, n} - X_{\kappa_n(s)}^{i,N, n}\big|^{p_0} &\leq K n^{-p_0/4} E \Big(1 + \big|X_{\kappa_n(s)}^{i,N, n}\big| + \mathcal{W}_2\big(\mu_{\kappa_n(s)}^{X,N,n}, \delta_0\big) \Big)^{p_0},
\end{align*}
for any $i\in\{1,\ldots,N\}$, $s\in[0,T]$ and $n, N\in\mathbb{N}$. 
\end{lem}
\begin{proof}
Due to equation (\ref{eq:euler}), 
\begin{align*}
\big|X_{s}^{i,N, n} - X_{\kappa_n(s)}^{i,N, n}\big|^{p_0} &\leq  K \Big| \int_{\kappa_n(s)}^s b_{\kappa_n(u)}^n \big(X_{\kappa_n(u)}^{i,N, n},  \mu_{\kappa_n(u)}^{X,N,n}\big) du \Big|^{p_0}  +  K \Big| \int_{\kappa_n(s)}^s \sigma_{\kappa_n(u)}^n\big(X_{\kappa_n(u)}^{i,N,n}, \mu_{\kappa_n(u)}^{X,N,n}\big) dW_u^{i} \Big|^{p_0} \nonumber
\\
& \qquad +  K \Big| \int_{\kappa_n(s)}^s \sigma^{0,n}_{\kappa_n(u)}\big(X_{\kappa_n(u)}^{i,N,n}, \mu_{\kappa_n(u)}^{X,N,n}\big) dW_u^{0} \Big|^{p_0}
\end{align*}
which on applying H\"{o}lder's inequality and Burkholder--Davis--Gundy inequality yields, 
\begin{align*}
E \big|X_{s}^{i,N, n} - X_{\kappa_n(s)}^{i,N, n}\big|^{p_0}   \leq  & K n^{-p_0+1} E \int_{\kappa_n(s)}^s \big|b_{\kappa_n(u)}^n \big(X_{\kappa_n(u)}^{i,N, n},  \mu_{\kappa_n(u)}^{X,N,n}\big)\big|^{p_0} du 
 \\
&+ K n^{-p_0/2+1} E  \int_{\kappa_n(s)}^s \big| \sigma_{\kappa_n(u)}^{n}\big(X_{\kappa_n(u)}^{i,N,n}, \mu_{\kappa_n(u)}^{X,N,n}\big) \big|^{p_0}  du 
\\
 & + K n^{-p_0/2+1}  E  \int_{\kappa_n(s)}^s \big| \sigma_{\kappa_n(u)}^{0,n} \big(X_{\kappa_n(u)}^{i,N,n}, \mu_{\kappa_n(u)}^{X,N,n}\big) \big|^{p_0} du
\end{align*}
and then the result follows by using Remark \ref{rem:growth:taming}. 
\end{proof}
In the following Lemma, we prove the moment boundedness of the tamed Euler scheme \eqref{eq:euler}. 
\begin{lem}[\textbf{Moment Bound}] \label{lem:mb:euler}
Let Assumptions \ref{as:x0}, \ref{as:coercivity}, \ref{as:monotonicity:rate} and \ref{as:polynomial:Lipschitz} hold. Then, 
\begin{equation*}
\sup_{i \in \lbrace 1, \ldots, N \rbrace } \sup_{t \in [0,T]} E \left(1 + \big|X_t^{i,N, n}\big|^2 \right)^{p_0/2} \leq K, 
\end{equation*}
where $K>0$ does not depend on $n, N\in\mathbb{N}$. Moreover, 
\begin{equation*}
\sup_{i \in \lbrace 1, \ldots, N \rbrace } E  \sup_{t \in [0,T]} \left(1 + \big|X_t^{i,N, n}\big|^2 \right)^{q/2} \leq K, 
\end{equation*}
for any $q<p_0$. 
\end{lem}
\begin{proof}
 It\^{o}'s formula gives, 
\begin{align}
 \big(1 + &|X_t^{i,N, n}|^{2} \big)^{p_0/2}  =  \big( 1 + |X_0^{i,N, n}|^{2} \big)^{p_0/2}  \nonumber 
 \\
& + p_0   \int_{0}^{t} \left(1+ |X_{s}^{i,N, n}|^2 \right)^{p_0/2 -1} X_s^{i,N, n} b_{\kappa_n(s)}^n \big(X_{\kappa_n(s)}^{i,N, n}, \mu_{\kappa_n(s)}^{X,N,n}\big) ds \nonumber 
\\
&  + p_0   \int_{0}^{t} \left(1+ |X_{s}^{i,N, n}|^2 \right)^{p_0/2 -1} X_s^{i,N, n} \sigma_{\kappa_n(s)}^n \big(X_{\kappa_n(s)}^{i,N, n}, \mu_{\kappa_n(s)}^{X,N,n}\big) dW_s^i  \nonumber 
\\
& + p_0   \int_{0}^{t} \left(1+ |X_{s}^{i,N, n}|^2 \right)^{p_0/2 -1} X_s^{i,N, n} \sigma_{\kappa_n(s)}^{0,n} \big(X_{\kappa_n(s)}^{i,N, n}, \mu_{\kappa_n(s)}^{X,N,n}\big) dW_s^0  \nonumber 
\\
&  +  \frac{p_0(p_0-2)}{2}  \int_{0}^{t} \left(1+ |X_{s}^{i,N, n}|^2 \right)^{p_0/2 -2} \big|  \sigma_{\kappa_n(s)}^{n,*} \big(X_{\kappa_n(s)}^{i,N,n}, \mu_{\kappa_n(s)}^{X,N,n}\big) X_{s}^{i,N, n}  \big|^2 ds \nonumber 
\\ 
& +  \frac{p_0(p_0-2)}{2}  \int_{0}^{t} \left(1+ |X_{s}^{i,N, n}|^2 \right)^{p_0/2 -2} \big|  \sigma_{\kappa_n(s)}^{0,n,*} \big(X_{\kappa_n(s)}^{i,N,n}, \mu_{\kappa_n(s)}^{X,N,n}\big) X_{s}^{i,N, n}  \big|^2 ds \nonumber 
\\ 
& +  \frac{p_0}{2}  \int_{0}^{t} \left(1+ |X_{s}^{i,N, n}|^2 \right)^{p_0/2 -1} \big|  \sigma_{\kappa_n(s)}^{n} \big(X_{\kappa_n(s)}^{i,N,n}, \mu_{\kappa_n(s)}^{X,N,n}\big) \big|^2 ds \nonumber 
\\ 
&  +  \frac{p_0}{2}  \int_{0}^{t} \left(1+ |X_{s}^{i,N, n}|^2 \right)^{p_0/2 -1} \big|  \sigma_{\kappa_n(s)}^{0,n} \big(X_{\kappa_n(s)}^{i,N,n}, \mu_{\kappa_n(s)}^{X,N,n}\big) \big|^2 ds \notag
\end{align}
almost surely for any $t\in[0,T]$, $i\in\{1,\ldots,N\}$ and $n,N\in\mathbb{N}$.  Thus, on taking expectation and using the Cauchy-Schwarz inequality, one obtains, 
\begin{align}
E \big(1 + & |X_t^{i,N, n}|^{2} \big)^{p_0/2}  \leq   E\big( 1 + |X_0^{i,N, n}|^{2} \big)^{p_0/2} \notag
 \\
& + \frac{p_0}{2}   E \int_{0}^{t} \left(1+ |X_{s}^{i,N, n}|^2 \right)^{p_0/2 -1} \Big\{ 2 X_{\kappa_n(s)}^{i,N, n} b_{\kappa_n(s)}^n \big(X_{\kappa_n(s)}^{i,N, n}, \mu_{\kappa_n(s)}^{X,N,n}\big) \notag
\\
&  +  (p_0-1)   \big|  \sigma_{\kappa_n(s)}^{n} \big(X_{\kappa_n(s)}^{i,N,n}, \mu_{\kappa_n(s)}^{X,N,n}\big)  \big|^2  +  (p_0-1)  \big|  \sigma_{\kappa_n(s)}^{0,n} \big(X_{\kappa_n(s)}^{i,N,n}, \mu_{\kappa_n(s)}^{X,N,n}\big)   \big|^2 \Big\} ds \notag
\\  
& + p_0  E \int_{0}^{t} \left(1+ |X_{s}^{i,N, n}|^2 \right)^{p_0/2 -1} \big( X_s^{i,N, n}  - X_{\kappa_n(s)}^{i,N, n} \big) b_{\kappa_n(s)}^n \big(X_{\kappa_n(s)}^{i,N, n}, \mu_{\kappa_n(s)}^{X,N,n}\big) ds \label{eq:ito:euler}
\end{align}
 for any $t\in[0,T]$, $i\in\{1,\ldots,N\}$ and $n,N\in\mathbb{N}$. From equation \eqref{eq:Taming}, notice that the denominators of $b^n$, $\sigma^n$ and $\sigma^{0,n}$ are the same. Hence, on  using equation \eqref{eq:euler}, Assumption \ref{as:coercivity} and  Young's inequality, one obtains,   
\begin{align}
E \big(1 + &|X_t^{i,N, n}|^{2} \big)^{p_0/2}  \leq E \big( 1 + |X_0^{i,N, n}|^{2} \big)^{p_0/2} \nonumber
\\
& + K E \int_{0}^{t} \left(1+ |X_{s}^{i,N, n}|^2 \right)^{p_0/2 -1} \Big\{ \big(1+\big|X_{\kappa_n(s)}^{i,N, n}\big| \big)^2 + \mathcal{W}_2^{2}\big( \mu_{\kappa_n(s)}^{X,N,n}, \delta_0\big)\Big\} ds \nonumber
 \\
&  + p_0 E \int_{0}^{t} \left(1+ |X_{s}^{i,N, n}|^2 \right)^{p_0/2 -1} \int_{\kappa_n(s)}^s b_{\kappa_n(r)}^n \big(X_{\kappa_n(r)}^{i,N, n}, \mu_{\kappa_n(r)}^{X,N,n}\big) dr  b_{\kappa_n(s)}^n \big(X_{\kappa_n(s)}^{i,N, n}, \mu_{\kappa_n(s)}^{X,N,n}\big) ds \nonumber
 \\
&  + p_0 E \int_{0}^{t} \left(1+ |X_{s}^{i,N, n}|^2 \right)^{p_0/2 -1} \int_{\kappa_n(s)}^s \sigma_{\kappa_n(r)}^n\big(X_{\kappa_n(r)}^{i,N,n}, \mu_{\kappa_n(r)}^{X,N,n}\big) dW_r^{i} b_{\kappa_n(s)}^n \big(X_{\kappa_n(s)}^{i,N, n}, \mu_{\kappa_n(s)}^{X,N,n}\big) ds \nonumber
 \\
&  + p_0 E \int_{0}^{t} \left(1+ |X_{s}^{i,N, n}|^2 \right)^{p_0/2 -1} \int_{\kappa_n(s)}^s \sigma^{0,n}_{\kappa_n(r)}\big(X_{\kappa_n(r)}^{i,N,n}, \mu_{\kappa_n(r)}^{X,N,n}\big) dW_r^{0}  b_{\kappa_n(s)}^n \big(X_{\kappa_n(s)}^{i,N, n}, \mu_{\kappa_n(s)}^{X,N,n}\big) ds \nonumber
\end{align}
which on using Remark \ref{rem:growth:taming} and Young's inequality, 
\begin{align}
E \big(1 + &|X_t^{i,N, n}|^{2} \big)^{p_0/2}  \leq E \big( 1 + |X_0^{i,N, n}|^{2} \big)^{p_0/2} + K  \int_{0}^{t} \sup_{0\leq r\leq s}E\big(1+ |X_{r}^{i,N, n}|^2 \big)^{p_0/2}  ds \nonumber
\\
&   + K  E\int_{0}^{t} \mathcal{W}_2^{p_0}\big( \mu_{\kappa_n(s)}^{X,N,n}, \delta_0\big) ds   + F_1 + F_2  \label{eq:F1+F2}
\end{align}
 for any $t\in[0,T]$, $i\in\{1,\ldots,N\}$ and $n,N\in\mathbb{N}$ where $F_1$ and $F_2$ are defined below. Notice that,  
\begin{align*}
F_1 & := p_0 E \int_{0}^{t} \left(1+ |X_{s}^{i,N, n}|^2 \right)^{p_0/2 -1} \int_{\kappa_n(s)}^s \sigma_{\kappa_n(r)}^n\big(X_{\kappa_n(r)}^{i,N,n}, \mu_{\kappa_n(r)}^{X,N,n}\big) dW_r^{i} b_{\kappa_n(s)}^n \big(X_{\kappa_n(s)}^{i,N, n}, \mu_{\kappa_n(s)}^{X,N,n}\big) ds \nonumber
\\
& \leq p_0 E \int_{0}^{t} \big(1+ |X_{\kappa_n(s)}^{i,N, n}|^2 \big)^{p_0/2 -1} \int_{\kappa_n(s)}^s \sigma_{\kappa_n(r)}^n\big(X_{\kappa_n(r)}^{i,N,n}, \mu_{\kappa_n(r)}^{X,N,n}\big) dW_r^{i} b_{\kappa_n(s)}^n \big(X_{\kappa_n(s)}^{i,N, n}, \mu_{\kappa_n(s)}^{X,N,n}\big) ds \nonumber
\\
&\qquad + p_0 E \int_{0}^{t} \Big|\left(1+ |X_{s}^{i,N, n}|^2 \right)^{p_0/2 -1} - \big(1+ |X_{\kappa_n(s)}^{i,N, n}|^2 \big)^{p_0/2 -1}  \Big|
\\
&  \qquad \qquad \times \Big| \int_{\kappa_n(s)}^s \sigma_{\kappa_n(r)}^n\big(X_{\kappa_n(r)}^{i,N,n}, \mu_{\kappa_n(r)}^{X,N,n}\big) dW_r^{i} b_{\kappa_n(s)}^n \big(X_{\kappa_n(s)}^{i,N, n}, \mu_{\kappa_n(s)}^{X,N,n}\big) ds \Big|\nonumber
\\
& \leq K E \int_{0}^{t} \left(1+ |X_{s}^{i,N, n}|^2 +  |X_{\kappa_n(s)}^{i,N, n}|^2 \right)^{(p_0-3)/2} \big|X_{s}^{i,N, n}- X_{\kappa_n(s)}^{i,N, n}\big|
\\
&  \qquad \times\Big|  \int_{\kappa_n(s)}^s b_{\kappa_n(s)}^n \big(X_{\kappa_n(s)}^{i,N, n}, \mu_{\kappa_n(s)}^{X,N,n}\big) \sigma_{\kappa_n(r)}^n\big(X_{\kappa_n(r)}^{i,N,n}, \mu_{\kappa_n(r)}^{X,N,n}\big) dW_r^{i} \Big| ds \nonumber
\end{align*}
for any $t\in[0,T]$, $i\in\{1,\ldots,N\}$ and $n,N\in\mathbb{N}$ where the last inequality follows due to the following expansion, for some  $\theta\in(0,1)$, 
\begin{align}
\big| (1+|x|^2)^{p_0/2-1}-(1+|y|^2)^{p_0/2-1} \big| & = \big| (p_0-2) (1+|\theta x+ (1-\theta) y|^2)^{p_0/2-2}(\theta x+ (1-\theta) y)(x-y) \big| \notag
\\
& \leq (p_0-2) \big\{1+|x|^2+|y|^2\big\}^{(p_0-3)/2}|x-y| \label{eq:mvt:mb}
\end{align}
for any $x,y\in\mathbb{R}^d$. Therefore, the application of Young's inequality, Burkholder--Davis--Gundy inequality and Lemma \ref{Lemma:onestep}  yields, 
\begin{align}
F_1 \leq &  K \int_{0}^{t}  \sup_{0\leq r\leq s}E \left(1+ |X_{r}^{i,N, n}|^2 \right)^{p_0/2} ds + K E \int_{0}^{t}n^{p_0/12} \big|X_{s}^{i,N, n}- X_{\kappa_n(s)}^{i,N, n}\big|^{p_0/3} \notag
\\
&   \times n^{-p_0/12}\Big|  \int_{\kappa_n(s)}^s b_{\kappa_n(s)}^n \big(X_{\kappa_n(s)}^{i,N, n}, \mu_{\kappa_n(s)}^{X,N,n}\big) \sigma_{\kappa_n(r)}^n\big(X_{\kappa_n(r)}^{i,N,n}, \mu_{\kappa_n(r)}^{X,N,n}\big) dW_r^{i} \Big|^{p_0/3}  ds \notag
\\
\leq & K \int_{0}^{t}  \sup_{0\leq r\leq s}E \left(1+ |X_{r}^{i,N, n}|^2 \right)^{p_0/2} ds + K  \int_{0}^{t}n^{p_0/4} E\big|X_{s}^{i,N, n}- X_{\kappa_n(s)}^{i,N, n}\big|^{p_0} ds \notag
\\
&   + K n^{-p_0/8}  \int_{0}^{t}  E\Big|  \int_{\kappa_n(s)}^s b_{\kappa_n(s)}^n \big(X_{\kappa_n(s)}^{i,N, n}, \mu_{\kappa_n(s)}^{X,N,n}\big) \sigma_{\kappa_n(r)}^n\big(X_{\kappa_n(r)}^{i,N,n}, \mu_{\kappa_n(r)}^{X,N,n}\big) dW_r^{i} \Big|^{p_0/2}  ds \notag
\\
 \leq  & K+ K  \int_{0}^{t}  \sup_{0\leq r\leq s}E \left(1+ |X_{r}^{i,N, n}|^2 \right)^{p_0/2} ds \notag 
\\
& + K n^{-p_0/8} n^{-p_0/4+1} \int_{0}^{t}  E  \int_{\kappa_n(s)}^s \big|b_{\kappa_n(s)}^n \big(X_{\kappa_n(s)}^{i,N, n}, \mu_{\kappa_n(s)}^{X,N,n}\big)\big|^{p_0/2} \big|\sigma_{\kappa_n(r)}^n\big(X_{\kappa_n(r)}^{i,N,n}, \mu_{\kappa_n(r)}^{X,N,n}\big) \big|^{p_0/2}dr  ds \label{eq:F1:ms}
\end{align}
and then one uses Remark \ref{rem:growth:taming}  to obtain the following, 
\begin{align}
F_1  \leq   K \int_{0}^{t}  \sup_{0\leq r\leq s}E \left(1+ |X_{r}^{i,N, n}|^2 \right)^{p_0/2} ds+  KE \int_{0}^{t}   \mathcal{W}_2^{p_0}\big( \mu_{\kappa_n(s)}^{X,N,n}, \delta_0\big) ds \label{eq:F1}
\end{align}
for any $t\in[0,T]$, $i\in\{1,\ldots,N\}$ and $n,N\in\mathbb{N}$. 
By adapting the similar arguments as used in the estimation of $F_1$, one can obtain, 
\begin{align}
F_2:= & p_0 E \int_{0}^{t} \left(1+ |X_{s}^{i,N, n}|^2 \right)^{p_0/2 -1} \int_{\kappa_n(s)}^s \sigma^{0,n}_{\kappa_n(r)}\big(X_{\kappa_n(r)}^{i,N,n}, \mu_{\kappa_n(r)}^{X,N,n}\big) dW_r^{0}  b_{\kappa_n(s)}^n \big(X_{\kappa_n(s)}^{i,N, n}, \mu_{\kappa_n(s)}^{X,N,n}\big) ds \notag
\\
 \leq &  K \int_{0}^{t}  \sup_{0\leq r\leq s}E \left(1+ |X_{r}^{i,N, n}|^2 \right)^{p_0/2} ds+  KE \int_{0}^{t}   \mathcal{W}_2^{p_0 }\big( \mu_{\kappa_n(s)}^{X,N,n}, \delta_0\big) ds, \label{eq:F2}
\end{align}
for any $t\in[0,T]$, $i\in\{1,\ldots,N\}$ and $n,N\in\mathbb{N}$.  

On substituting estimates of equations \eqref{eq:F1} and \eqref{eq:F2} in equation \eqref{eq:F1+F2}, one gets, 
 \begin{align*}
E \big(1 + |X_t^{i,N, n}|^{2} \big)^{p_0/2}  &\leq E \big( 1 + |X_0^{i,N, n}|^{2} \big)^{p_0/2} + K  \int_{0}^{t} \sup_{0\leq r\leq s}E\big(1+ |X_{r}^{i,N, n}|^2 \big)^{p_0/2}  ds 
\\
& \quad  + K  E\int_{0}^{t} \mathcal{W}_2^{p_0}\big( \mu_{\kappa_n(s)}^{X,N,n}, \delta_0\big) ds, 
\end{align*}
 for any $t\in[0,T]$, $i\in\{1,\ldots,N\}$ and $n,N\in\mathbb{N}$. By a simple calculation (see, e.g., Lemma 2.3 in \cite{reis2019b}), one can observe that, 
 \begin{align}
 \mathcal{W}_2^{2}\big( \mu_{\kappa_n(s)}^{X,N,n}, \delta_0\big)=\frac{1}{N} \sum_{j=1}^N \big|X_{\kappa_n(s)}^{i,N,n}\big|^2, \label{eq:w2:mb}
 \end{align}
which on using yields,
 \begin{align*}
\sup_{ i \in\{1,\ldots N\}}\sup_{0\leq r \leq t}  E \big(1 + &|X_r^{i,N, n}|^{2} \big)^{p_0/2}  \leq E \big( 1 + |X_0^{i,N, n}|^{2} \big)^{p_0/2} + K  \int_{0}^{t} \sup_{ i \in\{1,\ldots N\}}\sup_{0\leq r \leq s} E\left(1+ |X_{r}^{i,N, n}|^2 \right)^{p_0/2}  ds
 \end{align*}
 and  the application of Gronwall's inequality completes the proof of the first inequality. The second inequality follows due to Lemma \ref{lem:gk}.  
\end{proof}
Before proceeding with the proof of rate of convergence of the tamed Euler scheme \eqref{eq:euler}, we establish some lemmas as given below. 
\begin{lem} \label{lem:one-step:euler}
Let Assumptions \ref{as:x0}, \ref{as:coercivity}, \ref{as:monotonicity:rate} and \ref{as:polynomial:Lipschitz} hold. Then, 
\begin{align*}
E\big| X_s^{i,N,n}-  X_{\kappa_n(s)}^{i,N,n}\big|^{p} \leq K n^{-{p}/2},
\end{align*}
for any $i\in\{1,\ldots,N\}$, $s\in[0,T]$ and $n, N\in\mathbb{N}$ where the constant $K>0$ does not depend on $n$ or $N$. 
\end{lem}
\begin{proof}
The proof follows by replicating the proof of Lemma \ref{Lemma:onestep} and then using Remark \ref{rem:poly:growth} and Lemma~ \ref{lem:mb:euler}. 
\end{proof}
\begin{lem} \label{lem:euler:rate:coeff}
Let Assumptions \ref{as:x0}, \ref{as:coercivity}, \ref{as:monotonicity:rate}, \ref{as:polynomial:Lipschitz} and \ref{as:holder} hold. Then, 
\begin{align*}
E  \big|b_s\big(X_s^{i,N, n}, \mu_s^{X,N, n}\big) -b_{\kappa_n(s)}^n\big(X_{\kappa_n(s)}^{i,N,n}, \mu_{\kappa_n(s)}^{X,N,n}\big) \big|^p \leq K n^{-p/2},
\\
E  \big|\sigma_s\big(X_s^{i,N, n}, \mu_s^{X,N, n}\big) -\sigma_{\kappa_n(s)}^n\big(X_{\kappa_n(s)}^{i,N,n}, \mu_{\kappa_n(s)}^{X,N,n}\big) \big|^p \leq K n^{-p/2},
\\
E  \big|\sigma_s^0\big(X_s^{i,N, n}, \mu_s^{X,N, n}\big) -\sigma_{\kappa_n(s)}^{n,0}\big(X_{\kappa_n(s)}^{i,N,n}, \mu_{\kappa_n(s)}^{X,N,n}\big) \big|^p \leq K n^{-p/2},
\end{align*}
for any $s\in[0,T]$, $i\in\{1,\ldots, N\}$ and $n,N\in\mathbb{N}$ where the constant $K>0$ does not depend on $n$ and $N$. 
\end{lem}
\begin{proof}
On using Remark \ref{rem:poly:lipschitz} and Assumption \ref{as:holder}, one obtains, 
\begin{align*}
E  \big|b_s\big(X_s^{i,N, n},& \mu_s^{X,N, n}\big) -b_{\kappa_n(s)}^n\big(X_{\kappa_n(s)}^{i,N,n}, \mu_{\kappa_n(s)}^{X,N,n}\big) \big|^p \leq K E  \big|b_s\big(X_s^{i,N, n}, \mu_s^{X,N, n}\big) -b_{s}\big(X_{\kappa_n(s)}^{i,N,n}, \mu_{\kappa_n(s)}^{X,N,n}\big) \big|^p
\\
& + K E  \big|b_{s}\big(X_{\kappa_n(s)}^{i,N,n}, \mu_{\kappa_n(s)}^{X,N,n}\big)- b_{\kappa_n(s)}\big(X_{\kappa_n(s)}^{i,N,n}, \mu_{\kappa_n(s)}^{X,N,n}\big)\big|^p
\\
& +K E  \big| b_{\kappa_n(s)}\big(X_{\kappa_n(s)}^{i,N,n}, \mu_{\kappa_n(s)}^{X,N,n}\big) - b^n_{\kappa_n(s)}\big(X_{\kappa_n(s)}^{i,N,n}, \mu_{\kappa_n(s)}^{X,N,n}\big)\big|^p 
\\
 \leq & K E\big(1+ \big|X_s^{i,N, n}\big|+ \big|X_{\kappa_n(s)}^{i,N, n}\big|\big)^{\rho p/2}\big| X_s^{i,N, n} - X_{\kappa_n(s)}^{i,N, n}\big|^p + KE \mathcal{W}_2^p\big(\mu_s^{X,N, n}, \mu_{\kappa_n(s)}^{X,N,n} \big) 
 \\
 &+ K |s-\kappa_n(s)|^{p/2} + K n^{-p}
\end{align*}
which on the application of H\"older's inequality, Lemma \ref{lem:mb:euler}, Lemma \ref{lem:one-step:euler} and the following elementary estimate, 
\begin{align}
\mathcal{W}_2^2\big(\mu_s^{X,N, n}, \mu_{\kappa_n(s)}^{X,N,n} \big) \leq \frac{1}{N}\sum_{j=1}^n \big|X_s^{i,N,n}-X_{\kappa_n(s)}^{i,N,n}\big|^2, \label{eq:w2:scheme}
\end{align}
proves the estimate for $b$. The proof is completed by performing  similar calculations for $\sigma$ and $\sigma^0$.  
\end{proof}
The following is the main result of this section. 
\begin{thm}[\textbf{Rate of Convergence}] \label{thm:rate:euler}
Let Assumptions \ref{as:x0},  \ref{as:coercivity}, \ref{as:monotonicity:rate}, \ref{as:polynomial:Lipschitz} and \ref{as:holder} be satisfied. Then, the tamed Euler scheme \eqref{eq:euler} converges to the true solution of the interacting particle system \eqref{eq:interacting} associated with \eqref{eq:MVSDE} in a strong sense with $L^p$ rate of convergence given by, 
\begin{align*}
\sup_{i\in\{1,\ldots, N\}}E\sup_{t\in[0,T]} |X_t^{i,N}-X_t^{i,N,n}|^p \leq K n^{-p/2},
\end{align*} 
for any $p<\min\{p_1, p_0/(\rho+1)\}$, where the constant $K>0$ does not depend on $n, N\in\mathbb{N}$.  
\end{thm}
\begin{proof}
Let us first assume that $p<\min\{p_1, p_0/(\rho+1)\}$. From equation \eqref{eq:interacting} and \eqref{eq:euler}, 
\begin{align}
X_t^{i,N} & - X_t^{i,N, n} =  \int_0^t \big( b_s\big(X_s^{i,N}, \mu_s^{X,N}\big) - b_{\kappa_n(s)}^n \big(X_{\kappa_n(s)}^{i,N, n},  \mu_{\kappa_n(s)}^{X,N,n}\big) \big) ds  \notag
\\
&+ \int_0^t \big(\sigma_s\big(X_s^{i,N}, \mu_s^{X,N}\big) -\sigma_{\kappa_n(s)}^n\big(X_{\kappa_n(s)}^{i,N,n}, \mu_{\kappa_n(s)}^{X,N,n}\big) \big) dW_s^{i} \notag
\\
&+ \int_0^t \big(\sigma^0_s\big(X_s^{i,N}, \mu_s^{X,N}\big) - \sigma^{0,n}_{\kappa_n(s)}\big(X_{\kappa_n(s)}^{i,N,n}, \mu_{\kappa_n(s)}^{X,N,n}\big)  \big)dW_s^{0} \label{eq:true-euler}
\end{align}
almost surely for any $t\in[0,T]$, $i\in\{1,\ldots,N\}$ and $n, N\in\mathbb{N}$.  By It\^o's formula, 
\begin{align*}
 \big |&X_t^{i,N}  - X_t^{i,N, n} \big|^p 
\\
= & p \int_0^t \big|X_s^{i,N}  - X_s^{i,N, n} \big|^{p-2} \big(X_s^{i,N}  - X_s^{i,N, n} \big)\big(b_s\big(X_s^{i,N}, \mu_s^{X,N}\big) - b_{\kappa_n(s)}^n \big(X_{\kappa_n(s)}^{i,N, n},  \mu_{\kappa_n(s)}^{X,N,n}\big) \big) ds  
\\
&+ p \int_0^t \big|X_s^{i,N}  - X_s^{i,N, n} \big|^{p-2} \big(X_s^{i,N}  - X_s^{i,N, n} \big) \big(\sigma_s\big(X_s^{i,N}, \mu_s^{X,N}\big) -\sigma_{\kappa_n(s)}^n\big(X_{\kappa_n(s)}^{i,N,n}, \mu_{\kappa_n(s)}^{X,N,n}\big) \big) dW_s^{i} 
\\
&+p \int_0^t \big|X_s^{i,N}  - X_s^{i,N, n} \big|^{p-2} \big(X_s^{i,N}  - X_s^{i,N, n} \big) \big(\sigma^0_s\big(X_s^{i,N}, \mu_s^{X,N}\big) - \sigma^{0,n}_{\kappa_n(s)}\big(X_{\kappa_n(s)}^{i,N,n}, \mu_{\kappa_n(s)}^{X,N,n}\big)  \big)dW_s^{0}
\\
& + \frac{p(p-2)}{2} \int_0^t \big|X_s^{i,N}  - X_s^{i,N, n} \big|^{p-4}  \big|\big(\sigma_s\big(X_s^{i,N}, \mu_s^{X,N}\big) -\sigma_{\kappa_n(s)}^n\big(X_{\kappa_n(s)}^{i,N,n}, \mu_{\kappa_n(s)}^{X,N,n}\big)\big)^*  \big(X_s^{i,N}  - X_s^{i,N, n} \big)\big|^2 ds
\\
& + \frac{p(p-2)}{2} \int_0^t \big|X_s^{i,N}  - X_s^{i,N, n} \big|^{p-4}  \big|\big(\sigma_s\big(X_s^{i,N}, \mu_s^{X,N}\big) -\sigma_{\kappa_n(s)}^n\big(X_{\kappa_n(s)}^{i,N,n}, \mu_{\kappa_n(s)}^{X,N,n}\big)\big)^*  \big(X_s^{i,N}  - X_s^{i,N, n} \big)\big|^2 ds
\\
&+ \frac{p}{2} \int_0^t  \big|X_s^{i,N}  - X_s^{i,N, n} \big|^{p-2} \big|\sigma_s^0\big(X_s^{i,N}, \mu_s^{X,N}\big) -\sigma_{\kappa_n(s)}^{0,n}\big(X_{\kappa_n(s)}^{i,N,n}, \mu_{\kappa_n(s)}^{X,N,n}\big) \big|^2 ds 
\\
&+\frac{p}{2} \int_0^t \big|X_s^{i,N}  - X_s^{i,N, n} \big|^{p-2}  \big|\sigma^0_s\big(X_s^{i,N}, \mu_s^{X,N}\big) - \sigma^{0,n}_{\kappa_n(s)}\big(X_{\kappa_n(s)}^{i,N,n}, \mu_{\kappa_n(s)}^{X,N,n}\big)  \big|^2 ds 
\end{align*}
which on taking expectation along with the Cauchy-Schwarz inequality yields, 
\begin{align}
 E\big|&X_t^{i,N}  - X_t^{i,N, n} \big|^p \leq \frac{p}{2} E\int_0^t \big|X_s^{i,N}  - X_s^{i,N, n} \big|^{p-2} \Big\{ 2\big(X_s^{i,N}  - X_s^{i,N, n} \big)  \notag
\\
& \qquad\big(b_s\big(X_s^{i,N}, \mu_s^{X,N}\big)- b_s \big(X_s^{i,N, n},  \mu_s^{X,N, n}\big) + b_s \big(X_s^{i,N, n},  \mu_s^{X,N, n}\big) - b_{\kappa_n(s)}^n \big(X_{\kappa_n(s)}^{i,N, n},  \mu_{\kappa_n(s)}^{X,N,n}\big) \big)   \notag
\\
&+(p-1)  \big|\sigma_s\big(X_s^{i,N}, \mu_s^{X,N}\big)- \sigma_s\big(X_s^{i,N, n}, \mu_s^{X,N, n}\big)+ \sigma_s\big(X_s^{i,N, n}, \mu_s^{X,N, n}\big) -\sigma_{\kappa_n(s)}^n\big(X_{\kappa_n(s)}^{i,N,n}, \mu_{\kappa_n(s)}^{X,N,n}\big) \big|^2 \notag
\\
&+ (p-1)  \big|\sigma^0_s\big(X_s^{i,N}, \mu_s^{X,N}\big)-\sigma^0_s\big(X_s^{i,N, n}, \mu_s^{X,N, n}\big) + \sigma^0_s\big(X_s^{i,N, n}, \mu_s^{X,N, n}\big)  - \sigma^{0,n}_{\kappa_n(s)}\big(X_{\kappa_n(s)}^{i,N,n}, \mu_{\kappa_n(s)}^{X,N,n}\big)  \big|^2 \Big\} ds \notag 
\end{align}
and then using Young's inequality yields, 
\begin{align*}
 E\big|&X_t^{i,N}  - X_t^{i,N, n} \big|^p \leq \frac{p}{2} E\int_0^t \big|X_s^{i,N}  - X_s^{i,N, n} \big|^{p-2} \Big\{ 2 \big(X_s^{i,N}  - X_s^{i,N, n} \big)\big(b_s\big(X_s^{i,N}, \mu_s^{X,N}\big)- b_s \big(X_s^{i,N, n},  \mu_s^{X,N, n}\big) 
\\
& + (p_1-1)\big|\sigma_s\big(X_s^{i,N}, \mu_s^{X,N}\big)- \sigma_s\big(X_s^{i,N, n}, \mu_s^{X,N, n}\big)\big|^2 
\\
& + (p_1-1) \big|\sigma^0_s\big(X_s^{i,N}, \mu_s^{X,N}\big)-\sigma^0_s\big(X_s^{i,N, n}, \mu_s^{X,N, n}\big) \big|^2\Big\} ds 
\\
& + K E\int_0^t \big|X_s^{i,N}  - X_s^{i,N, n} \big|^{p-2}  \big(X_s^{i,N}  - X_s^{i,N, n} \big) \big(b_s \big(X_s^{i,N, n},  \mu_s^{X,N, n}\big) - b_{\kappa_n(s)}^n \big(X_{\kappa_n(s)}^{i,N, n},  \mu_{\kappa_n(s)}^{X,N,n}\big) \big) ds  
\\
&+K E \int_0^t \big|X_s^{i,N}  - X_s^{i,N, n} \big|^{p-2}   \big|\sigma_s\big(X_s^{i,N, n}, \mu_s^{X,N, n}\big) -\sigma_{\kappa_n(s)}^n\big(X_{\kappa_n(s)}^{i,N,n}, \mu_{\kappa_n(s)}^{X,N,n}\big) \big|^2 ds 
\\
&+ K E\int_0^t  \big|X_s^{i,N}  - X_s^{i,N, n} \big|^{p-2}  \big| \sigma^0_s\big(X_s^{i,N, n}, \mu_s^{X,N, n}\big)  - \sigma^{0,n}_{\kappa_n(s)}\big(X_{\kappa_n(s)}^{i,N,n}, \mu_{\kappa_n(s)}^{X,N,n}\big)  \big|^2 ds
\end{align*}
for any $t\in[0,T]$, $i\in\{1,\ldots,N\}$ and $n, N\in\mathbb{N}$. By using the Cauchy-Schwarz inequality, Young's inequality and Assumption \ref{as:monotonicity:rate}, one obtains, 
\begin{align}
E\big|X_t^{i,N}  - X_t^{i,N, n} \big|^p \leq & E\int_0^t \big|X_s^{i,N}  - X_s^{i,N, n} \big|^p ds + E\int_0^t \mathcal{W}_2^p\big(\mu_s^{X,N},  \mu_s^{X,N, n}\big)  ds \notag
\\
&+K E \int_0^t  \big|b_s\big(X_s^{i,N, n}, \mu_s^{X,N, n}\big) -b_{\kappa_n(s)}^n\big(X_{\kappa_n(s)}^{i,N,n}, \mu_{\kappa_n(s)}^{X,N,n}\big) \big|^p ds \notag
\\
&+K E \int_0^t  \big|\sigma_s\big(X_s^{i,N, n}, \mu_s^{X,N, n}\big) -\sigma_{\kappa_n(s)}^n\big(X_{\kappa_n(s)}^{i,N,n}, \mu_{\kappa_n(s)}^{X,N,n}\big) \big|^p ds \notag
\\
&+ K E\int_0^t  \big| \sigma^0_s\big(X_s^{i,N, n}, \mu_s^{X,N, n}\big)  - \sigma^{0,n}_{\kappa_n(s)}\big(X_{\kappa_n(s)}^{i,N,n}, \mu_{\kappa_n(s)}^{X,N,n}\big)  \big|^p ds, \label{eq:ito:est:euler}
\end{align}
which on the application of Lemma \ref{lem:euler:rate:coeff} along with the following elementary estimate 
\begin{align}
 \mathcal{W}_2^2 \big(\mu_s^{X,N},  \mu_s^{X,N, n}\big) \leq \frac{1}{N}\sum_{j=1}^N  \big|X_s^{i,N}- X_s^{i,N, n}\big|^2,  \label{eq:w2:particle}
\end{align} 
yields, 
\begin{align*}
\sup_{i\in\{1,\ldots,N\}}\sup_{0\leq r\leq t} E\big|X_r^{i,N} & - X_r^{i,N, n} \big|^p \leq \int_0^t \sup_{i\in\{1,\ldots,N\}}\sup_{0\leq r\leq s}E \big|X_r^{i,N}  - X_r^{i,N, n} \big|^p ds  + K n^{-p/2},
\end{align*}
for any $s\in[0,T]$, $i\in\{1,\ldots,N\}$ and $n, N\in\mathbb{N}$.  
The application of Gronwall's inequality gives,
\begin{align*}
\sup_{i\in\{1,\ldots,N\}}\sup_{0\leq r\leq T}  E\big|X_t^{i,N} & - X_t^{i,N, n} \big|^p \leq  K n^{-p/2},
\end{align*}
for any $p<\min\{p_1, p_0/(\rho+1)\}$, $s\in[0,T]$, $i\in\{1,\ldots,N\}$ and $n, N\in\mathbb{N}$.  
Then, the application of Lemma \ref{lem:gk} completes the proof. 
\end{proof}
\section{Tamed Milstein Scheme} 
In this section, we propose a tamed Milstein scheme for the interacting particle system \eqref{eq:interacting}  associated with  McKean--Vlasov SDE  \eqref{eq:MVSDE} when the coefficients $b$, ~$\sigma$ and ~$\sigma^0$ are allowed to grow super-linearly. 
The proposed Milstein scheme is given below in equation \eqref{eq:milstein}.
Note that we use the same notation $X^{i,N,n}$ for the tamed Milstein scheme  \eqref{eq:milstein} and the tamed Euler scheme \eqref{eq:euler} (discussed in Section 3) which should not cause any confusion in the reader's mind. 
We study the moment bounds and the rate of convergence of our scheme  in Lemma \ref{lem:mb:milstein} and Theorem \ref{thm:rate:milstein}, respectively.
 Indeed, the rate of  convergence of our tamed Milstein scheme is shown to be equal to $1$ which is consistent with the classical Milstein scheme for SDEs.  
 For this purpose, we replace Assumption \ref{as:holder} by Assumption \ref{as:lipschitz} and add more assumptions on the regularity of the coefficients as stated below. 
 \begin{assumption} \label{as:lipschitz}
 There exists a constant $L>0$ such that
 \begin{align*}
 \big|b_t(x,\mu)-b_s(x,\mu)\big| + \big|\sigma_t(x,\mu)-\sigma_s(x,\mu)\big| + \big|\sigma_t^0(x,\mu)-\sigma_s^0(x,\mu)\big| \leq L |t-s|,
 \end{align*}
 for all $t, s\in[0,T]$, $x\in\mathbb{R}^d$ and $\mu\in\mathcal{P}_2(\mathbb{R}^d)$. 
 \end{assumption}
 \begin{assumption} \label{as:der:x:poly:lip}
 There exists a constant $L>0$ such that, for every  $j\in\{1,\ldots m\}$ and $j'\in\{1,\ldots m^0\}$ 
  \begin{align*}
 \big|\partial_x  b_t(x,\mu)-\partial_x b_t(\bar{x},\bar{\mu})\big| & \leq L \big\{\big(1+|x|+|\bar{x}|\big)^{\rho/2-1}\big|x-\bar{x}\big|+ \mathcal{W}_2\big(\mu,\bar{\mu}\big) \big\},
 \\
 \big|\partial_x \sigma_t^{(j)}(x,\mu)-\partial_x \sigma_t^{(j)}(\bar{x},\bar{\mu})\big|  & \leq L \big\{\big(1+|x|+|\bar{x}|\big)^{\rho/4-1}\big|x-\bar{x}\big|+ \mathcal{W}_2\big(\mu,\bar{\mu} \big)\big\},
 \\
 \big|\partial_x \sigma_t^{0, (j')}(x,\mu)-\partial_x \sigma_t^{0, (j')}(\bar{x},\bar{\mu})\big| & \leq L \big\{\big(1+|x|+|\bar{x}|\big)^{\rho/4-1}\big|x-\bar{x}\big|+ \mathcal{W}_2\big(\mu,\bar{\mu} \big)\big\},
 \end{align*}
 for all $t\in[0,T]$, $x,\bar{x}\in\mathbb{R}^d$ and $\mu,\bar{\mu}\in\mathcal{P}_2\big(\mathbb{R}^d\big)$. 
 \end{assumption} 
 \begin{assumption} \label{as:der:mea:poly:lip}
 There exists a constant $L>0$ such that, for every $k\in\{1,\ldots,d\}$,  $j\in\{1,\ldots m\}$ and $j'\in\{1,\ldots m^0\}$ ,  
  \begin{align*}
 \big|\partial_\mu  b_t^{(k)}(x,\mu, y)-\partial_\mu b_t^{(k)}(\bar{x},\bar{\mu},\bar{y})\big|^2 & \leq L \big\{\big(1+|x|+|\bar{x}|\big)^{\rho}\big|x-\bar{x}\big|^2+ \mathcal{W}_2^2\big(\mu,\bar{\mu}\big) +|y-\bar{y}|^2 \big\},
 \\
 \big|\partial_\mu \sigma_t^{(k,j)}(x,\mu, y)-\partial_\mu \sigma_t^{(k,j)}(\bar{x},\bar{\mu}, \bar{y})\big|^2  & \leq L \big\{\big(1+|x|+|\bar{x}|\big)^{\rho/2}\big|x-\bar{x}\big|^2+ \mathcal{W}_2^2\big(\mu,\bar{\mu} \big) +\big|y-\bar{y}\big|^2\big\},
 \\
 \big|\partial_\mu \sigma_t^{0, (k,j')}(x,\mu, y)-\partial_\mu \sigma_t^{0, (k,j')}(\bar{x},\bar{\mu}, \bar{y})\big|^2 & \leq L \big\{\big(1+|x|+|\bar{x}|\big)^{\rho/2}\big|x-\bar{x}\big|^2+ \mathcal{W}_2^2\big(\mu,\bar{\mu} \big)+\big|y-\bar{y}\big|^2\big\},
 \end{align*}
 for all $t\in[0,T]$, $x,\bar{x}, y,\bar{y}\in\mathbb{R}^d$ and $\mu,\bar{\mu}\in\mathcal{P}_2\big(\mathbb{R}^d\big)$. 
 \end{assumption} 
 \begin{rem} \label{rem:der:growth:poly}
 As a consequence of Assumptions \ref{as:polynomial:Lipschitz} and Remark \ref{rem:poly:lipschitz}, there is a constant $K:=K(L)>0$ such that, for every $k\in\{1,\ldots,d\}$, $j\in\{1,\ldots m\}$ and $j'\in\{1,\ldots m^0\}$ 
 \begin{align*}
 \big|\partial_x  b_t(x,\mu)\big|  &\leq K \big(1+\big|x\big|\big)^{\rho/2}, 
 \\
  \big|\partial_x  \sigma_t^{(j)}(x,\mu)\big|  + & \big|\partial_x  \sigma_t^{0,{(j')}}(x,\mu)\big|  \leq K  \big(1+\big|x\big|\big)^{\rho/4},
 \\
 \big|\partial_\mu  b_t^{(k)}(x,\mu, y)\big| + \big|\partial_\mu & \sigma_t^{(k,j)}(x,\mu, y)\big|  +  \big|\partial_\mu  \sigma_t^{0, (k,j')}(x,\mu, y)\big|  \leq K,
 \end{align*}
 for all $t\in[0,T]$, $x, y\in\mathbb{R}^d$ and $\mu\in\mathcal{P}_2\big(\mathbb{R}^d\big)$. 
 \end{rem}
As in Section 2, for introducing the tamed  Milstein scheme for the interacting particle system \eqref{eq:interacting} associated with McKean--Vlasov SDE \eqref{eq:MVSDE}, we partition  $[0, T]$   into $n$ sub-intervals of size $h:=T/n$ and define $\kappa_n(t):=\lfloor nt\rfloor/n$ for any $t\in[0,T]$ and $n \in\mathbb{N}$. 
Further, for every $t \in [0,T]$, $x \in \mathbb{R}^d$ and $\mu \in \mathcal{P}_2( \mathbb{R}^d)$, we define,
\begin{align}\label{eq:Taming:milstein}
b_{t}^n \big(x, \mu):= \frac{b_{t} \big(x, \mu)}{1+n^{-1} |x|^{\rho}}, & \quad \sigma_{t}^n \big(x, \mu):= \frac{\sigma_{t}\big(x, \mu)}{1+n^{-1}|x|^{\rho}}, \quad \sigma_{t}^{0, n} \big(x, \mu):= \frac{\sigma_{t}^{0} \big(x, \mu)}{1+n^{-1}|x|^{\rho}},
\end{align}
and propose the  tamed  Milstein scheme given by, 
\begin{align}
X_t^{i,N, n}& =X_0^{i}+ \int_0^t b^n_{\kappa_n(s)} \big(X_{\kappa_n(s)}^{i,N, n}, \mu_{\kappa_n(s)}^{X,N,n}\big) ds + \int_0^t \tilde{\sigma}^n_{\kappa_n(s)} \big(s,X_{\kappa_n(s)}^{i,N,n}, \mu_{\kappa_n(s)}^{X,N,n}\big) dW_s^{i} \notag
\\
 & \qquad + \int_0^t \tilde{\sigma}^{0,n}_{\kappa_n(s)} \big(s,X_{\kappa_n(s)}^{i,N,n}, \mu_{\kappa_n(s)}^{X,N,n}\big) dW_s^{0}   \label{eq:milstein}
\end{align}
almost surely for any $t\in[0,T]$,  $i\in\{1,\ldots,N\}$ and $n,N\in\mathbb{N}$. The coefficients  $\tilde{\sigma}^n$ and $\tilde{\sigma}^{0,n}$ are defined below. For any $s\in[0,T]$,  $i\in\{1,\ldots, N\}$ and $n,N\in\mathbb{N}$, 
\begin{align}
\tilde{\sigma}^n_{\kappa_n(s)} \big(s,X_{\kappa_n(s)}^{i,N,n}, \mu_{\kappa_n(s)}^{X,N,n}\big)&:=\sigma_{\kappa_n(s)}^n \big(X_{\kappa_n(s)}^{i,N,n}, \mu_{\kappa_n(s)}^{X,N,n}\big) + \Gamma_{\kappa_n(s)}^{n, \sigma} \big(s, X_{\kappa_n(s)}^{i,N,n}, \mu_{\kappa_n(s)}^{X,N,n}\big)   \label{eq:tilde:sigma:n}
\end{align}
where $\Gamma^{n, \sigma}$ is further expressed as a sum of four matrices, i.e., 
\begin{align*}
\Gamma_{\kappa_n(s)}^{n, \sigma} \big(s, X_{\kappa_n(s)}^{i,N,n}, \mu_{\kappa_n(s)}^{X,N,n}\big) & := \Lambda_{\kappa_{n}(s)}^{n, \sigma \sigma} \big(s,X_{\kappa_n(s)}^{i,N,n}, \mu_{\kappa_n(s)}^{X,N,n}\big) +\Lambda_{\kappa_{n}(s)}^{n, \sigma \sigma^0} \big(s,X_{\kappa_n(s)}^{i,N,n}, \mu_{\kappa_n(s)}^{X,N,n}\big)   \notag
\\
& \qquad +\bar{\Lambda}^{n,\sigma \sigma}_{\kappa_n(s)} \big( s,X_{\kappa_n(s)}^{i,N,n}, \mu_{\kappa_n(s)}^{X,N,n}\big)  +\bar{\Lambda}^{n,\sigma \sigma^0}_{\kappa_n(s)} \big( s,X_{\kappa_n(s)}^{i,N,n}, \mu_{\kappa_n(s)}^{X,N,n}\big) \notag
\end{align*}
where   $\Lambda^{n, \sigma \sigma}$, $\Lambda^{n, \sigma \sigma^0}$,  $\bar{\Lambda}^{n,\sigma \sigma}$ and $\bar{\Lambda}^{n,\sigma \sigma^0}$  are  $d\times m$-matrices defined as follows,  
\begin{align*}
\Lambda_{\kappa_n(s)}^{n, \sigma \sigma, (u,v) } & \big(s,X_{\kappa_n(s)}^{i,N,n}, \mu_{\kappa_n(s)}^{X,N,n}\big)  := \partial_x \sigma^{(u,v)}_{\kappa_n(s)} \big(X_{\kappa_n(s)}^{i,N,n}, \mu_{\kappa_n(s)}^{X,N,n}\big)   \int_{\kappa_n(s)}^s \sigma^n_{\kappa_n(r)} \big(X_{\kappa_n(r)}^{i,N,n}, \mu_{\kappa_n(r)}^{X,N,n}\big) dW_r^{i}, 
\\
\Lambda_{\kappa_n(s)}^{n, \sigma \sigma^0, (u,v) } & \big(s,X_{\kappa_n(s)}^{i,N,n}, \mu_{\kappa_n(s)}^{X,N,n}\big)  := \partial_x \sigma^{(u,v)}_{\kappa_n(s)} \big(X_{\kappa_n(s)}^{i,N,n}, \mu_{\kappa_n(s)}^{X,N,n}\big)   \int_{\kappa_n(s)}^s \sigma^{0,n}_{\kappa_n(r)} \big(X_{\kappa_n(r)}^{i,N,n}, \mu_{\kappa_n(r)}^{X,N,n}\big) dW_r^{0}, 
\\
 \bar{\Lambda}^{n, \sigma \sigma, (u,v)}_{\kappa_n(s)}  & \big( s,X_{\kappa_n(s)}^{i,N,n}, \mu_{\kappa_n(s)}^{X,N,n}\big)  
\\
& :=  \frac{1}{N}\sum_{j=1}^N  \partial_\mu \sigma^{(u,v)}_{\kappa_n(s)} \big( X_{\kappa_n(s)}^{i,N,n}, \mu_{\kappa_{n}(s)}^{X,N,n},  X_{\kappa_n(s)}^{j,N,n} \big)   \int_{\kappa_n(s)}^s \sigma^n_{\kappa_n(s)} \big(X_{\kappa_n(r)}^{j,N,n}, \mu_{\kappa_n(r)}^{X,N,n}\big) dW_r^{j}, 
\\
\bar{\Lambda}^{n, \sigma \sigma^0, (u,v)}_{\kappa_n(s)} &  \big( s,X_{\kappa_n(s)}^{i,N,n}, \mu_{\kappa_n(s)}^{X,N,n}\big) 
\\
&  :=  \frac{1}{N}\sum_{j=1}^N  \partial_\mu \sigma^{(u,v)}_{\kappa_n(s)} \big( X_{\kappa_n(s)}^{i,N,n}, \mu_{\kappa_{n}(s)}^{X,N,n},  X_{\kappa_n(s)}^{j,N,n} \big)    \int_{\kappa_n(s)}^s \sigma^{0,n}_{\kappa_n(s)} \big(X_{\kappa_n(r)}^{j,N,n}, \mu_{\kappa_n(r)}^{X,N,n}\big) dW_r^{0}. 
\end{align*}
for every $u\in\{1,\ldots,d\}$ and $v\in\{1,\ldots,m\}$. 

Again, for any $s\in[0,T]$,  $i\in\{1,\ldots, N\}$ and $n,N\in\mathbb{N}$,  
\begin{align}
\tilde{\sigma}^{0,n}_{\kappa_n(s)} \big(s,X_{\kappa_n(s)}^{i,N,n}, \mu_{\kappa_n(s)}^{X,N,n}\big)&:=\sigma_{\kappa_n(s)}^{0, n} \big(X_{\kappa_n(s)}^{i,N,n}, \mu_{\kappa_n(s)}^{X,N,n}\big) +\Gamma_{\kappa_{n}(s)}^{n, \sigma^0} \big(s,X_{\kappa_n(s)}^{i,N,n}, \mu_{\kappa_n(s)}^{X,N,n}\big) \label{eq:tilde:sigma:0n}
\end{align}
where $\Gamma^{n, \sigma^0}$  is further expressed as a sum of four matrices, i.e., 
\begin{align*}
\Gamma_{\kappa_{n}(s)}^{n, \sigma^0} \big(s,X_{\kappa_n(s)}^{i,N,n}, \mu_{\kappa_n(s)}^{X,N,n}\big) &:=\Lambda_{\kappa_{n}(s)}^{n, \sigma^0 \sigma} \big(s,X_{\kappa_n(s)}^{i,N,n}, \mu_{\kappa_n(s)}^{X,N,n}\big) +\Lambda_{\kappa_{n}(s)}^{n, \sigma^0 \sigma^0} \big(s,X_{\kappa_n(s)}^{i,N,n}, \mu_{\kappa_n(s)}^{X,N,n}\big)   \notag
\\
& \qquad +\bar{\Lambda}^{n,\sigma^0 \sigma}_{\kappa_n(s)} \big( s,X_{\kappa_n(s)}^{i,N,n}, \mu_{\kappa_n(s)}^{X,N,n}\big)  +\bar{\Lambda}^{n,\sigma^0 \sigma^0}_{\kappa_n(s)} \big( s,X_{\kappa_n(s)}^{i,N,n}, \mu_{\kappa_n(s)}^{X,N,n}\big) \notag
\end{align*}
 where $\Lambda^{n, \sigma^0 \sigma} $, $ \Lambda^{n, \sigma^0 \sigma^0}$, $\bar{\Lambda}^{n,\sigma^0 \sigma}$ and $\bar{\Lambda}^{n,\sigma^0 \sigma^0}$ are $d\times m^0$-matrices whose $(u,v)$-th elements are given in this order by
\begin{align}
\Lambda_{\kappa_n(s)}^{n,  \sigma^0 \sigma , (u,v) } & \big(s,X_{\kappa_n(s)}^{i,N,n}, \mu_{\kappa_n(s)}^{X,N,n}\big)  \notag
\\
&  := \partial_x \sigma^{0,(u,v)}_{\kappa_n(s)} \big(X_{\kappa_n(s)}^{i,N,n}, \mu_{\kappa_n(s)}^{X,N,n}\big) \int_{\kappa_n(s)}^s \sigma^{n}_{\kappa_n(r)} \big(X_{\kappa_n(r)}^{i,N,n}, \mu_{\kappa_n(r)}^{X,N,n}\big) dW_r^{i},  \notag
\\
\Lambda_{\kappa_n(s)}^{n, \sigma^0 \sigma^0, (u,v) } & \big(s,X_{\kappa_n(s)}^{i,N,n}, \mu_{\kappa_n(s)}^{X,N,n}\big)   \notag
\\
& := \partial_x \sigma^{0,(u,v)}_{\kappa_n(s)} \big(X_{\kappa_n(s)}^{i,N,n}, \mu_{\kappa_n(s)}^{X,N,n}\big)  \int_{\kappa_n(s)}^s \sigma^{0,n}_{\kappa_n(r)} \big(X_{\kappa_n(r)}^{i,N,n}, \mu_{\kappa_n(r)}^{X,N,n}\big) dW_r^{0}, \notag
\\
\bar{\Lambda}^{n,  \sigma^0 \sigma, (u,v)}_{\kappa_n(s)} & \big( s,X_{\kappa_n(s)}^{i,N,n}, \mu_{\kappa_n(s)}^{X,N,n}\big) \notag
\\
&   :=  \frac{1}{N}\sum_{j=1}^N  \partial_\mu \sigma^{0, (u,v)}_{\kappa_n(s)} \big( X_{\kappa_n(s)}^{i,N,n}, \mu_{\kappa_{n}(s)}^{X,N,n},  X_{\kappa_n(s)}^{j,N,n} \big)   \int_{\kappa_n(s)}^s \sigma^{n}_{\kappa_n(s)} \big(X_{\kappa_n(r)}^{j,N,n}, \mu_{\kappa_n(r)}^{X,N,n}\big) dW_r^{j}, \notag
\\
\bar{\Lambda}^{n, \sigma^0 \sigma^0, (u,v)}_{\kappa_n(s)} &  \big( s,X_{\kappa_n(s)}^{i,N,n}, \mu_{\kappa_n(s)}^{X,N,n}\big)   \notag
\\
&  :=  \frac{1}{N}\sum_{j=1}^N  \partial_\mu \sigma^{0,(u,v)}_{\kappa_n(s)} \big( X_{\kappa_n(s)}^{i,N,n}, \mu_{\kappa_{n}(s)}^{X,N,n},  X_{\kappa_n(s)}^{j,N,n} \big) \int_{\kappa_n(s)}^s \sigma^{0,n}_{\kappa_n(s)} \big(X_{\kappa_n(r)}^{j,N,n}, \mu_{\kappa_n(r)}^{X,N,n}\big) dW_r^{0}, \notag
\end{align}
for every $u\in\{1,\ldots,d\}$ and $v\in\{1,\ldots,m^0\}$. 
\begin{rem} \label{rem:growth:taming:milstein}
Using equation \eqref{eq:Taming:milstein}, Remarks  \ref{rem:poly:growth} and \ref{rem:der:growth:poly}, one obtains,
\begin{align*}
|b_{t}^n \big(x, \mu)| & \leq  K \min \Big\{n^{1/2}  \left(1+|x| + \mathcal{W}_2(\mu,\delta_0) \right), |b_{t}\big(x, \mu)| \Big\}, \nonumber 
\\
 |\sigma_{t}^n \big(x, \mu)| & \leq  K \min \Big\{ n^{1/4} \left(1+|x| + \mathcal{W}_2(\mu,\delta_0) \right), |\sigma_{t} \big(x, \mu)| \Big\}, \nonumber
  \\
 |\sigma_{t}^{0,n} \big(x, \mu)| & \leq  K \min \Big\{ n^{1/4} \big(1+|x| + \mathcal{W}_2(\mu,\delta_0) \big), |\sigma_{t}^{0} \big(x, \mu)|  \Big\}, \nonumber
 \\
 |\partial_x \sigma^{(u,v)}_t(x,\mu)| |\sigma^n_t(x,\mu)|&  \leq  K    \min \Big\{ n^{1/2} \left(1+|x| + \mathcal{W}_2(\mu,\delta_0) \right), |\partial_x \sigma^{(u,v)}_t(x,\mu)| |\sigma_t(x,\mu)| \Big\}, \nonumber
 \\
 |\partial_x \sigma^{(u,v)}_t(x,\mu)| |\sigma^{0,n}_t(x,\mu)|&  \leq  K    \min \Big\{ n^{1/2} \left(1+|x| + \mathcal{W}_2(\mu,\delta_0) \right), |\partial_x \sigma^{(u,v)}_t(x,\mu)| |\sigma_t^0 (x,\mu)| \Big\}, \nonumber
\\
 |\partial_\mu \sigma^{(u,v)}_t(x,\mu, y)| |\sigma^n_t(x,\mu)|&  \leq  K    \min \Big\{ n^{1/4} \left(1+|x| + \mathcal{W}_2(\mu,\delta_0) \right), |\partial_\mu \sigma^{(u,v)}_t(x,\mu, y)| |\sigma_t(x,\mu)| \Big\}, \nonumber
 \\
 |\partial_\mu \sigma^{(u,v)}_t(x,\mu, y)| |\sigma^{0,n}_t(x,\mu)|&  \leq  K    \min \Big\{ n^{1/4} \left(1+|x| + \mathcal{W}_2(\mu,\delta_0) \right), |\partial_\mu \sigma^{(u,v)}_t(x,\mu,y)| |\sigma_t^0 (x,\mu)| \Big\}, \nonumber
 \\
 |\partial_x \sigma^{0,(u,v)}_t(x,\mu)| |\sigma^n_t(x,\mu)|&  \leq  K    \min \Big\{ n^{1/2} \left(1+|x| + \mathcal{W}_2(\mu,\delta_0) \right), |\partial_x \sigma^{0,(u,v)}_t(x,\mu)| |\sigma_t(x,\mu)| \Big\}, \nonumber
 \\
 |\partial_x \sigma^{0,(u,v)}_t(x,\mu)| |\sigma^{0,n}_t(x,\mu)|&  \leq  K    \min \Big\{ n^{1/2} \left(1+|x| + \mathcal{W}_2(\mu,\delta_0) \right), |\partial_x \sigma^{0,(u,v)}_t(x,\mu)| |\sigma_t^0 (x,\mu)| \Big\}, \nonumber
\\
 |\partial_\mu \sigma^{0,(u,v)}_t(x,\mu,y)| |\sigma^n_t(x,\mu)|&  \leq  K    \min \Big\{ n^{1/4} \left(1+|x| + \mathcal{W}_2(\mu,\delta_0) \right), |\partial_\mu \sigma^{0,(u,v)}_t(x,\mu,y)| |\sigma_t(x,\mu)| \Big\}, \nonumber
 \\
 |\partial_\mu \sigma^{0,(u,v)}_t(x,\mu, y)| |\sigma^{0,n}_t(x,\mu)|&  \leq  K    \min \Big\{ n^{1/4} \left(1+|x| + \mathcal{W}_2(\mu,\delta_0) \right), |\partial_\mu \sigma^{0,(u,v)}_t(x,\mu, y)| |\sigma_t^0 (x,\mu)| \Big\}, \nonumber
 \end{align*}
for all $t \in [0,T]$, $x \in \mathbb{R}^d$ and $\mu \in \mathcal{P}_2( \mathbb{R}^d)$ and for some constant $K>0$ independent of $n$.
\end{rem}
\begin{rem}
In view of Remark \ref{rem:no:taming}, no taming is needed in equation \eqref{eq:Taming:milstein} when $\rho=0$.  
\end{rem}
Before presenting the result on moment boundedness of the Milstein scheme  (\ref{eq:milstein}), we  establish several lemmas as given below. 

\begin{lem} \label{lem:gamma:mb}
Let Assumptions \ref{as:x0}, \ref{as:coercivity}, \ref{as:polynomial:Lipschitz}, \ref{as:der:x:poly:lip} and \ref{as:der:mea:poly:lip} be satisfied. Then, for each $i\in\{1,\ldots,N\}$,
\begin{align*}
E\big| \Gamma_{\kappa_n(s)}^{n, \sigma} \big(s, X_{\kappa_n(s)}^{i,N,n}, \mu_{\kappa_n(s)}^{X,N,n}\big) \big|^{p_0} & \leq K E  \big(1+\big| X_{\kappa_n(s)}^{i,N,n}\big|^2\big)^{p_0/2} +K  E\mathcal{W}_2^{p_0} \big( \mu_{\kappa_n(s)}^{X,N,n}, \delta_0 \big), 
\\
E\big| \Gamma_{\kappa_n(s)}^{n, \sigma^0} \big(s, X_{\kappa_n(s)}^{i,N,n}, \mu_{\kappa_n(s)}^{X,N,n}\big) \big|^{p_0} &  \leq K E  \big(1+\big| X_{\kappa_n(s)}^{i,N,n}\big|^2\big)^{p_0/2} +K  E\mathcal{W}_2^{p_0} \big( \mu_{\kappa_n(s)}^{X,N,n}, \delta_0 \big),
\end{align*}
for all $s\in[0,T]$ and $n,N\in\mathbb{N}$ where the constant $K>0$ does not depend on $n$ and $N$. 
\end{lem}
\begin{proof} 
Using  Remark \ref{rem:growth:taming:milstein}, one obtains
\begin{align*}
E\big|\Lambda_{\kappa_n(s)}^{n, \sigma \sigma, (u,v) }  \big(s,X_{\kappa_n(s)}^{i,N,n}, \mu_{\kappa_n(s)}^{X,N,n}\big)\big|^{p_0} & = E\Big|\partial_x \sigma^{(u,v)}_{\kappa_n(s)} \big(X_{\kappa_n(s)}^{i,N,n}, \mu_{\kappa_n(s)}^{X,N,n}\big)   \int_{\kappa_n(s)}^s \sigma^n_{\kappa_n(r)} \big(X_{\kappa_n(r)}^{i,N,n}, \mu_{\kappa_n(r)}^{X,N,n}\big) dW_r^{i}\Big|^{p_0}  
\\
& \leq K n^{-p_0/2} E \big |\partial_x \sigma^{(u,v)}_{\kappa_n(s)} \big(X_{\kappa_n(s)}^{i,N,n}, \mu_{\kappa_n(s)}^{X,N,n}\big)\sigma^n_{\kappa_n(s)} \big(X_{\kappa_n(s)}^{i,N,n}, \mu_{\kappa_n(s)}^{X,N,n}\big) \big|^{p_0}  
\\
& \leq K E \Big(1+ \big| X_{\kappa_n(s)}^{i,N,n}\big|+\mathcal{W}_2\big( \mu_{\kappa_n(s)}^{X,N,n}, \delta_0 \big) \Big)^{p_0}
\end{align*}
for any $s\in[0,T]$ and $n, N\in\mathbb{R}^d$.  Similarly, one  obtains
\begin{align*}
E\big| & \Lambda_{\kappa_n(s)}^{n, \sigma \sigma^0, (u,v) }  \big(s,X_{\kappa_n(s)}^{i,N,n}, \mu_{\kappa_n(s)}^{X,N,n}\big)\big|^{p_0} \leq K E \Big(1+ \big| X_{\kappa_n(s)}^{i,N,n}\big|+\mathcal{W}_2\big( \mu_{\kappa_n(s)}^{X,N,n}, \delta_0 \big) \Big)^{p_0}
\end{align*}
for any $s\in[0,T]$ and $n, N\in\mathbb{R}^d$. 

Again, the application of Remark \ref{rem:growth:taming:milstein} yields
\begin{align*}
E\big|& \bar{\Lambda}^{n, \sigma \sigma, (u,v)}_{\kappa_n(s)}  \big( s,X_{\kappa_n(s)}^{i,N,n}, \mu_{\kappa_n(s)}^{X,N,n}\big)\big|^{p_0}   \notag  
\\
&  = E\Big|\frac{1}{N}\sum_{j=1}^N  \partial_\mu \sigma^{(u,v)}_{\kappa_n(s)} \big( X_{\kappa_n(s)}^{i,N,n}, \mu_{\kappa_{n}(s)}^{X,N,n},  X_{\kappa_n(s)}^{j,N,n} \big) \int_{\kappa_n(s)}^s \sigma^n_{\kappa_n(r)} \big(X_{\kappa_n(r)}^{j,N,n}, \mu_{\kappa_n(r)}^{X,N,n}\big) dW_r^{j}\Big|^{p_0}  
\\
&  \leq K  \frac{1}{N}\sum_{j=1}^N E \Big|  \partial_\mu \sigma^{(u,v)}_{\kappa_n(s)} \big( X_{\kappa_n(s)}^{i,N,n}, \mu_{\kappa_{n}(s)}^{X,N,n},  X_{\kappa_n(s)}^{j,N,n} \big) \int_{\kappa_n(s)}^s \sigma^n_{\kappa_n(r)} \big(X_{\kappa_n(r)}^{j,N,n}, \mu_{\kappa_n(r)}^{X,N,n}\big) dW_r^{j}\Big|^{p_0}  
\\
&  \leq K n^{-p_0/2} \frac{1}{N}\sum_{j=1}^N E \Big\{  \big| \partial_\mu \sigma^{(u,v)}_{\kappa_n(s)} \big( X_{\kappa_n(s)}^{i,N,n}, \mu_{\kappa_{n}(s)}^{X,N,n},  X_{\kappa_n(s)}^{j,N,n} \big)\big|\big| \sigma^n_{\kappa_n(s)} \big(X_{\kappa_n(s)}^{j,N,n}, \mu_{\kappa_n(s)}^{X,N,n}\big)\big| \Big\}^{p_0} 
\\
& \leq K n^{-p_0/4}   K E \Big(1+ \big| X_{\kappa_n(s)}^{i,N,n}\big|+\mathcal{W}_2\big( \mu_{\kappa_n(s)}^{X,N,n}, \delta_0 \big) \Big)^{p_0}
\end{align*}
for any $s\in[0,T]$ and $n, N\in\mathbb{R}^d$.  Similarly, 
\begin{align*}
E \big| \bar{\Lambda}^{n, \sigma \sigma^0, (u,v)}_{\kappa_n(s)}   \big( s,X_{\kappa_n(s)}^{i,N,n}, \mu_{\kappa_n(s)}^{X,N,n}\big)\big|^{p_0}   \leq K n^{-p_0/4}   K E \Big(1+ \big| X_{\kappa_n(s)}^{i,N,n}\big|+\mathcal{W}_2\big( \mu_{\kappa_n(s)}^{X,N,n}, \delta_0 \big) \Big)^{p_0}  
\end{align*}
for any $s\in[0,T]$ and $n, N\in\mathbb{R}^d$. Adding the above inequalities leads to 
\begin{align*}
E \big| & \Gamma_{\kappa_n(s)}^{n, \sigma}  \big(s, X_{\kappa_n(s)}^{i,N,n}, \mu_{\kappa_n(s)}^{X,N,n}\big)\big|^{p_0}  \leq  K E\big|\Lambda_{\kappa_{n}(s)}^{n, \sigma \sigma} \big(s,X_{\kappa_n(s)}^{i,N,n}, \mu_{\kappa_n(s)}^{X,N,n}\big)\big|^{p_0} 
\\
&  +K E\big|\Lambda_{\kappa_{n}(s)}^{n, \sigma \sigma^0} \big(s,X_{\kappa_n(s)}^{i,N,n}, \mu_{\kappa_n(s)}^{X,N,n}\big) \big|^{p_0}  +K E \big|\bar{\Lambda}^{n,\sigma \sigma}_{\kappa_n(s)} \big( s,X_{\kappa_n(s)}^{i,N,n}, \mu_{\kappa_n(s)}^{X,N,n}\big) \big|^{p_0} 
\\
& \qquad + K E\big| \bar{\Lambda}^{n,\sigma \sigma^0}_{\kappa_n(s)} \big( s,X_{\kappa_n(s)}^{i,N,n}, \mu_{\kappa_n(s)}^{X,N,n}\big) \big|^{p_0} \leq     K E \Big(1+ \big| X_{\kappa_n(s)}^{i,N,n}\big|+\mathcal{W}_2\big( \mu_{\kappa_n(s)}^{X,N,n}, \delta_0 \big) \Big)^{p_0}  
\end{align*}
for any $s\in[0,T]$ and $n, N\in\mathbb{R}^d$. This completes the proof of the first inequality. The second inequality can be proved similarly.  
\end{proof}
As a consequence of the above lemma and Remark \ref{rem:growth:taming:milstein}, one obtains the following corollary. 
\begin{cor} \label{cor:sigma:mb}
Let Assumptions \ref{as:x0}, \ref{as:coercivity}, \ref{as:polynomial:Lipschitz}, \ref{as:der:x:poly:lip} and \ref{as:der:mea:poly:lip} be satisfied. Then, for each $i\in\{1,\ldots,N\}$,
\begin{align*}
E\big| \tilde{\sigma}^n_{\kappa_n(s)} \big(s, X_{\kappa_n(s)}^{i,N,n}, \mu_{\kappa_n(s)}^{X,N,n}\big) \big|^{p_0} & \leq K n^{p_0/4} \Big\{ E  \big(1+\big| X_{\kappa_n(s)}^{i,N,n}\big|^2\big)^{p_0/2} +K  E\mathcal{W}_2^{p_0} \big( \mu_{\kappa_n(s)}^{X,N,n}, \delta_0 \big)\Big\} , 
\\
E\big|  \tilde{\sigma}^{0,n}_{\kappa_n(s)} \big(s, X_{\kappa_n(s)}^{i,N,n}, \mu_{\kappa_n(s)}^{X,N,n}\big) \big|^{p_0} &  \leq K n^{p_0/4}\Big\{ E  \big(1+\big| X_{\kappa_n(s)}^{i,N,n}\big|^2\big)^{p_0/2} +K  E\mathcal{W}_2^{p_0} \big( \mu_{\kappa_n(s)}^{X,N,n}, \delta_0 \big)\Big\},
\end{align*}
for all $s\in[0,T]$ and $n,N\in\mathbb{N}$ where the constant $K>0$ does not depend on $n$ and $N$. 
\end{cor} 
\begin{lem}\label{Lemma:onestepMilstein}
Let Assumptions \ref{as:x0}, \ref{as:coercivity}, \ref{as:polynomial:Lipschitz}, \ref{as:der:x:poly:lip} and  \ref{as:der:mea:poly:lip} be satisfied. Then, 
\begin{align*}
E |X_{s}^{i,N, n} - X_{\kappa_n(s)}^{i,N, n}|^{p_0} &\leq K n^{-p_0/4} \Big\{ E  \big(1+\big| X_{\kappa_n(s)}^{i,N,n}\big|^2\big)^{p_0/2} +K  E\mathcal{W}_2^{p_0} \big( \mu_{\kappa_n(s)}^{X,N,n}, \delta_0 \big)\Big\},
\end{align*}
for any $t\in[0,T]$, $i\in\{1,\ldots,N\}$ and $n,N\in\mathbb{N}$ where the  constant $K>0$ is independent of $N$ and $n$.
\end{lem}
\begin{proof}
From (\ref{eq:milstein}), 
\begin{align*}
|X_{s}^{i,N, n} - X_{\kappa_n(s)}^{i,N, n}|^{p_0} &\leq  K \Big| \int_{\kappa_n(s)}^s b_{\kappa_n(r)}^n \big(X_{\kappa_n(r)}^{i,N, n},  \mu_{\kappa_n(r)}^{X,N,n}\big) dr \Big|^{p_0}  +  K \Big| \int_{\kappa_n(s)}^s \tilde{\sigma}_{\kappa_n(r)}^n\big(r, X_{\kappa_n(r)}^{i,N,n}, \mu_{\kappa_n(r)}^{X,N,n}\big) dW_r^{i} \Big|^{p_0} \nonumber
\\
& \qquad +  K \Big| \int_{\kappa_n(s)}^s \tilde{\sigma}^{0,n}_{\kappa_n(r)}\big(r, X_{\kappa_n(r)}^{i,N,n}, \mu_{\kappa_n(r)}^{X,N,n}\big) dW_r^{0} \Big|^{p_0}.
\end{align*}
which on applying H\"{o}lder's inequality and Burkholder--Davis--Gundy inequality yields
\begin{align*}
E |X_{s}^{i,N, n}  - X_{\kappa_n(s)}^{i,N, n}|^{p_0}   \leq & K n^{-p_0+1} E \int_{\kappa_n(s)}^s \left|b_{\kappa_n(r)}^n \big(X_{\kappa_n(r)}^{i,N, n},  \mu_{\kappa_n(r)}^{X,N,n}\big)\right|^{p_0} dr
  \\
& + K n^{-p_0/2+1} E  \int_{\kappa_n(s)}^s \big | \tilde{\sigma}_{\kappa_n(r)}^{n}\big(r, X_{\kappa_n(r)}^{i,N,n}, \mu_{\kappa_n(r)}^{X,N,n}\big) \big |^{p_0}  dr  
\\
& + K n^{-p_0/2+1}  E  \int_{\kappa_n(s)}^s \big | \tilde{\sigma}_{\kappa_n(r)}^{0,n} \big(r, X_{\kappa_n(r)}^{i,N,n}, \mu_{\kappa_n(r)}^{X,N,n}\big) \big  |^{p_0} dr
\end{align*}
and then the result follows on using  Corollary \ref{cor:sigma:mb} and Remark \ref{rem:growth:taming:milstein}. 
\end{proof}

\begin{lem}[\textbf{Moment Bounds}] \label{lem:mb:milstein}
Let Assumptions \ref{as:x0}, \ref{as:coercivity}, \ref{as:polynomial:Lipschitz}, \ref{as:der:x:poly:lip} and  \ref{as:der:mea:poly:lip} be satisfied. Then, 
\begin{equation*}
\sup_{i \in \lbrace 1, \ldots, N \rbrace } \sup_{t \in [ 0,T ]} E \big(1 + |X_t^{i,N, n}|^2 \big)^{p_0/2} \leq K, 
\end{equation*}
for any $n, N\in\mathbb{N}$ where $K>0$ is a constant independent of $n$ and $N$. Moreover, 
\begin{equation*}
\sup_{i \in \lbrace 1, \ldots, N \rbrace } E \sup_{t \in [ 0,T ]} \big(1 + |X_t^{i,N, n}|^2 \big)^{q/2} \leq K,
\end{equation*}
for any $q<p_0$. 
\end{lem}
\begin{proof}
On using the arguments used in the proof of Lemma \ref{lem:mb:euler} with equation \eqref{eq:milstein},  we obtain the following equation analogous to equation \eqref{eq:ito:euler}, 
\begin{align}
E \big(1 + |X_t^{i,N, n}|^{2} \big)^{p_0/2}  \leq  & E\big( 1 + |X_0^{i,N, n}|^{2} \big)^{p_0/2} \notag
 \\
& + \frac{p_0}{2}   E \int_{0}^{t} \left(1+ |X_{s}^{i,N, n}|^2 \right)^{p_0/2 -1} \Big\{ 2 X_{\kappa_n(s)}^{i,N, n} b_{\kappa_n(s)}^n \big(X_{\kappa_n(s)}^{i,N, n}, \mu_{\kappa_n(s)}^{X,N,n}\big) \notag
\\
&  +  (p_0-1)   \big|  \tilde{\sigma}_{\kappa_n(s)}^{n} \big(s, X_{\kappa_n(s)}^{i,N,n}, \mu_{\kappa_n(s)}^{X,N,n}\big)  \big|^2  +  (p_0-1)  \big|  \tilde{\sigma}_{\kappa_n(s)}^{0,n} \big(s, X_{\kappa_n(s)}^{i,N,n}, \mu_{\kappa_n(s)}^{X,N,n}\big)   \big|^2 \Big\} ds \notag
\\  
& + p_0  E \int_{0}^{t} \left(1+ |X_{s}^{i,N, n}|^2 \right)^{p_0/2 -1} \big( X_s^{i,N, n}  - X_{\kappa_n(s)}^{i,N, n} \big) b_{\kappa_n(s)}^n \big(X_{\kappa_n(s)}^{i,N, n}, \mu_{\kappa_n(s)}^{X,N,n}\big) ds \notag
\end{align}
 for any $t\in[0,T]$, $i\in\{1,\ldots,N\}$ and $n,N\in\mathbb{N}$. Observe that $\tilde{\sigma}^n$ and $\tilde{\sigma}^{0,n}$ are sum of two matrices, see equations \eqref{eq:tilde:sigma:n} and \eqref{eq:tilde:sigma:0n}. Thus, on using  
 $
 |A+B|^2 =|A|^2+|B|^2+2\sum_{u=1}^d \sum_{v=1}^m A^{(u,v)}B^{(u,v)}
 $
for matrices $A$ and $B$, one obtains, 
\begin{align*}
E \big(1 +  |X_t^{i,N, n}|^{2} \big)^{p_0/2}   \leq  & E \big( 1 + |X_0^{i,N, n}|^{2} \big)^{p_0/2} \nonumber 
\\
&  + \frac{p_0}{2} E  \int_{0}^{t} \left(1+ |X_{s}^{i,N, n}|^2 \right)^{p_0/2 -1} \Big\{X_{\kappa_n(s)}^{i,N, n} b_{\kappa_n(s)}^n \big(X_{\kappa_n(s)}^{i,N, n}, \mu_{\kappa_n(s)}^{X,N,n}\big)  
\\
& +  (p_0-1)  \big|  \sigma_{\kappa_n(s)}^{n} \big(X_{\kappa_n(s)}^{i,N,n}, \mu_{\kappa_n(s)}^{X,N,n}\big) \big|^2  + (p_0-1)  \big| \sigma_{\kappa_n(s)}^{0,n} \big(X_{\kappa_n(s)}^{i,N,n}, \mu_{\kappa_n(s)}^{X,N,n}\big) \big|^2 \Big\}ds \nonumber 
\\ 
&  + p_0  E \int_{0}^{t} \left(1+ |X_{s}^{i,N, n}|^2 \right)^{p_0/2 -1} \big(X_{s}^{i,N, n}-X_{\kappa_n(s)}^{i,N, n}\big) b_{\kappa_n(s)}^n \big(X_{\kappa_n(s)}^{i,N, n}, \mu_{\kappa_n(s)}^{X,N,n}\big) ds 
\\ 
&  +  \frac{p_0(p_0-1)}{2} E \int_{0}^{t} \left(1+ |X_{s}^{i,N, n}|^2 \right)^{p_0/2 -1} \big|  \Gamma_{\kappa_n(s)}^{n,\sigma} \big(s, X_{\kappa_n(s)}^{i,N,n}, \mu_{\kappa_n(s)}^{X,N,n}\big) \big|^2 ds \nonumber
 \\ 
&  +  \frac{p_0(p_0-1)}{2} E \int_{0}^{t} \left(1+ |X_{s}^{i,N, n}|^2 \right)^{p_0/2 -1} \big| \Gamma_{\kappa_n(s)}^{n,\sigma^{0}} \big(s, X_{\kappa_n(s)}^{i,N,n}, \mu_{\kappa_n(s)}^{X,N,n}\big) \big|^2 ds \nonumber 
\\ 
&  + p_0(p_0-1) E \int_{0}^{t} \left(1+ |X_{s}^{i,N, n}|^2 \right)^{p_0/2 -1} \sum_{u=1}^{d} \sum_{v=1}^{m} \sigma_{\kappa_n(s)}^{n,(u,v)} \big(X_{\kappa_n(s)}^{i,N,n}, \mu_{\kappa_n(s)}^{X,N,n}\big) 
\\
& \qquad\qquad \times \Gamma_{\kappa_n(s)}^{n,\sigma,(u,v)} \big(s, X_{\kappa_n(s)}^{i,N,n}, \mu_{\kappa_n(s)}^{X,N,n}\big)  ds \nonumber 
\\ 
& + p_0(p_0-1) E \int_{0}^{t} \left(1+ |X_{s}^{i,N, n}|^2 \right)^{p_0/2 -1} \sum_{u=1}^{d} \sum_{v=1}^{m} \sigma_{\kappa_n(s)}^{n,0,(u,v)} \big(X_{\kappa_n(s)}^{i,N,n}, \mu_{\kappa_n(s)}^{X,N,n}\big) 
\\
& \qquad\qquad \times \Gamma_{\kappa_n(s)}^{n,\sigma^{0},(u,v)} \big(s, X_{\kappa_n(s)}^{i,N,n}, \mu_{\kappa_n(s)}^{X,N,n}\big)  ds \nonumber 
\end{align*}
which on using Assumption \ref{as:coercivity}, Young's inequality, equation \eqref{eq:milstein} and Lemma \ref{lem:gamma:mb} yields, 
\begin{align}
 E \big(1 + & |X_t^{i,N, n}|^{2} \big)^{p_0/2} \leq E \big( 1 + |X_0^{i,N, n}|^{2} \big)^{p_0/2}  + K \int_{0}^{t} \sup_{0 \leq r \leq s} E  \left(1 + |X_r^{i,N, n}|^{2} \right)^{p_0/2} ds \nonumber
 \\
&  +  K  E \int_{0}^{t}  \big|  \Gamma_{\kappa_n(s)}^{n,\sigma} \big(s, X_{\kappa_n(s)}^{i,N,n}, \mu_{\kappa_n(s)}^{X,N,n}\big) \big|^{p_0} ds + K E \int_{0}^{t} \big| \Gamma_{\kappa_n(s)}^{n,\sigma^{0}} \big(s, X_{\kappa_n(s)}^{i,N,n}, \mu_{\kappa_n(s)}^{X,N,n}\big) \big|^{p_0} ds  \nonumber 
 \\
&  + K E \int_{0}^{t} \left(1+ |X_{s}^{i,N, n}|^2 \right)^{p_0/2 -1} \int_{\kappa_n(s)}^s  b^n_{\kappa_n(r)} \big(X_{\kappa_n(r)}^{i,N, n}, \mu_{\kappa_n(r)}^{X,N,n}\big) dr b_{\kappa_n(s)}^n \big(X_{\kappa_n(s)}^{i,N, n}, \mu_{\kappa_n(s)}^{X,N,n}\big) ds \nonumber  
\\
&  + K E \int_{0}^{t} \left(1+ |X_{s}^{i,N, n}|^2 \right)^{p_0/2 -1} \int_{\kappa_n(s)}^s  \tilde{\sigma}^n_{\kappa_n(r)} \big(r,X_{\kappa_n(r)}^{i,N,n}, \mu_{\kappa_n(r)}^{X,N,n}\big) dW_r^{i} b_{\kappa_n(s)}^n \big(X_{\kappa_n(s)}^{i,N, n}, \mu_{\kappa_n(s)}^{X,N,n}\big) ds \nonumber  
\\
&  + K E \int_{0}^{t} \left(1+ |X_{s}^{i,N, n}|^2 \right)^{p_0/2 -1} \int_{\kappa_n(s)}^s  \tilde{\sigma}^{0,n}_{\kappa_n(r)} \big(r,X_{\kappa_n(r)}^{i,N,n}, \mu_{\kappa_n(r)}^{X,N,n}\big) dW_r^{0}   b_{\kappa_n(s)}^n \big(X_{\kappa_n(s)}^{i,N, n}, \mu_{\kappa_n(s)}^{X,N,n}\big) ds \nonumber  
\\
& + K E \int_{0}^{t} \left(1+ |X_{s}^{i,N, n}|^2 \right)^{p_0/2 -1} \sum_{u=1}^{d} \sum_{v=1}^{m} \sigma_{\kappa_n(s)}^{n,(u,v)} \big(X_{\kappa_n(s)}^{i,N,n}, \mu_{\kappa_n(s)}^{X,N,n}\big) \Gamma_{\kappa_n(s)}^{n,\sigma,(u,v)} \big(s, X_{\kappa_n(s)}^{i,N,n}, \mu_{\kappa_n(s)}^{X,N,n}\big)  ds \nonumber 
\\ 
&  + K E \int_{0}^{t} \left(1+ |X_{s}^{i,N, n}|^2 \right)^{p_0/2 -1} \sum_{u=1}^{d} \sum_{v=1}^{m} \sigma_{\kappa_n(s)}^{n,0,(u,v)} \big(X_{\kappa_n(s)}^{i,N,n}, \mu_{\kappa_n(s)}^{X,N,n}\big) \Gamma_{\kappa_n(s)}^{n,\sigma^{0},(u,v)} \big(s, X_{\kappa_n(s)}^{i,N,n}, \mu_{\kappa_n(s)}^{X,N,n}\big)  ds \nonumber 
\\
 =: & E \big( 1 + |X_0^{i,N, n}|^{2} \big)^{p_0/2}  + K \int_{0}^{t} \sup_{0 \leq r \leq s} E  \left(1 + |X_r^{i,N, n}|^{2} \right)^{p_0/2} ds  + \sum_{i=1}^{7} \Pi_i,  \label{eq:Momentp1Mils}
\end{align}
 for any $t\in[0,T]$, $i\in\{1,\ldots,N\}$ and $n,N\in\mathbb{N}$. 
  
 By Lemma \ref{lem:gamma:mb}, Remark \ref{rem:growth:taming:milstein} and Young's inequality, one obtains
 \begin{align}
 \Pi_1+&\Pi_2 +\Pi_3 :=  K  E \int_{0}^{t}  \big|  \Gamma_{\kappa_n(s)}^{n,\sigma} \big(s, X_{\kappa_n(s)}^{i,N,n}, \mu_{\kappa_n(s)}^{X,N,n}\big) \big|^{p_0} ds + K E \int_{0}^{t} \big| \Gamma_{\kappa_n(s)}^{n,\sigma^{0}} \big(s, X_{\kappa_n(s)}^{i,N,n}, \mu_{\kappa_n(s)}^{X,N,n}\big) \big|^{p_0} ds  \nonumber 
 \\
 & + K E \int_{0}^{t} \left(1+ |X_{s}^{i,N, n}|^2 \right)^{p_0/2 -1} \int_{\kappa_n(s)}^s  b^n_{\kappa_n(r)} \big(X_{\kappa_n(r)}^{i,N, n}, \mu_{\kappa_n(r)}^{X,N,n}\big) dr b_{\kappa_n(s)}^n \big(X_{\kappa_n(s)}^{i,N, n}, \mu_{\kappa_n(s)}^{X,N,n}\big) ds  \nonumber  
 \\
  \leq & K E  \big(1+\big| X_{\kappa_n(s)}^{i,N,n}\big|^2\big)^{p_0/2} +K  E\mathcal{W}_2^{p_0} \big( \mu_{\kappa_n(s)}^{X,N,n}, \delta_0 \big) \label{eq:pi1+pi2+pi3}
 \end{align}
 for any $t\in[0,T]$, $i\in\{1,\ldots,N\}$ and $n,N\in\mathbb{N}$. 

Notice that $\Pi_4$ is similar to $F_1$ in the proof of Lemma \ref{lem:mb:euler} and hence  by adapting the same technique, one can obtain an analogue of inequality \eqref{eq:F1:ms} with $\sigma^n$ replaced by $\tilde{\sigma}^n$, i.e.,  
 \begin{align*}
 \Pi_4:= &  K E \int_{0}^{t} \left(1+ |X_{s}^{i,N, n}|^2 \right)^{p_0/2 -1} \int_{\kappa_n(s)}^s  \tilde{\sigma}^n_{\kappa_n(r)} \big(r,X_{\kappa_n(r)}^{i,N,n}, \mu_{\kappa_n(r)}^{X,N,n}\big) dW_r^{i} b_{\kappa_n(s)}^n \big(X_{\kappa_n(s)}^{i,N, n}, \mu_{\kappa_n(s)}^{X,N,n}\big) ds \nonumber  
\\
 \leq &   K  \int_{0}^{t}  \sup_{0\leq r\leq s}E \left(1+ |X_{r}^{i,N, n}|^2 \right)^{p_0/2} ds \notag 
\\
& + K n^{-p_0/8} n^{-p_0/4+1}  \int_{0}^{t}  E  \int_{\kappa_n(s)}^s \big|b_{\kappa_n(s)}^n \big(X_{\kappa_n(s)}^{i,N, n}, \mu_{\kappa_n(s)}^{X,N,n}\big)\big|^{p_0/2} \big|\tilde{\sigma}_{\kappa_n(r)}^n\big(r, X_{\kappa_n(r)}^{i,N,n}, \mu_{\kappa_n(r)}^{X,N,n}\big) \big|^{p_0/2}dr  ds 
 \\
 \leq &   K  \int_{0}^{t}  \sup_{0\leq r\leq s}E \left(1+ |X_{r}^{i,N, n}|^2 \right)^{p_0/2} ds \notag 
\\
& + K  n  \int_{0}^{t}  E  \int_{\kappa_n(s)}^s n^{-p_0/4}  \big|b_{\kappa_n(s)}^n \big(X_{\kappa_n(s)}^{i,N, n}, \mu_{\kappa_n(s)}^{X,N,n}\big)\big|^{p_0/2}n^{-p_0/8} \big|\tilde{\sigma}_{\kappa_n(r)}^n\big(r, X_{\kappa_n(r)}^{i,N,n}, \mu_{\kappa_n(r)}^{X,N,n}\big) \big|^{p_0/2}dr  ds 
\end{align*}
which on the application of Young's inequality, Corollary \ref{cor:sigma:mb} and Remark \ref{rem:growth:taming:milstein} yields
\begin{align}
\Pi_4\leq  &   K  \int_{0}^{t}  \sup_{0\leq r\leq s}E \left(1+ |X_{r}^{i,N, n}|^2 \right)^{p_0/2} ds \notag 
\\
& + K  n \int_{0}^{t}   \int_{\kappa_n(s)}^s  \Big\{n^{-p_0/2}E   \big|b_{\kappa_n(s)}^n \big(X_{\kappa_n(s)}^{i,N, n}, \mu_{\kappa_n(s)}^{X,N,n}\big)\big|^{p_0} +n^{-p_0/4} E  \big|\tilde{\sigma}_{\kappa_n(r)}^n\big(r, X_{\kappa_n(r)}^{i,N,n}, \mu_{\kappa_n(r)}^{X,N,n}\big) \big|^{p_0} \Big\} dr  ds \notag
\\
\leq &   K  \int_{0}^{t}  \sup_{0\leq r\leq s}E \left(1+ |X_{r}^{i,N, n}|^2 \right)^{p_0/2} ds +K   \int_{0}^{t}   E\mathcal{W}_2^{p_0} \big( \mu_{\kappa_n(s)}^{X,N,n}, \delta_0 \big) ds \label{eq:pi4}
\end{align}
 for any $t\in[0,T]$, $i\in\{1,\ldots,N\}$ and $n,N\in\mathbb{N}$. 
 
 By following the steps of estimating $\Pi_4$, one can easily obtains
 \begin{align}
 \Pi_5 & := K E \int_{0}^{t} \left(1+ |X_{s}^{i,N, n}|^2 \right)^{p_0/2 -1} \int_{\kappa_n(s)}^s  \tilde{\sigma}^{0,n}_{\kappa_n(r)} \big(r,X_{\kappa_n(r)}^{i,N,n}, \mu_{\kappa_n(r)}^{X,N,n}\big) dW_r^{0}   b_{\kappa_n(s)}^n \big(X_{\kappa_n(s)}^{i,N, n}, \mu_{\kappa_n(s)}^{X,N,n}\big) ds \nonumber  
 \\
 &\leq   K  \int_{0}^{t}  \sup_{0\leq r\leq s}E \left(1+ |X_{r}^{i,N, n}|^2 \right)^{p_0/2} ds +K   \int_{0}^{t}   E\mathcal{W}_2^{p_0} \big( \mu_{\kappa_n(s)}^{X,N,n}, \delta_0 \big) ds \label{eq:pi5}
 \end{align}
 for any $t\in[0,T]$, $i\in\{1,\ldots,N\}$ and $n,N\in\mathbb{N}$.  
 
 It remains to analyse $\Pi_6$ and $\Pi_7$ now. For $\Pi_6$, it is easy to see that
 \begin{align*}
 \Pi_6:= & K E \int_{0}^{t} \left(1+ |X_{s}^{i,N, n}|^2 \right)^{p_0/2 -1} \sum_{u=1}^{d} \sum_{v=1}^{m} \sigma_{\kappa_n(s)}^{n,(u,v)} \big(X_{\kappa_n(s)}^{i,N,n}, \mu_{\kappa_n(s)}^{X,N,n}\big) \Gamma_{\kappa_n(s)}^{n,\sigma,(u,v)} \big(s, X_{\kappa_n(s)}^{i,N,n}, \mu_{\kappa_n(s)}^{X,N,n}\big)  ds \nonumber 
 \\
 \leq & K E \int_{0}^{t} \big(1+ |X_{\kappa_n(s)}^{i,N, n}|^2 \big)^{p_0/2 -1} \sum_{u=1}^{d} \sum_{v=1}^{m} \sigma_{\kappa_n(s)}^{n,(u,v)} \big(X_{\kappa_n(s)}^{i,N,n}, \mu_{\kappa_n(s)}^{X,N,n}\big) \Gamma_{\kappa_n(s)}^{n,\sigma,(u,v)} \big(s, X_{\kappa_n(s)}^{i,N,n}, \mu_{\kappa_n(s)}^{X,N,n}\big)  ds \nonumber 
 \\
 & + K E \int_{0}^{t} \Big| \left(1+ |X_{s}^{i,N, n}|^2 \right)^{p_0/2 -1} - \big(1+ |X_{\kappa_n(s)}^{i,N, n}|^2 \big)^{p_0/2 -1} \Big| 
 \\
 & \qquad \times\Big| \sum_{u=1}^{d} \sum_{v=1}^{m} \sigma_{\kappa_n(s)}^{n,(u,v)} \big(X_{\kappa_n(s)}^{i,N,n}, \mu_{\kappa_n(s)}^{X,N,n}\big) \Gamma_{\kappa_n(s)}^{n,\sigma,(u,v)} \big(s, X_{\kappa_n(s)}^{i,N,n}, \mu_{\kappa_n(s)}^{X,N,n}\big)  ds \Big| \nonumber 
 \end{align*}
  for any $t\in[0,T]$, $i\in\{1,\ldots,N\}$ and $n,N\in\mathbb{N}$. 
 Notice that  $\Gamma^{n,\sigma}$ is a martingale and thus the first term in the right hand side vanishes. For the second term, one uses inequality \eqref{eq:mvt:mb} and Young's inequality to obtain
 \begin{align*}
\Pi_6 \leq &  K E \int_{0}^{t} \big\{1+ \big| X_{s}^{i,N, n}\big|^2 + \big|X_{\kappa_n(s)}^{i,N, n} \big|^2  \big\}^{(p_0-3)/2} \big| X_{s}^{i,N, n}- X_{\kappa_n(s)}^{i,N, n}\big|
\\
 & \qquad \times\Big| \sum_{u=1}^{d} \sum_{v=1}^{m} \sigma_{\kappa_n(s)}^{n,(u,v)} \big(X_{\kappa_n(s)}^{i,N,n}, \mu_{\kappa_n(s)}^{X,N,n}\big) \Gamma_{\kappa_n(s)}^{n,\sigma,(u,v)} \big(s, X_{\kappa_n(s)}^{i,N,n}, \mu_{\kappa_n(s)}^{X,N,n}\big)  ds \Big| \nonumber 
 \\
 \leq & K  \int_{0}^{t}   \big\{1+ \big| X_{s}^{i,N, n}\big|^2 + \big|X_{\kappa_n(s)}^{i,N, n} \big|^2  \big\}^{p_0/2} ds + K  E\int_{0}^{t} n^{p_0/12}\big| X_{s}^{i,N, n}- X_{\kappa_n(s)}^{i,N, n}\big|^{p_0/3}
\\
 & \qquad \times n^{-p_0/12} \Big| \sum_{u=1}^{d} \sum_{v=1}^{m} \sigma_{\kappa_n(s)}^{n,(u,v)} \big(X_{\kappa_n(s)}^{i,N,n}, \mu_{\kappa_n(s)}^{X,N,n}\big) \Gamma_{\kappa_n(s)}^{n,\sigma,(u,v)} \big(s, X_{\kappa_n(s)}^{i,N,n}, \mu_{\kappa_n(s)}^{X,N,n}\big)  ds \Big|^{p_0/3} ds 
 \\
 \leq & K  \int_{0}^{t}  \sup_{0\leq r\leq s}E \left(1+ |X_{r}^{i,N, n}|^2 \right)^{p_0/2} ds + K  \int_{0}^{t} n^{p_0/4} E \big| X_{s}^{i,N, n}- X_{\kappa_n(s)}^{i,N, n}\big|^{p_0} ds 
 \\
 & + K n^{-p_0/8} E\int_{0}^{t}  \Big| \sum_{u=1}^{d} \sum_{v=1}^{m} \sigma_{\kappa_n(s)}^{n,(u,v)} \big(X_{\kappa_n(s)}^{i,N,n}, \mu_{\kappa_n(s)}^{X,N,n}\big) \Gamma_{\kappa_n(s)}^{n,\sigma,(u,v)} \big(s, X_{\kappa_n(s)}^{i,N,n}, \mu_{\kappa_n(s)}^{X,N,n}\big)  ds \Big|^{p_0/2} ds 
 \end{align*}
 which on the application of H\"older's inequality, Remark \ref{rem:growth:taming:milstein} and Lemma \ref{lem:gamma:mb} yields 
 \begin{align}
 \Pi_6  \leq  & K  \int_{0}^{t}  \sup_{0\leq r\leq s}E \left(1+ |X_{r}^{i,N, n}|^2 \right)^{p_0/2} ds + K  \int_{0}^{t} n^{p_0/4} E \big| X_{s}^{i,N, n}- X_{\kappa_n(s)}^{i,N, n}\big|^{p_0} ds \notag
 \\
& + K n^{-p_0/8} \int_{0}^{t}   \sum_{u=1}^{d} \sum_{v=1}^{m}  \big\{E \big|\sigma_{\kappa_n(s)}^{n,(u,v)} \big(X_{\kappa_n(s)}^{i,N,n}, \mu_{\kappa_n(s)}^{X,N,n}\big)\big|^{p_0}\big\}^{1/2} \big\{E\big| \Gamma_{\kappa_n(s)}^{n,\sigma,(u,v)} \big(s, X_{\kappa_n(s)}^{i,N,n}, \mu_{\kappa_n(s)}^{X,N,n}\big)\big|^{p_0} \big\}^{1/2}ds  \notag 
\\
\leq &  K  \int_{0}^{t}  \sup_{0\leq r\leq s}E \left(1+ |X_{r}^{i,N, n}|^2 \right)^{p_0/2} ds +  K  \int_{0}^{t}  \mathcal{W}_2^{p_0}\big(\mu_{\kappa_n(s)}^{X,N,n}, \delta_0\big) ds  \label{eq:Pi6}
 \end{align}
 for any $t\in[0,T]$, $i\in\{1,\ldots,N\}$ and $n,N\in\mathbb{N}$.  Notice that $\Pi_6$
 and $\Pi_7$ are similar terms and hence one also obtains, 
\begin{align}
\Pi_7& := K E \int_{0}^{t} \left(1+ |X_{s}^{i,N, n}|^2 \right)^{p_0/2 -1} \sum_{u=1}^{d} \sum_{v=1}^{m} \sigma_{\kappa_n(s)}^{n,0,(u,v)} \big(X_{\kappa_n(s)}^{i,N,n}, \mu_{\kappa_n(s)}^{X,N,n}\big) \Gamma_{\kappa_n(s)}^{n,\sigma^{0},(u,v)} \big(s, X_{\kappa_n(s)}^{i,N,n}, \mu_{\kappa_n(s)}^{X,N,n}\big)  ds \nonumber
\\
&\leq   K  \int_{0}^{t}  \sup_{0\leq r\leq s}E \left(1+ |X_{r}^{i,N, n}|^2 \right)^{p_0/2} ds +  K  \int_{0}^{t}  \mathcal{W}_2^{p_0}\big(\mu_{\kappa_n(s)}^{X,N,n}, \delta_0\big) ds   \label{eq:Pi7}
\end{align}
 for any $t\in[0,T]$, $i\in\{1,\ldots,N\}$ and $n,N\in\mathbb{N}$.  
 
 On substituting the estimates obtained in \eqref{eq:pi1+pi2+pi3} to \eqref{eq:Pi7} in equation \eqref{eq:Momentp1Mils} gives 
 \begin{align*}
 E \big(1 +  |X_t^{i,N, n}|^{2} \big)^{p_0/2} \leq & E \big( 1 + |X_0^{i,N, n}|^{2} \big)^{p_0/2} +K  \int_{0}^{t}  \sup_{0\leq r\leq s}E \left(1+ |X_{r}^{i,N, n}|^2 \right)^{p_0/2} ds 
 \\
 & +  K  \int_{0}^{t}  \mathcal{W}_2^{p_0}\big(\mu_{\kappa_n(s)}^{X,N,n}, \delta_0\big) ds   
 \end{align*}
 for any $t\in[0,T]$, $i\in\{1,\ldots,N\}$ and $n, N\in\mathbb{N}$. Thus, the proof of the first inequality can be completed by using equation \eqref{eq:w2:mb} and Gronwall's inequality. The second inequality follows by Lemma \ref{lem:gk}. 
\end{proof}

We now proceed to the rate of convergence of the tamed Milstein scheme \eqref{eq:milstein}. For this, we prove  several lemmas as given below.  

\begin{lem} \label{lem:gamma:rate}
Let Assumptions \ref{as:x0}, \ref{as:coercivity}, \ref{as:polynomial:Lipschitz}, \ref{as:der:x:poly:lip} and  \ref{as:der:mea:poly:lip} be satisfied. Then, 
\begin{align*}
E\big| \Gamma_{\kappa_n(s)}^{n, \sigma} \big(s, X_{\kappa_n(s)}^{i,N,n}, \mu_{\kappa_n(s)}^{X,N,n}\big) \big|^{p} \leq K n^{-{p}/2}, 
\\
E\big| \Gamma_{\kappa_n(s)}^{n, \sigma^0} \big(s, X_{\kappa_n(s)}^{i,N,n}, \mu_{\kappa_n(s)}^{X,N,n}\big) \big|^{p} \leq K n^{-{p}/2}, 
\end{align*}
for all $p\leq p_0/(\rho/2+1)$, $s\in[0,T]$, $i\in\{1,\ldots,N\}$  and $n,N\in\mathbb{N}$ where the constant $K>0$ does not depend on $n$ and $N$. 
\end{lem}
\begin{proof} 
Using   Remarks \ref{rem:der:growth:poly}, \ref{rem:growth:taming:milstein}, Lemma \ref{lem:mb:milstein} and equation \eqref{eq:w2:mb},
\begin{align*}
E\big|\Lambda_{\kappa_n(s)}^{n, \sigma \sigma, (u,v) } & \big(s,X_{\kappa_n(s)}^{i,N,n}, \mu_{\kappa_n(s)}^{X,N,n}\big)\big|^{p}  = E\Big|\partial_x \sigma^{(u,v)}_{\kappa_n(s)} \big(X_{\kappa_n(s)}^{i,N,n}, \mu_{\kappa_n(s)}^{X,N,n}\big)   \int_{\kappa_n(s)}^s \sigma^n_{\kappa_n(r)} \big(X_{\kappa_n(r)}^{i,N,n}, \mu_{\kappa_n(r)}^{X,N,n}\big) dW_r^{i}\Big|^{p}  
\\
& \leq K n^{-p/2} E \big |\partial_x \sigma^{(u,v)}_{\kappa_n(s)} \big(X_{\kappa_n(s)}^{i,N,n}, \mu_{\kappa_n(s)}^{X,N,n}\big)\sigma^n_{\kappa_n(s)} \big(X_{\kappa_n(s)}^{i,N,n}, \mu_{\kappa_n(s)}^{X,N,n}\big) \big|^{p}   \leq K n^{-p/2}
\end{align*}
for any $s\in[0,T]$ and $n, N\in\mathbb{R}^d$.  Similarly, one  obtains, 
\begin{align*}
E\big| & \Lambda_{\kappa_n(s)}^{n, \sigma \sigma^0, (u,v) }  \big(s,X_{\kappa_n(s)}^{i,N,n}, \mu_{\kappa_n(s)}^{X,N,n}\big)\big|^{p}
\\
& = E\Big| \partial_x \sigma^{(u,v)}_{\kappa_n(s)} \big(X_{\kappa_n(s)}^{i,N,n}, \mu_{\kappa_n(s)}^{X,N,n}\big)    \int_{\kappa_n(s)}^s \sigma^{0,n}_{\kappa_n(r)} \big(X_{\kappa_n(r)}^{i,N,n}, \mu_{\kappa_n(r)}^{X,N,n}\big) dW_r^{0}\Big|^{p}  \leq K n^{-p/2},
\\
E\big|& \bar{\Lambda}^{n, \sigma \sigma, (u,v)}_{\kappa_n(s)}  \big( s,X_{\kappa_n(s)}^{i,N,n}, \mu_{\kappa_n(s)}^{X,N,n}\big)\big|^{p}   \notag  
\\
&  = E\Big|\frac{1}{N}\sum_{j=1}^N  \partial_\mu \sigma^{(u,v)}_{\kappa_n(s)} \big( X_{\kappa_n(s)}^{i,N,n}, \mu_{\kappa_{n}(s)}^{X,N,n},  X_{\kappa_n(s)}^{j,N,n} \big) \int_{\kappa_n(s)}^s \sigma^n_{\kappa_n(s)} \big(X_{\kappa_n(r)}^{j,N,n}, \mu_{\kappa_n(r)}^{X,N,n}\big) dW_r^{j}\Big|^{p}   \leq K n^{-p/2},
\\
E \big|& \bar{\Lambda}^{n, \sigma \sigma^0, (u,v)}_{\kappa_n(s)}   \big( s,X_{\kappa_n(s)}^{i,N,n}, \mu_{\kappa_n(s)}^{X,N,n}\big)\big|^{p}   \notag
\\
& = E\Big| \frac{1}{N}\sum_{j=1}^N  \partial_\mu \sigma^{(u,v)}_{\kappa_n(s)} \big( X_{\kappa_n(s)}^{i,N,n}, \mu_{\kappa_{n}(s)}^{X,N,n},  X_{\kappa_n(s)}^{j,N,n} \big)  \int_{\kappa_n(s)}^s \sigma^{0,n}_{\kappa_n(s)} \big(X_{\kappa_n(r)}^{j,N,n}, \mu_{\kappa_n(r)}^{X,N,n}\big) dW_r^{0} \Big|^{p}  \leq K n^{-p/2} \notag,
\end{align*}
for any $s\in[0,T]$ and $n, N\in\mathbb{R}^d$. By using the above estimates, one obtains,
\begin{align*}
E \big|\Gamma_{\kappa_n(s)}^{n, \sigma} & \big(s, X_{\kappa_n(s)}^{i,N,n}, \mu_{\kappa_n(s)}^{X,N,n}\big)\big|^{p}  \leq  K E\big|\Lambda_{\kappa_{n}(s)}^{n, \sigma \sigma} \big(s,X_{\kappa_n(s)}^{i,N,n}, \mu_{\kappa_n(s)}^{X,N,n}\big)\big|^{p} 
\\
&  +K E\big|\Lambda_{\kappa_{n}(s)}^{n, \sigma \sigma^0} \big(s,X_{\kappa_n(s)}^{i,N,n}, \mu_{\kappa_n(s)}^{X,N,n}\big) \big|^{p}  +K E \big|\bar{\Lambda}^{n,\sigma \sigma}_{\kappa_n(s)} \big( s,X_{\kappa_n(s)}^{i,N,n}, \mu_{\kappa_n(s)}^{X,N,n}\big) \big|^{p} 
\\
& \qquad + K E\big| \bar{\Lambda}^{n,\sigma \sigma^0}_{\kappa_n(s)} \big( s,X_{\kappa_n(s)}^{i,N,n}, \mu_{\kappa_n(s)}^{X,N,n}\big) \big|^{p} \leq K n^{-p/2} 
\end{align*}
for any $s\in[0,T]$ and $n, N\in\mathbb{R}^d$. This completes the proof for the first inequality. The second inequality follows similarly.  
\end{proof}
As a consequence of Remark \ref{rem:growth:taming:milstein}, Lemmas \ref{lem:mb:milstein} and \ref{lem:gamma:rate}, one obtains the following corollary. 
\begin{cor} \label{cor:sigma:rate}
Let Assumptions \ref{as:x0}, \ref{as:coercivity}, \ref{as:polynomial:Lipschitz}, \ref{as:der:x:poly:lip} and  \ref{as:der:mea:poly:lip} be satisfied. Then,  
\begin{align*}
E \big|\tilde{\sigma}^n_{\kappa_n(s)} \big(s,X_{\kappa_n(s)}^{i,N,n}, \mu_{\kappa_n(s)}^{X,N,n}\big)\big|^{p} &\leq K, 
\\
E \big|\tilde{\sigma}^{0,n}_{\kappa_n(s)} \big(s,X_{\kappa_n(s)}^{i,N,n}, \mu_{\kappa_n(s)}^{X,N,n}\big)\big|^{p} & \leq K, 
\end{align*}
for any $p\leq p_0/(\rho/2+1)$,  $s\in[0,T]$, $i\in\{1,\ldots,N\}$, $n, N\in\mathbb{N}$ where the constant $K>0$ does not depend on $n$ and $N$. 
\end{cor}
\begin{lem} \label{lem:one-step:rate:milstein}
Let Assumptions \ref{as:x0}, \ref{as:coercivity}, \ref{as:polynomial:Lipschitz}, \ref{as:der:x:poly:lip} and  \ref{as:der:mea:poly:lip} be satisfied. Then, 
\begin{align*}
E\big| X_{s}^{i,N,n}-X_{\kappa_n(s)}^{i,N,n}\big|^{p} \leq K n^{-p/2},
\end{align*}
for any  $p\leq p_0/(\rho/2+1)$,  $s\in[0,T]$ and $n,N\in\mathbb{N}$ where the constant $K>0$ does not depend on $n$ and $N$. 
\end{lem}
\begin{proof}
The proof follows by adapting arguments similar to the proof of Lemma \ref{Lemma:onestepMilstein} along with Remark \ref{rem:poly:growth}, Lemma \ref{lem:mb:milstein} and Corollary \ref{cor:sigma:rate}. 
\end{proof}

We reproduce for the reader's convenience the following lemma from \cite{kumar2020}, which is used later.  
\begin{lem}[Lemma $7$ of \cite{kumar2020}] \label{lem:f:rate:local}
Let $f:\mathbb{R}^d\times \mathcal{P}_2(\mathbb{R}^d) \to \mathbb{R}$ be a  function such that its derivative $\partial_x f: \mathbb{R}^d\times \mathcal{P}_2(\mathbb{R}^d)\to\mathbb{R}^d$ and measure derivative $\partial_\mu f:\mathbb{R}^d\times\mathcal{P}_2(\mathbb{R}^d)\times\mathbb{R}^d \to \mathbb{R}^d$ satisfy polynomial Lipschitz condition \textit{i.e.}, there exist constants  $L>0$ and $\chi\geq 0$ such that 
 \begin{align*}
| \partial_x f(x,\mu)-\partial_x f(\bar{x},\bar{\mu})| & \leq L\big\{ (1+|x|+|\bar{x}|)^{\chi}|x-\bar{x}| + \mathcal{W}_2(\mu,\bar{\mu}) \big\},
\\
| \partial_\mu f(x,\mu, y)-\partial_\mu f(\bar{x}, \bar{\mu}, \bar{y})| & \leq L\big\{ (1+|x|+|\bar{x}|)^{\xi+1}|x-\bar{x}|+ \mathcal{W}_2(\mu,\bar{\mu}) +  |y-\bar{y}| \big\},
 \end{align*}
  for all $x, y,\bar{x}, \bar{y}\in\mathbb{R}^d$ and $\mu,\bar{\mu} \in\mathcal{P}_2(\mathbb{R}^d)$.
Then, 
\begin{align*}
\Big| f\Big(  x^i,\frac{1}{N} & \sum_{j=1}^N\delta_{x^j}\Big)  -f\Big(\bar{x}^i,\frac{1}{N}\sum_{j=1}^N\delta_{\bar{x}^j}\Big)- \partial_x f\Big(\bar{x}^i,\frac{1}{N}\sum_{j=1}^N\delta_{\bar{x}^j}\Big) \big( x^i-\bar{x}^i \big ) 
\\
&\qquad- \frac{1}{N} \sum_{j=1}^N  \partial_\mu f\Big(\bar{x}^i,\frac{1}{N}\sum_{j=1}^N \delta_{\bar{x}^j}, \bar{x}^j\Big) \big( x^j-\bar{x}^j \big)  \Big|
\\
&\leq  K (1+|x^i|+|\bar{x}^i|)^{\chi} |x^i-\bar{x}^i|^2 + K \frac{1}{N}\sum_{j=1}^N|x^j-\bar{x}^j|^2,
\end{align*}
for every $i \in\{1,\ldots,N\}$ where the constant $K>0$ does not depend on $N\in\mathbb{N}$.
\end{lem}

\begin{lem} \label{lem:sigma-sigma:milstein}
Let Assumptions \ref{as:x0}, \ref{as:coercivity}, \ref{as:monotonicity:rate}, \ref{as:polynomial:Lipschitz}, \ref{as:lipschitz}, \ref{as:der:x:poly:lip} and \ref{as:der:mea:poly:lip} be satisfied. Then, 
\begin{align*}
E|\sigma_{s}\big(X_s^{i,N,n},\mu_s^{X,N,n}\big)-\tilde{\sigma}_{\kappa_n(s)}^n\big(s,X_{\kappa_n(s)}^{i,N,n},\mu_{\kappa_n(s)}^{X,N,n}\big)\big|^p & \leq K n^{-p},
\\
E\big|\sigma^{0}_{s}\big(X_s^{i,N,n},\mu_s^{X,N,n}\big)-\tilde{\sigma}_{\kappa_n(s)}^{0,n}\big(s,X_{\kappa_n(s)}^{i,N,n},\mu_{\kappa_n(s)}^{X,N,n}\big)\big|^p & \leq K n^{-p},
\end{align*}
for any $p\leq p_0/(2\rho+4)$, $s\in[0,T]$, $i\in\{1,\ldots,N\}$  and $n,N\in\mathbb{N}$ where the constant $K>0$ does not depend on $n$ and $N$. 
\end{lem}
\begin{proof} 
From equation \eqref{eq:milstein}, 
\begin{align} 
  \partial_x \sigma^{(u,v)}_{\kappa_n(s)}& \big(X_{\kappa_n(s)}^{i,N,n},\mu_{\kappa_n(s)}^{X,N,n}\big)\big( X_{s}^{i,N,n}-X_{\kappa_n(s)}^{i,N,n} \big) \notag
 \\
 =&  \partial_x \sigma^{(u,v)}_{\kappa_n(s)} \big(X_{\kappa_n(s)}^{i,N,n},\mu_{\kappa_n(s)}^{X,N,n}\big) \int_{\kappa_n(s)}^s b^n_{\kappa_n(r)}\big(X_{\kappa_n(r)}^{i,N, n}, \mu_{\kappa_n(r)}^{X,N,n}\big) dr  \notag
 \\
 & + \Lambda_{\kappa_n(s)}^{n, \sigma \sigma, (u, v)} \big(s,X_{\kappa_n(s)}^{i,N,n}, \mu_{\kappa_n(s)}^{X,N,n}\big) + \Lambda_{\kappa_n(s)}^{n, \sigma \sigma^0, (u, v)} \big(s,X_{\kappa_n(s)}^{i,N,n}, \mu_{\kappa_n(s)}^{X,N,n}\big)  \notag
 \\
 & +  \partial_x \sigma^{(u,v)}_{\kappa_n(s)} (X_{\kappa_n(s)}^{i,N,n},\mu_{\kappa_n(s)}^{X,N,n}) \int_{\kappa_n(s)}^s \Gamma_{\kappa_{n}(r)}^{n, \sigma} \big(r,X_{\kappa_n(r)}^{i,N,n}, \mu_{\kappa_n(r)}^{X,N,n}\big) dW_r^{i}   \notag
\\
& +  \partial_x \sigma^{(u,v)}_{\kappa_n(s)} (X_{\kappa_n(s)}^{i,N,n},\mu_{\kappa_n(s)}^{X,N,n}) \int_{\kappa_n(s)}^s \Gamma_{\kappa_{n}(r)}^{n, \sigma^0} \big(r,X_{\kappa_n(r)}^{i,N,n}, \mu_{\kappa_n(r)}^{X,N,n}\big)  dW^0_r \label{eq:sigma-sigma:1}
\end{align}
and also, 
\begin{align} 
 \frac{1}{N} \sum_{j=1}^N  & \partial_\mu \sigma^{(u,v)}_{\kappa_n(s)} \big(X_{\kappa_n(s)}^{i,N,n},\mu_{\kappa_n(s)}^{X,N,n}, X_{\kappa_n(s)}^{j,N,n} \big) \big(X_{s}^{j,N,n}-X_{\kappa_n(s)}^{j,N,n} \big) \notag
\\
 = & \frac{1}{N} \sum_{j=1}^N  \partial_\mu \sigma^{(u,v)}_{\kappa_n(s)} \big(X_{\kappa_n(s)}^{i,N,n},\mu_{\kappa_n(s)}^{X,N,n}, X_{\kappa_n(s)}^{j,N,n} \big) \int_{\kappa_n(s)}^s b^n_{\kappa_n(r)}\big(X_{\kappa_n(r)}^{j,N, n}, \mu_{\kappa_n(r)}^{X,N,n}\big) dr   \notag 
\\
& +  \bar{\Lambda}^{n,\sigma \sigma, (u, v)}_{\kappa_n(s)} \big(s,X_{\kappa_n(s)}^{i,N,n}, \mu_{\kappa_n(s)}^{X,N,n}\big) +  \bar{\Lambda}^{n,\sigma \sigma^0, (u, v)}_{\kappa_n(s)} \big(s,X_{\kappa_n(s)}^{i,N,n}, \mu_{\kappa_n(s)}^{X,N,n}\big) \notag 
\\
& +  \frac{1}{N} \sum_{j=1}^N  \partial_\mu \sigma^{(u,v)}_{\kappa_n(s)} \big(X_{\kappa_n(s)}^{i,N,n},\mu_{\kappa_n(s)}^{X,N,n}, X_{\kappa_n(s)}^{j,N,n} \big) \int_{\kappa_n(s)}^s \Gamma_{\kappa_n(r)}^{n, \sigma} \big(r, X_{\kappa_n(r)}^{j,N,n}, \mu_{\kappa_n(r)}^{X,N,n}\big) dW_r^{j}  \notag
\\
& + \frac{1}{N} \sum_{j=1}^N  \partial_\mu \sigma^{(u,v)}_{\kappa_n(s)} \big(X_{\kappa_n(s)}^{i,N,n},\mu_{\kappa_n(s)}^{X,N,n}, X_{\kappa_n(s)}^{j,N,n} \big)   \int_{\kappa_n(s)}^s \Gamma_{\kappa_n(r)}^{n, \sigma^0} \big(r, X_{\kappa_n(r)}^{j,N,n}, \mu_{\kappa_n(r)}^{X,N,n}\big)  dW_r^{0}  
 \label{eq:sigma-sigma:2}
\end{align}
almost surely for any $s\in[0,T]$, $i\in\{1,\ldots, N\}$ and $n,N\in\mathbb{N}$. Furthermore, 
\begin{align*} 
\sigma^{(u, v)}_{\kappa_n(s)}&\big(X_s^{i,N,n}, \mu_s^{X,N,n}\big)-\tilde{\sigma}^{n,(u,v)}_{\kappa_n(s)}\big(s,X_{\kappa_n(s)}^{i,N,n},\mu_{\kappa_n(s)}^{X,N,n}\big) = \sigma^{(u, v)}_{\kappa_n(s)}\big(X_s^{i,N,n},\mu_s^{X,N,n}\big)  
\\
 & -\sigma^{n,(u, v)}_{\kappa_n(s)}\big(X_{\kappa_n(s)}^{i,N,n},\mu_{\kappa_n(s)}^{X,N,n}\big) - \Lambda_{\kappa_{n}(s)}^{n, \sigma \sigma, (u,v)} \big(s,X_{\kappa_n(s)}^{i,N,n}, \mu_{\kappa_n(s)}^{X,N,n}\big) - \Lambda_{\kappa_{n}(s)}^{n, \sigma \sigma^0, (u,v)} \big(s,X_{\kappa_n(s)}^{i,N,n}, \mu_{\kappa_n(s)}^{X,N,n}\big)   
\\
&  - \bar{\Lambda}^{n,\sigma \sigma, (u,v)}_{\kappa_n(s)} \big( s,X_{\kappa_n(s)}^{i,N,n}, \mu_{\kappa_n(s)}^{X,N,n}\big)  - \bar{\Lambda}^{n,\sigma \sigma^0, (u,v)}_{\kappa_n(s)} \big( s,X_{\kappa_n(s)}^{i,N,n}, \mu_{\kappa_n(s)}^{X,N,n}\big)  
\\
& -   \partial_x \sigma^{(u,v)}_{\kappa_n(s)} (X_{\kappa_n(s)}^{i,N,n},\mu_{\kappa_n(s)}^{X,N,n}) \big( X_{s}^{i,N,n}-X_{\kappa_n(s)}^{i,N,n} \big) 
\\
&  - \frac{1}{N} \sum_{j=1}^N  \partial_\mu \sigma^{(u,v)}_{\kappa_n(s)} \big(X_{\kappa_n(s)}^{i,N,n},\mu_{\kappa_n(s)}^{X,N,n}, X_{\kappa_n(s)}^{j,N,n} \big)\big( X_{s}^{j,N,n}-X_{\kappa_n(s)}^{j,N,n} \big) 
\\
&  +  \partial_x \sigma^{(u,v)}_{\kappa_n(s)} \big(X_{\kappa_n(s)}^{i,N,n},\mu_{\kappa_n(s)}^{X,N,n}\big) \big( X_{s}^{i,N,n}-X_{\kappa_n(s)}^{i,N,n} \big) 
\\
&  + \frac{1}{N} \sum_{j=1}^N  \partial_\mu \sigma^{(u,v)}_{\kappa_n(s)} \big(X_{\kappa_n(s)}^{i,N,n},\mu_{\kappa_n(s)}^{X,N,n}, X_{\kappa_n(s)}^{j,N,n} \big) \big(X_{s}^{j,N,n}-X_{\kappa_n(s)}^{j,N,n} \big)   
\end{align*} 
which on using equations \eqref{eq:sigma-sigma:1} and \eqref{eq:sigma-sigma:2} yields, 
\begin{align*}
\sigma^{(u, v)}_{\kappa_n(s)} & \big(X_s^{i,N,n}, \mu_s^{X,N,n}\big)-\tilde{\sigma}^{n,(u,v)}_{\kappa_n(s)}\big(s,X_{\kappa_n(s)}^{i,N,n},\mu_{\kappa_n(s)}^{X,N,n}\big) = \sigma^{(u, v)}_{\kappa_n(s)}\big(X_s^{i,N,n},\mu_s^{X,N,n}\big)
\\
& -\sigma^{(u, v)}_{\kappa_n(s)}\big(X_{\kappa_n(s)}^{i,N,n},\mu_{\kappa_n(s)}^{X,N,n}\big) -   \partial_x \sigma^{(u,v)}_{\kappa_n(s)} (X_{\kappa_n(s)}^{i,N,n},\mu_{\kappa_n(s)}^{X,N,n}) \big( X_{s}^{i,N,n}-X_{\kappa_n(s)}^{i,N,n} \big) 
\\
& - \frac{1}{N} \sum_{j=1}^N  \partial_\mu \sigma^{(u,v)}_{\kappa_n(s)} \big(X_{\kappa_n(s)}^{i,N,n},\mu_{\kappa_n(s)}^{X,N,n}, X_{\kappa_n(s)}^{j,N,n} \big)\big( X_{s}^{j,N,n}-X_{\kappa_n(s)}^{j,N,n} \big) 
\\
&+ \sigma^{(u, v)}_{\kappa_n(s)}\big(X_{\kappa_n(s)}^{i,N,n},\mu_{\kappa_n(s)}^{X,N,n}\big)-\sigma^{n,(u, v)}_{\kappa_n(s)}\big(X_{\kappa_n(s)}^{i,N,n},\mu_{\kappa_n(s)}^{X,N,n}\big)
\\
& + \partial_x \sigma^{(u,v)}_{\kappa_n(s)} \big(X_{\kappa_n(s)}^{i,N,n},\mu_{\kappa_n(s)}^{X,N,n}\big) \int_{\kappa_n(s)}^s b^n_{\kappa_n(r)}\big(X_{\kappa_n(r)}^{i,N, n}, \mu_{\kappa_n(r)}^{X,N,n}\big) dr 
\\
 & +  \partial_x \sigma^{(u,v)}_{\kappa_n(s)} (X_{\kappa_n(s)}^{i,N,n},\mu_{\kappa_n(s)}^{X,N,n}) \int_{\kappa_n(s)}^s \Gamma_{\kappa_{n}(r)}^{n, \sigma} \big(r,X_{\kappa_n(r)}^{i,N,n}, \mu_{\kappa_n(r)}^{X,N,n}\big) dW_r^{i}   
\\
& +  \partial_x \sigma^{(u,v)}_{\kappa_n(s)} (X_{\kappa_n(s)}^{i,N,n},\mu_{\kappa_n(s)}^{X,N,n}) \int_{\kappa_n(s)}^s \Gamma_{\kappa_{n}(r)}^{n, \sigma^0} \big(r,X_{\kappa_n(r)}^{i,N,n}, \mu_{\kappa_n(r)}^{X,N,n}\big)  dW^0_r
\\
& +\frac{1}{N} \sum_{j=1}^N  \partial_\mu \sigma^{(u,v)}_{\kappa_n(s)} \big(X_{\kappa_n(s)}^{i,N,n},\mu_{\kappa_n(s)}^{X,N,n}, X_{\kappa_n(s)}^{j,N,n} \big) \int_{\kappa_n(s)}^s b^n_{\kappa_n(r)}\big(X_{\kappa_n(r)}^{j,N, n}, \mu_{\kappa_n(r)}^{X,N,n}\big) dr 
\\
& +  \frac{1}{N} \sum_{j=1}^N  \partial_\mu \sigma^{(u,v)}_{\kappa_n(s)} \big(X_{\kappa_n(s)}^{i,N,n},\mu_{\kappa_n(s)}^{X,N,n}, X_{\kappa_n(s)}^{j,N,n} \big) \int_{\kappa_n(s)}^s \Gamma_{\kappa_n(r)}^{n, \sigma} \big(r, X_{\kappa_n(r)}^{j,N,n}, \mu_{\kappa_n(r)}^{X,N,n}\big) dW_r^{j}  
\\
& + \frac{1}{N} \sum_{j=1}^N  \partial_\mu \sigma^{(u,v)}_{\kappa_n(s)} \big(X_{\kappa_n(s)}^{i,N,n},\mu_{\kappa_n(s)}^{X,N,n}, X_{\kappa_n(s)}^{j,N,n} \big)   \int_{\kappa_n(s)}^s \Gamma_{\kappa_n(r)}^{n, \sigma^0} \big(r, X_{\kappa_n(r)}^{j,N,n}, \mu_{\kappa_n(r)}^{X,N,n}\big)  dW_r^{0}  
\end{align*}
almost surely for any $s\in[0,T]$, $i\in\{1,\ldots, N\}$ and $n,N\in \mathbb{N}$. On using Lemma \ref{lem:f:rate:local} with  $f=\sigma_{\kappa_n(s)}$ and Remark \ref{rem:poly:lipschitz}, one obtains, 
\begin{align*}
E\big|\sigma^{(u, v)}_{\kappa_n(s)} & \big(X_s^{i,N,n}, \mu_s^{X,N,n}\big)-\tilde{\sigma}^{n,(u,v)}_{\kappa_n(s)}\big(s,X_{\kappa_n(s)}^{i,N,n},\mu_{\kappa_n(s)}^{X,N,n}\big) \big|^p 
\\
 \leq & K E\big(1+\big|X_s^{i,N,n}\big|+ \big|X_{\kappa_n(s)}^{i,N,n}\big|\big)^{\rho p/4} \big|X_s^{i,N,n}-X_{\kappa_n(s)}^{i,N,n}\big|^{2p} +K \frac{1}{N} \sum_{j=1}^N E\big|X_s^{j,N,n}-X_{\kappa_n(s)}^{j,N,n}\big|^{2p}
\\
& + K E\big| \sigma^{(u, v)}_{\kappa_n(s)}\big(X_{\kappa_n(s)}^{i,N,n},\mu_{\kappa_n(s)}^{X,N,n}\big)-\sigma^{n,(u, v)}_{\kappa_n(s)}\big(X_{\kappa_n(s)}^{i,N,n},\mu_{\kappa_n(s)}^{X,N,n}\big)\big|^p 
\\
& + K E \Big|\partial_x \sigma^{(u,v)}_{\kappa_n(s)} \big(X_{\kappa_n(s)}^{i,N,n},\mu_{\kappa_n(s)}^{X,N,n}\big) \int_{\kappa_n(s)}^s b^n_{\kappa_n(r)}\big(X_{\kappa_n(r)}^{i,N, n}, \mu_{\kappa_n(r)}^{X,N,n}\big) dr\Big|^p  
\\
 & + K E \Big| \partial_x \sigma^{(u,v)}_{\kappa_n(s)} (X_{\kappa_n(s)}^{i,N,n},\mu_{\kappa_n(s)}^{X,N,n}) \int_{\kappa_n(s)}^s \Gamma_{\kappa_{n}(r)}^{n, \sigma} \big(r,X_{\kappa_n(r)}^{i,N,n}, \mu_{\kappa_n(r)}^{X,N,n}\big) dW_r^{i}  \Big|^p  
\\
& +  K E \Big|\partial_x \sigma^{(u,v)}_{\kappa_n(s)} (X_{\kappa_n(s)}^{i,N,n},\mu_{\kappa_n(s)}^{X,N,n}) \int_{\kappa_n(s)}^s \Gamma_{\kappa_{n}(r)}^{n, \sigma^0} \big(r,X_{\kappa_n(r)}^{i,N,n}, \mu_{\kappa_n(r)}^{X,N,n}\big)  dW^0_r\Big|^p
\\
& + K E \Big|\frac{1}{N} \sum_{j=1}^N  \partial_\mu \sigma^{(u,v)}_{\kappa_n(s)} \big(X_{\kappa_n(s)}^{i,N,n},\mu_{\kappa_n(s)}^{X,N,n}, X_{\kappa_n(s)}^{j,N,n} \big) \int_{\kappa_n(s)}^s b^n_{\kappa_n(r)}\big(X_{\kappa_n(r)}^{j,N, n}, \mu_{\kappa_n(r)}^{X,N,n}\big) dr \Big|^p
\\
& +  K E \Big|\frac{1}{N} \sum_{j=1}^N  \partial_\mu \sigma^{(u,v)}_{\kappa_n(s)} \big(X_{\kappa_n(s)}^{i,N,n},\mu_{\kappa_n(s)}^{X,N,n}, X_{\kappa_n(s)}^{j,N,n} \big) \int_{\kappa_n(s)}^s \Gamma_{\kappa_n(r)}^{n, \sigma} \big(r, X_{\kappa_n(r)}^{j,N,n}, \mu_{\kappa_n(r)}^{X,N,n}\big) dW_r^{j} \Big|^p  
\\
& + K E \Big|\frac{1}{N} \sum_{j=1}^N  \partial_\mu \sigma^{(u,v)}_{\kappa_n(s)} \big(X_{\kappa_n(s)}^{i,N,n},\mu_{\kappa_n(s)}^{X,N,n}, X_{\kappa_n(s)}^{j,N,n} \big)   \int_{\kappa_n(s)}^s \Gamma_{\kappa_n(r)}^{n, \sigma^0} \big(r, X_{\kappa_n(r)}^{j,N,n}, \mu_{\kappa_n(r)}^{X,N,n}\big)  dW_r^{0}  \Big|^p
\end{align*}
for any $s\in[0,T]$, $i\in\{1,\ldots, N\}$ and $n, N\in\mathbb{N}$. Notice that due to equation \eqref{eq:Taming:milstein}, Remark \ref{rem:poly:growth}, equation \eqref{eq:w2:mb} and Lemma \ref{lem:mb:milstein}, 
\begin{align}
E\Big|\sigma_{\kappa_n(s)}^{(u,v)} \big(X_{\kappa_n(s)}^{i,N,n}, \mu_{\kappa_n(s)}^{X,N,n}\big)- \frac{\sigma_{\kappa_n(s)}^{(u,v)} \big(X_{\kappa_n(s)}^{i,N,n}, \mu_{\kappa_n(s)}^{X,N,n}\big)}{1+n^{-1}\big|X_{\kappa_n(s)}^{i,N,n}\big|^{\rho}} \Big|^p& \leq K n^{-p} E \big|\sigma_{\kappa_n(s)}^{(u,v)} \big(X_{\kappa_n(s)}^{i,N,n}, \mu_{\kappa_n(s)}^{X,N,n}\big)\big|^p \big|X_{\kappa_n(s)}^{i,N,n}\big|^{\rho p} \notag
\\
& \leq K n^{-p} \label{eq:sig-sign}
\end{align}
for any $s\in[0,T]$, $i\in\{1,\ldots, N\}$ and $n, N\in\mathbb{N}$.  Furthermore, one uses H\"older's inequality, equation \eqref{eq:sig-sign}, Lemmas \ref{lem:one-step:rate:milstein}, \ref{lem:mb:milstein} and Remarks \ref{rem:poly:growth}, \ref{rem:der:growth:poly} to obtain, 
\begin{align*}
E\big|\sigma^{(u, v)}_{\kappa_n(s)} & \big(X_s^{i,N,n}, \mu_s^{X,N,n}\big)-\tilde{\sigma}^{n,(u,v)}_{\kappa_n(s)}\big(s,X_{\kappa_n(s)}^{i,N,n},\mu_{\kappa_n(s)}^{X,N,n}\big) \big|^p 
\\
 \leq & K \big\{E\big(1+\big|X_s^{i,N,n}\big|+ \big|X_{\kappa_n(s)}^{i,N,n}\big|\big)^{\rho p/2} \big\}^{1/2} \big\{E\big|X_s^{i,N,n}-X_{\kappa_n(s)}^{i,N,n}\big|^{4p}\big\}^{1/2} +K n^{-p}
\\
& + K n^{-p} E \big|\partial_x \sigma^{(u,v)}_{\kappa_n(s)} \big(X_{\kappa_n(s)}^{i,N,n},\mu_{\kappa_n(s)}^{X,N,n}\big)\big|^p \big| b^n_{\kappa_n(s)}\big(X_{\kappa_n(s)}^{i,N, n}, \mu_{\kappa_n(s)}^{X,N,n}\big) \big|^p  
\\
 & + K n^{-p/2+1}E  \int_{\kappa_n(s)}^s \big| \partial_x \sigma^{(u,v)}_{\kappa_n(s)} (X_{\kappa_n(s)}^{i,N,n},\mu_{\kappa_n(s)}^{X,N,n}) \big|^p \big| \Gamma_{\kappa_{n}(r)}^{n, \sigma} \big(r,X_{\kappa_n(r)}^{i,N,n}, \mu_{\kappa_n(r)}^{X,N,n}\big)\big|^p dr  
\\
& +  K n^{-p/2+1} E  \int_{\kappa_n(s)}^s  \big|\partial_x \sigma^{(u,v)}_{\kappa_n(s)} (X_{\kappa_n(s)}^{i,N,n},\mu_{\kappa_n(s)}^{X,N,n}) \big|^p \big| \Gamma_{\kappa_{n}(r)}^{n, \sigma^0} \big(r,X_{\kappa_n(r)}^{i,N,n}, \mu_{\kappa_n(r)}^{X,N,n}\big)\big|^p  dr 
\\
& + K n^{-p} \frac{1}{N} \sum_{j=1}^N   E \big|\partial_\mu \sigma^{(u,v)}_{\kappa_n(s)} \big(X_{\kappa_n(s)}^{i,N,n},\mu_{\kappa_n(s)}^{X,N,n}, X_{\kappa_n(s)}^{j,N,n} \big)\big|^p \big|b^n_{\kappa_n(s)}\big(X_{\kappa_n(s)}^{j,N, n}, \mu_{\kappa_n(s)}^{X,N,n}\big)\big|^p 
\\
& +  K n^{-p/2+1}  \frac{1}{N} \sum_{j=1}^N E \int_{\kappa_n(s)}^s \big| \partial_\mu \sigma^{(u,v)}_{\kappa_n(s)} \big(X_{\kappa_n(s)}^{i,N,n},\mu_{\kappa_n(s)}^{X,N,n}, X_{\kappa_n(s)}^{j,N,n} \big) \big|^p \big|\Gamma_{\kappa_n(r)}^{n, \sigma} \big(r, X_{\kappa_n(r)}^{j,N,n}, \mu_{\kappa_n(r)}^{X,N,n}\big) \big|^p dr  
\\
& + Kn^{-p/2+1}  \frac{1}{N} \sum_{j=1}^N  E    \int_{\kappa_n(s)}^s \big| \partial_\mu \sigma^{(u,v)}_{\kappa_n(s)} \big(X_{\kappa_n(s)}^{i,N,n},\mu_{\kappa_n(s)}^{X,N,n}, X_{\kappa_n(s)}^{j,N,n} \big) \big|^p \big|\Gamma_{\kappa_n(r)}^{n, \sigma^0} \big(r, X_{\kappa_n(r)}^{j,N,n}, \mu_{\kappa_n(r)}^{X,N,n}\big)\big|^p   dr 
\\
 \leq & K n^{-p}+ K  n^{-p/2+1}  \int_{\kappa_n(s)}^s \big\{E \big| \partial_x \sigma^{(u,v)}_{\kappa_n(s)} (X_{\kappa_n(s)}^{i,N,n},\mu_{\kappa_n(s)}^{X,N,n}) \big|^{2p}\big\}^{1/2} \big\{E\big| \Gamma_{\kappa_{n}(r)}^{n, \sigma} \big(r,X_{\kappa_n(r)}^{i,N,n}, \mu_{\kappa_n(r)}^{X,N,n}\big)\big|^{2p}\big\}^{1/2} dr
 \\
 &+ K  n^{-p/2+1}  \int_{\kappa_n(s)}^s \big\{ E\big|\partial_x \sigma^{(u,v)}_{\kappa_n(s)} (X_{\kappa_n(s)}^{i,N,n},\mu_{\kappa_n(s)}^{X,N,n}) \big|^{2p} \big\}^{1/2} \big\{E\big| \Gamma_{\kappa_{n}(r)}^{n, \sigma^0} \big(r,X_{\kappa_n(r)}^{i,N,n}, \mu_{\kappa_n(r)}^{X,N,n}\big)\big|^{2p}\big\}^{1/2}  dr   
 \\
 & +  K n^{-p/2+1}  \frac{1}{N} \sum_{j=1}^N  \int_{\kappa_n(s)}^s  \big\{E \big| \partial_\mu \sigma^{(u,v)}_{\kappa_n(s)} \big(X_{\kappa_n(s)}^{i,N,n},\mu_{\kappa_n(s)}^{X,N,n}, X_{\kappa_n(s)}^{j,N,n} \big) \big|^{2p}\big\}^{1/2} 
 \\
 & \qquad \qquad \qquad \times \big\{ E\big|\Gamma_{\kappa_n(r)}^{n, \sigma} \big(r, X_{\kappa_n(r)}^{j,N,n}, \mu_{\kappa_n(r)}^{X,N,n}\big) \big|^{2p} \big\}^\frac{1}{2} dr  
\\
& + K n^{-p/2+1} \frac{1}{N} \sum_{j=1}^N      \int_{\kappa_n(s)}^s \big\{E\big| \partial_\mu \sigma^{(u,v)}_{\kappa_n(s)} \big(X_{\kappa_n(s)}^{i,N,n},\mu_{\kappa_n(s)}^{X,N,n}, X_{\kappa_n(s)}^{j,N,n} \big) \big|^{2p} \big\}^\frac{1}{2} 
\\
& \qquad \qquad \qquad \times \big\{E\big|\Gamma_{\kappa_n(r)}^{n, \sigma^0} \big(r, X_{\kappa_n(r)}^{j,N,n}, \mu_{\kappa_n(r)}^{X,N,n}\big)\big|^{2p} \big\}^\frac{1}{2}   dr 
\end{align*}
for any $s\in[0,T]$, $i\in\{1,\ldots, N\}$ and $n, N\in\mathbb{N}$.  On using  Lemma \ref{lem:gamma:rate}, Remarks \ref{rem:der:growth:poly} and equation \eqref{eq:w2:mb},  one obtains
\begin{align*}
E\big|\sigma_{\kappa_n(s)} & \big(X_s^{i,N,n}, \mu_s^{X,N,n}\big)-\tilde{\sigma}^{n}_{\kappa_n(s)}\big(s,X_{\kappa_n(s)}^{i,N,n},\mu_{\kappa_n(s)}^{X,N,n}\big) \big|^p \leq K n^{-p},
\end{align*}
for any $s\in[0,T]$, $i\in\{1,\ldots, N\}$ and $n, N\in\mathbb{N}$. Furthermore, the application of Assumption \ref{as:lipschitz} yields, 
\begin{align*}
 K E\big|\sigma_{s}  \big(X_s^{i,N,n}, \mu_s^{X,N,n}\big)-\sigma_{\kappa_n(s)}  \big(X_s^{i,N,n}, \mu_s^{X,N,n}\big) \big|^p \leq  K n^{-p},
\end{align*} 
for any $s\in[0,T]$, $i\in\{1,\ldots, N\}$ and $n, N\in\mathbb{N}$. Also, similar calculations give the second inequality.  
\end{proof}
\begin{lem} \label{lem:b-b}
Let Assumptions \ref{as:x0}, \ref{as:coercivity}, \ref{as:monotonicity:rate}, \ref{as:polynomial:Lipschitz}, \ref{as:lipschitz}, \ref{as:der:x:poly:lip} and \ref{as:der:mea:poly:lip} be satisfied. Then, 
\begin{align*}
E\big|b_s(X_s^{i,N,n},\mu_s^{X,N,n})-b_{\kappa_n(s)}^n(X_{\kappa_n(s)}^{i,N,n},\mu_{\kappa_n(s)}^{X,N,n})|^p \leq K n^{-p/2},
\end{align*}
for any $p\leq p_0/(2\rho+4)$, $s\in[0,T]$, $i\in\{1,\ldots,N\}$ and $n,N\in\mathbb{N}$ where  constant $K>0$ does not depend on $n$ and $N$.  
\end{lem}
\begin{proof}
On using Assumptions \ref{as:polynomial:Lipschitz}, \ref{as:lipschitz}, equation \eqref{eq:Taming:milstein}, H\"older's inequality, equation \eqref{eq:w2:scheme} and Remark \ref{rem:poly:growth}, one obtains
\begin{align*}
E\big|b_{s}\big( X_s^{i,N,n},& \mu_s^{X,N,n}\big)  -b^n_{\kappa_n(s)}\big(X_{\kappa_n(s)}^{i,N,n},\mu_{\kappa_n(s)}^{X,N,n}\big)\big|^p \leq K E\big|b_{s}\big( X_s^{i,N,n}, \mu_s^{X,N,n}\big)-b_{\kappa_n(s)}\big(X_s^{i,N,n},\mu_s^{X,N,n}\big)\big|^p
\\
& + KE \big|b_{\kappa_n(s)}\big(X_s^{i,N,n},\mu_s^{X,N,n}\big)-b_{\kappa_n(s)}\big(X_{\kappa_n(s)}^{i,N,n},\mu_{\kappa_n(s)}^{X,N,n}\big)\big|^p
\\
& + K E\big|b_{\kappa_n(s)}\big(X_{\kappa_n(s)}^{i,N,n},\mu_{\kappa_n(s)}^{X,N,n}\big)-b_{\kappa_n(s)}^n\big(X_{\kappa_n(s)}^{i,N,n},\mu_{\kappa_n(s)}^{X,N,n}\big)\big|^p
\\
\leq & K n^{-p}+ K E\big(1+\big| X_s^{i,N,n} \big|+\big|X_{\kappa_n(s)}^{i,N,n} \big|\big)^{\rho p/2} \big|X_s^{i,N,n}- X_{\kappa_n(s)}^{i,N,n}\big|^p + K E\mathcal{W}_2^p\big(\mu_s^{X,N,n}, \mu_{\kappa_n(s)}^{X,N,n}\big)
\\
&+K  E\Big|b_{\kappa_n(s)}\big(X_{\kappa_n(s)}^{i,N,n},\mu_{\kappa_n(s)}^{X,N,n}\big)-\frac{b_{\kappa_n(s)} \big(X_{\kappa_n(s)}^{i,N, n}, \mu_{\kappa_n(s)}^{X,N,n}\big)}{1+ n^{-1} \big|X_{\kappa_n(s)}^{i,N,n}\big|^{\rho}} \Big|^p
\\
\leq & K n^{-p}+ K \big\{E\big(1+\big| X_s^{i,N,n} \big|+\big|X_{\kappa_n(s)}^{i,N,n} \big|\big)^{\rho p} E\big|X_s^{i,N,n}- X_{\kappa_n(s)}^{i,N,n}\big|^{2p}\big\}^{1/2} 
\\
& + \frac{1}{N}\sum_{j=1}^n E\big|X_s^{i,N,n}-X_{\kappa_n(s)}^{i,N,n}\big|^p + n^{-p} E\big(1+\big|X_{\kappa_n(s)}^{i,N,n}\big|^{3\rho p/2}\big)
\end{align*} 
and then the application of Lemmas \ref{lem:mb:milstein} and \ref{lem:one-step:rate:milstein}  completes the proof. 
\end{proof}
\begin{lem} \label{lem:b-b:x}
Let Assumptions \ref{as:x0}, \ref{as:coercivity}, \ref{as:monotonicity:rate}, \ref{as:polynomial:Lipschitz}, \ref{as:lipschitz}, \ref{as:der:x:poly:lip} and \ref{as:der:mea:poly:lip} be satisfied. Then, 
\begin{align*} 
E \big|X_s^{i,N}  - X_{s}^{i,N,n}\big|^{p-2} \big(X_s^{i,N}  & - X_{s}^{i,N,n}\big)  \big(b_{s}\big(X_s^{i,N,n},\mu_s^{X,N,n}\big)-b_{\kappa_n(s)}^n\big(X_{\kappa_n(s)}^{i,N,n},\mu_{\kappa_n(s)}^{X,N,n}\big) \big)
\\
&   \leq K n^{-p}  + K \sup_{i\in\{1,\cdots,N\}}\sup_{r\in[0,s]} E\big|X_r^{i,N}-X_r^{i,N,n}\big|^p, 
\end{align*}
for any $p\leq p_0/(2\rho+4)$, $s\in[0,T]$, $i\in\{1,\ldots,N\}$  and $n,N\in\mathbb{N}$ where  constant $K>0$ does not depend on $n$ and $N$.  
\end{lem}
\begin{proof} 
We first prove  
\begin{align} 
E \big|X_s^{i,N}  - X_{s}^{i,N,n}\big|^{p-2} \big(X_s^{i,N}  & - X_{s}^{i,N,n}\big)  \big(b_{\kappa_n(s)}\big(X_s^{i,N,n},\mu_s^{X,N,n}\big)-b_{\kappa_n(s)}\big(X_{\kappa_n(s)}^{i,N,n},\mu_{\kappa_n(s)}^{X,N,n}\big) \big) \notag
\\
&   \leq K n^{-p}  + K \sup_{i\in\{1,\cdots,N\}}\sup_{r\in[0,s]} E\big|X_r^{i,N}-X_r^{i,N,n}\big|^p \label{eq:rate:b:inter}
\end{align}
for any $s\in[0,T]$  and $n,N\in\mathbb{N}$. For this, notice that,
\begin{align} 
 E \big|&X_s^{i,N}  - X_{s}^{i,N,n}\big|^{p-2}  \big( X_s^{i,N}  - X_{s}^{i,N,n}\big) \big(b_{\kappa_n(s)}\big(  X_s^{i,N,n},\mu_{s}^{X,N,n}\big)  -b_{\kappa_n(s)}\big(X_{\kappa_n(s)}^{i,N,n},\mu_{\kappa_n(s)}^{X,N,n}\big) \big) \notag
\\
 = &  E \big|X_s^{i,N}  - X_{s}^{i,N,n}\big|^{p-2}  \sum_{k=1}^d  \big(X_s^{(k),i,N}  - X_{s}^{(k), i,N,n}\big)\Big\{  b^{(k)}_{\kappa_n(s)}\big(  X_s^{i,N,n},\mu_{s}^{X,N,n}\big)  -b^{(k)}_{\kappa_n(s)}\big(X_{\kappa_n(s)}^{i,N,n},\mu_{\kappa_n(s)}^{X,N,n}\big) \notag
\\
& \quad \qquad-  \partial_x b^{(k)}_{\kappa_n(s)}\big(X_{\kappa_n(s)}^{i,N,n},\mu_{\kappa_n(s)}^{X,N,n}\big) \big( X_{s}^{i,N,n}- X_{\kappa_n(s)}^{i,N,n} \big )  \notag
\\
& \quad \qquad-  \frac{1}{N} \sum_{j=1}^N  \partial_{\mu}b^{(k)}_{\kappa_n(s)}\big(X_{\kappa_n(s)}^{i,N,n}, \mu_{\kappa_n(s)}^{X,N,n}, X_{\kappa_n(s)}^{j,N,n}\big) \big(X_{s}^{j,N,n}- X_{\kappa_n(s)}^{j,N,n} \big) \Big\} \notag
\\
&  + E\big|X_s^{i,N}  - X_{s}^{i,N,n}\big|^{p-2}   \sum_{k=1}^d \big(X_s^{(k),i,N}  - X_{s}^{(k), i,N,n}\big)  \partial_x b^{(k)}_{\kappa_n(s)}\big(X_{\kappa_n(s)}^{i,N,n},\mu_{\kappa_n(s)}^{X,N,n}\big) \big(X_{s}^{i,N,n}- X_{\kappa_n(s)}^{i,N,n} \big)  \notag
\\
& +E\big|X_s^{i,N}  - X_{s}^{i,N,n}\big|^{p-2}   \sum_{k=1}^d \big(X_s^{(k),i,N}  - X_{s}^{(k), i,N,n}\big) \notag
\\
& \qquad \quad\times \frac{1}{N} \sum_{j=1}^N  \partial_{\mu}b^{(k)}_{\kappa_n(s)}\big(X_{\kappa_n(s)}^{i,N,n}, \mu_{\kappa_n(s)}^{X,N,n}, X_{\kappa_n(s)}^{j,N,n}\big) \big(X_{s}^{j,N,n}- X_{\kappa_n(s)}^{j,N,n} \big) \notag
\\
=:& T_1+T_2+T_3 \label{eq:T1+T2+T3}
\end{align}
for any $s\in[0,T]$, $i\in\{1,\ldots, N\}$ and $n, N\in \mathbb{N}$. By using Lemma \ref{lem:f:rate:local} and Young's inequality, 
\begin{align*}
T_1 := & E \big|X_s^{i,N}  - X_{s}^{i,N,n}\big|^{p-2}  \sum_{k=1}^d  \big(X_s^{(k),i,N}  - X_{s}^{(k), i,N,n}\big)\Big\{  b^{(k)}_{\kappa_n(s)}\big(  X_s^{i,N,n},\mu_{s}^{X,N,n}\big)  -b^{(k)}_{\kappa_n(s)}\big(X_{\kappa_n(s)}^{i,N,n},\mu_{\kappa_n(s)}^{X,N,n}\big) 
\\
& \qquad\qquad -  \partial_x b^{(k)}_{\kappa_n(s)}\big(X_{\kappa_n(s)}^{i,N,n},\mu_{\kappa_n(s)}^{X,N,n}\big) \big( X_{s}^{i,N,n}- X_{\kappa_n(s)}^{i,N,n} \big )  
\\
& \qquad \qquad -  \frac{1}{N} \sum_{j=1}^N  \partial_{\mu}b^{(k)}_{\kappa_n(s)}\big(X_{\kappa_n(s)}^{i,N,n}, \mu_{\kappa_n(s)}^{X,N,n}, X_{\kappa_n(s)}^{j,N,n}\big) \big(X_{s}^{j,N,n}- X_{\kappa_n(s)}^{j,N,n} \big) \Big\} 
\\
\leq &K E \big|X_s^{i,N}  - X_{s}^{i,N,n}\big|^{p-2}   \sum_{k=1}^d \big|X_s^{(k),i,N}  - X_{s}^{(k), i,N,n}\big| 
\\
& \times \Big\{\big(1+\big|X_s^{i,N, n}\big|+\big|X_{\kappa_n(s)}^{i,N,n}\big|\big)^{\rho/2-1} \big|X_s^{i,N, n}-X_{\kappa_n(s)}^{i,N,n}\big|^2  + \frac{1}{N}\sum_{j=1}^N  \big|X_s^{j,N, n}-X_{\kappa_n(s)}^{j,N,n}\big|^2 \Big\}
\\
\leq & K E \big|X_s^{i,N}  - X_{s}^{i,N,n}\big|^{p} + K E\Big\{\big(1+\big|X_s^{i,N, n}\big|+\big|X_{\kappa_n(s)}^{i,N,n}\big|\big)^{\rho/2-1} \big|X_s^{i,N, n}-X_{\kappa_n(s)}^{i,N,n}\big|^2 
\\
&+ \frac{1}{N}\sum_{j=1}^N  \big|X_s^{j,N, n}-X_{\kappa_n(s)}^{j,N,n}\big|^2 \Big\}^p
\end{align*} 
which on the application of H\"older's inequality, Lemmas \ref{lem:mb:milstein} and \ref{lem:one-step:rate:milstein} yields, 
\begin{align}
T_1 \leq K E \big|X_s^{i,N}  - X_{s}^{i,N,n}\big|^{p}+ K n^{-p} \label{eq:T1},
\end{align}
for any $s\in[0,T]$, $i\in\{1,\ldots, N\}$ and  $n, N\in\mathbb{N}$. 

Notice that $T_2$ can be written as, 
\begin{align}
T_2 := & E\big|X_s^{i,N}  - X_{s}^{i,N,n}\big|^{p-2}   \sum_{k=1}^d \big(X_s^{(k),i,N}  - X_{s}^{(k), i,N,n}\big)  \partial_x b^{(k)}_{\kappa_n(s)}\big(X_{\kappa_n(s)}^{i,N,n},\mu_{\kappa_n(s)}^{X,N,n}\big) \big(X_{s}^{i,N,n}- X_{\kappa_n(s)}^{i,N,n} \big) \notag
\\
= & E\big|X_s^{i,N}  - X_{s}^{i,N,n}\big|^{p-2}   \sum_{k=1}^d \big(X_s^{(k),i,N}  - X_{s}^{(k), i,N,n}\big)  \partial_x b^{(k)}_{\kappa_n(s)}\big(X_{\kappa_n(s)}^{i,N,n},\mu_{\kappa_n(s)}^{X,N,n}\big) \notag
\\
&  \qquad \times \Big\{\int_{\kappa_n(s)}^s  b^n_{\kappa_n(r)} \big(X_{\kappa_n(r)}^{i,N, n}, \mu_{\kappa_n(r)}^{X,N,n}\big) dr +   \int_{\kappa_n(s)}^s \tilde{\sigma}^n_{\kappa_n(r)} \big(r,X_{\kappa_n(r)}^{i,N,n}, \mu_{\kappa_n(r)}^{X,N,n}\big) dW_r^{i} \notag
\\
& \qquad \qquad +  \int_{\kappa_n(s)}^s \tilde{\sigma}^{0,n}_{\kappa_n(r)} \big(r,X_{\kappa_n(r)}^{i,N,n}, \mu_{\kappa_n(r)}^{X,N,n}\big) dW_r^{0}    \Big\}=: T_{21}+T_{22}+T_{23} \label{eq:T21+T22+T23}
\end{align}
for any $s\in[0,T]$, $i\in\{1,\ldots, N\}$ and $n, N\in\mathbb{N}$. For $T_{21}$, one uses the Cauchy-Schwarz inequality, Young's inequality, equation \eqref{eq:w2:mb} and Remarks \ref{rem:poly:growth}, \ref{rem:der:growth:poly} and Lemma \ref{lem:mb:milstein} to obtain, 
\begin{align}
T_{21} := & E\big|X_s^{i,N}  - X_{s}^{i,N,n}\big|^{p-2}   \sum_{k=1}^d \big(X_s^{(k),i,N}  - X_{s}^{(k), i,N,n}\big)  \notag
\\
& \times  \partial_x b^{(k)}_{\kappa_n(s)}\big(X_{\kappa_n(s)}^{i,N,n},\mu_{\kappa_n(s)}^{X,N,n}\big) \int_{\kappa_n(s)}^s  b^n_{\kappa_n(r)} \big(X_{\kappa_n(r)}^{i,N, n}, \mu_{\kappa_n(r)}^{X,N,n}\big) dr \notag
\\
\leq & K E\big|X_s^{i,N}  - X_{s}^{i,N,n}\big|^{p-1}     n^{-1} \Big\{ 1+ \big|X_{\kappa_n(s)}^{i,N, n}\big|^{\rho+1} +\mathcal{W}_2\big(\mu_{\kappa_n(s)}^{X,N,n}, \delta_0\big)\Big\}  \notag
\\
\leq & K \sup_{i\in\{1,\ldots, N\}}\sup_{0\leq r\leq s}E\big|X_s^{i,N}  - X_{s}^{i,N,n}\big|^{p} + K n^{-p} \label{eq:T21}
\end{align}
for any $s\in[0,T]$ and $n, N\in\mathbb{N}$.  

For $T_{22}$, let us define, for $k=1,\ldots,d$, 
\begin{align}
\mathcal{M}^{(k)}(\kappa_n(s), s):=\partial_x b^{(k)}_{\kappa_n(s)}\big(X_{\kappa_n(s)}^{i,N,n},\mu_{\kappa_n(s)}^{X,N,n}\big)  \int_{\kappa_n(s)}^s \tilde{\sigma}^n_{\kappa_n(r)} \big(r,X_{\kappa_n(r)}^{i,N,n}, \mu_{\kappa_n(r)}^{X,N,n}\big) dW_r^{i} \label{eq:mar1},
\end{align}
and notice that, due to Burkholder--Davis--Gundy inequality, H\"older's inequality and Remark \ref{rem:der:growth:poly}, 
\begin{align}
E\big|\mathcal{M}^{(k)}(\kappa_n(s), s)\big|^{q}& =E\Big|\partial_x b^{(k)}_{\kappa_n(s)}\big(X_{\kappa_n(s)}^{i,N,n},\mu_{\kappa_n(s)}^{X,N,n}\big)  \int_{\kappa_n(s)}^s \tilde{\sigma}^n_{\kappa_n(r)} \big(r,X_{\kappa_n(r)}^{i,N,n}, \mu_{\kappa_n(r)}^{X,N,n}\big) dW_r^{i}\Big|^{q} \notag
\\
& \leq K n^{-q/2+1}E \int_{\kappa_n(s)}^s \big|\partial_x b^{(k)}_{\kappa_n(s)}\big(X_{\kappa_n(s)}^{i,N,n},\mu_{\kappa_n(s)}^{X,N,n}\big) \big|^q \big|\tilde{\sigma}^n_{\kappa_n(r)} \big(r,X_{\kappa_n(r)}^{i,N,n}, \mu_{\kappa_n(r)}^{X,N,n}\big) \big|^q dr \notag
\\
& \leq K n^{-q/2+1}E \int_{\kappa_n(s)}^s \big( 1+ \big|X_{\kappa_n(s)}^{i,N,n}\big|\big)^{(q\rho)/2} \big|\tilde{\sigma}^n_{\kappa_n(r)} \big(r,X_{\kappa_n(r)}^{i,N,n}, \mu_{\kappa_n(r)}^{X,N,n}\big) \big|^q dr \notag
\\
& \leq K n^{-q/2+1} \int_{\kappa_n(s)}^s \big\{E\big( 1+ \big|X_{\kappa_n(s)}^{i,N,n}\big|\big)^{q\rho p_0}  \big\{E\big|\tilde{\sigma}^n_{\kappa_n(r)} \big(r,X_{\kappa_n(r)}^{i,N,n}, \mu_{\kappa_n(r)}^{X,N,n}\big) \big|^{2q}\big\}^{1/2} dr  \notag
\end{align}
which on using Lemma \ref{lem:mb:milstein} and Corollary \ref{cor:sigma:rate} yields, 
\begin{align}
E\big|\mathcal{M}(\kappa_n(s), s)\big|^{q} \leq  K n^{-q/2} \label{eq:mar1:rate}
\end{align}
for any $q\leq p_0/(\rho+2)$ and $s\in[0,T]$.  Using the notation in equation \eqref{eq:mar1} along with Lemma~ \ref{lem:mb:milstein} and the Cauchy-Schwarz inequality, 
\begin{align*}
T_{22}:=& E\big|X_s^{i,N}  - X_{s}^{i,N,n}\big|^{p-2}   \sum_{k=1}^d \big(X_s^{(k),i,N}  - X_{s}^{(k), i,N,n}\big)  
\\
& \qquad \times \partial_x b^{(k)}_{\kappa_n(s)}\big(X_{\kappa_n(s)}^{i,N,n},\mu_{\kappa_n(s)}^{X,N,n}\big)  \int_{\kappa_n(s)}^s \tilde{\sigma}^n_{\kappa_n(r)} \big(r,X_{\kappa_n(r)}^{i,N,n}, \mu_{\kappa_n(r)}^{X,N,n}\big) dW_r^{i} 
\\
= & E\big|X_{\kappa_n(s)}^{i,N}  - X_{{\kappa_n(s)}}^{i,N,n}\big|^{p-2}   \sum_{k=1}^d \big(X_{\kappa_n(s)}^{(k),i,N}  - X_{{\kappa_n(s)}}^{(k), i,N,n}\big) \mathcal{M}^{(k)}(\kappa_n(s), s)
\\
& + E\Big\{\big|X_s^{i,N}  - X_{s}^{i,N,n}\big|^{p-2}   \sum_{k=1}^d \big(X_s^{(k),i,N}  - X_{s}^{(k), i,N,n}\big)   
\\
& \qquad - \big|X_{\kappa_n(s)}^{i,N}  - X_{{\kappa_n(s)}}^{i,N,n}\big|^{p-2}   \sum_{k=1}^d \big(X_{\kappa_n(s)}^{(k),i,N}  - X_{{\kappa_n(s)}}^{(k), i,N,n}\big)  \Big\}\mathcal{M}^{(k)}(\kappa_n(s), s)
\\
=& E\sum_{k=1}^d  \Big\{\big|X_s^{i,N}  - X_{s}^{i,N,n}\big|^{p-2}   \big(X_s^{(k),i,N}  - X_{s}^{(k), i,N,n}\big)   
\\
& \qquad - \big|X_{\kappa_n(s)}^{i,N}  - X_{{\kappa_n(s)}}^{i,N,n}\big|^{p-2}   \big(X_{\kappa_n(s)}^{(k),i,N}  - X_{{\kappa_n(s)}}^{(k), i,N,n}\big)  \Big\}\mathcal{M}^{(k)}(\kappa_n(s), s)
\\
\leq & E  \Big|\big|X_s^{i,N}  - X_{s}^{i,N,n}\big|^{p-2}   \big(X_s^{i,N}  - X_{s}^{i,N,n}\big)   
\\
& \qquad - \big|X_{\kappa_n(s)}^{i,N}  - X_{{\kappa_n(s)}}^{i,N,n}\big|^{p-2}   \big(X_{\kappa_n(s)}^{i,N}  - X_{{\kappa_n(s)}}^{i,N,n}\big)  \Big| \big|\mathcal{M}(\kappa_n(s), s)\big| 
\end{align*}
for any $s\in[0,T]$, $i\in\{1,\ldots, N\}$ and $n,N\in\mathbb{N}$. Also, for an $\mathbb{R}^d$-valued  function $f(z)=|z|^{p-2} z$ for any $z\in\mathbb{R}^d$ and a $\theta \in (0,1)$, 
\begin{align*}
\Big||x|^{p-2} x- |y|^{p-2} y \Big| =\big|f(x)-f(y) \big| \leq K (p-1)|\theta x +(1-\theta)y|^{p-2}  |x-y| \leq K \big\{|x|^{p-2}  +|y|^{p-2}  \big\} |x-y|
\end{align*}
for any $x, y \in \mathbb{R}^d$ which further implies on taking $x=X_s^{i,N}  - X_{s}^{i,N,n}$ and $y=X_{\kappa_n(s)}^{i,N}  - X_{{\kappa_n(s)}}^{i,N,n}$, 
\begin{align}
\Big| \big|X_s^{i,N}  & - X_{s}^{i,N,n}\big|^{p-2} \big(X_s^{i,N}  - X_{s}^{i,N,n}\big)-\big|X_{\kappa_n(s)}^{i,N}  - X_{{\kappa_n(s)}}^{i,N,n}\big|^{p-2} \big(X_{\kappa_n(s)}^{i,N}  - X_{{\kappa_n(s)}}^{i,N,n}\big) \Big| \notag
\\
& \leq  K \big\{\big|X_s^{i,N}  - X_{s}^{i,N,n}\big|^{p-2}  +\big|X_{\kappa_n(s)}^{i,N}  - X_{{\kappa_n(s)}}^{i,N,n}\big|^{p-2}  \big\} \big|X_s^{i,N}  - X_{s}^{i,N,n}-\big(X_{\kappa_n(s)}^{i,N}  - X_{{\kappa_n(s)}}^{i,N,n}\big)\big| \label{eq:mvt}
\end{align}
for any $s\in[0,T]$, $i\in\{1,\ldots, N\}$ and $n,N\in\mathbb{N}$. Thus, 
\begin{align}
T_{22}\leq & K  E \big\{\big|X_s^{i,N}  - X_{s}^{i,N,n}\big|^{p-2}  +\big|X_{\kappa_n(s)}^{i,N}  - X_{{\kappa_n(s)}}^{i,N,n}\big|^{p-2}  \big\} \big|X_s^{i,N} - X_{\kappa_n(s)}^{i,N} -\big( X_{s}^{i,N,n}   - X_{{\kappa_n(s)}}^{i,N,n}\big)\big| \notag
\\
&  \times \big| \mathcal{M}(\kappa_n(s), s) \big| \notag
\\
\leq & K  E \big\{\big|X_s^{i,N}  - X_{s}^{i,N,n}\big|^{p-2}  +\big|X_{\kappa_n(s)}^{i,N}  - X_{{\kappa_n(s)}}^{i,N,n}\big|^{p-2}  \big\}  \notag
\\
& \times \Big\{\,\Big|\int_{\kappa_n(s)}^s \big\{b_r\big(X_r^{i,N}, \mu_r^{X,N}\big)-b^n_{\kappa_n(r)} \big(X_{\kappa_n(r)}^{i,N, n}, \mu_{\kappa_n(r)}^{X,N,n}\big)  \big\} dr \Big| \notag
\\
&\qquad  + \Big|\int_{\kappa_n(s)}^s \big\{ \sigma_r\big(X_r^{i,N}, \mu_r^{X,N}\big)- \tilde{\sigma}^n_{\kappa_n(r)} \big(r,X_{\kappa_n(r)}^{i,N,n}, \mu_{\kappa_n(r)}^{X,N,n}\big) \big\} dW_r^{i}  \Big| \notag
\\
&\qquad  + \Big|\int_{\kappa_n(s)}^s \big\{ \sigma^0_r\big(X_r^{i,N}, \mu_r^{X,N}\big) - \tilde{\sigma}^{0,n}_{\kappa_n(r)} \big(r,X_{\kappa_n(r)}^{i,N,n}, \mu_{\kappa_n(r)}^{X,N,n}\big) \big\} dW_r^{0} \Big| \, \Big\}  \big|  \mathcal{M}(\kappa_n(s), s) \big|  \notag
\\
=: & \, T_{22A}+T_{22B}+T_{22C}  \label{eq:T22A+T22B+T22C}
\end{align}
 for any $s\in[0,T]$, $i\in\{1,\ldots, N\}$ and $n,N\in\mathbb{N}$. Also, by Assumption \ref{as:polynomial:Lipschitz}, 
 \begin{align*}
 T_{22A} : = & K  E \big\{\big|X_s^{i,N}  - X_{s}^{i,N,n}\big|^{p-2}  +\big|X_{\kappa_n(s)}^{i,N}  - X_{{\kappa_n(s)}}^{i,N,n}\big|^{p-2}  \big\} 
 \\
 & \times \Big|\int_{\kappa_n(s)}^s \big\{b_r\big(X_r^{i,N}, \mu_r^{X,N}\big)-b^n_{\kappa_n(r)} \big(X_{\kappa_n(r)}^{i,N, n}, \mu_{\kappa_n(r)}^{X,N,n}\big)  \big\} dr \Big| \big|  \mathcal{M}(\kappa_n(s), s) \big|
 \\
 \leq  & K  E \big\{\big|X_s^{i,N}  - X_{s}^{i,N,n}\big|^{p-2}  +\big|X_{\kappa_n(s)}^{i,N}  - X_{{\kappa_n(s)}}^{i,N,n}\big|^{p-2}  \big\} 
 \\
 & \times \int_{\kappa_n(s)}^s \Big\{ \big|b_r\big(X_r^{i,N}, \mu_r^{X,N}\big) - b_r\big(X_r^{i,N, n}, \mu_r^{X,N, n}\big)\big| 
 \\
 &\qquad +  \big| b_r\big(X_r^{i,N, n}, \mu_r^{X,N, n}\big) -b^n_{\kappa_n(r)} \big(X_{\kappa_n(r)}^{i,N, n}, \mu_{\kappa_n(r)}^{X,N,n}\big)  \big| \Big\} dr  \big|  \mathcal{M}(\kappa_n(s), s) \big|
 \\
 \leq  & K  E \big\{\big|X_s^{i,N}  - X_{s}^{i,N,n}\big|^{p-2}  +\big|X_{\kappa_n(s)}^{i,N}  - X_{{\kappa_n(s)}}^{i,N,n}\big|^{p-2}  \big\} 
 \\
 & \times \int_{\kappa_n(s)}^s \Big\{ \big(1+ \big|X_r^{i,N}\big|+\big|X_r^{i,N,n}\big|\big)^{\rho/2}\big|X_r^{i,N}-X_r^{i,N,n}\big| +\mathcal{W}_2 \big( \mu_r^{X,N},  \mu_r^{X,N, n} \big)
 \\
 &\qquad +  \big| b_r\big(X_r^{i,N, n}, \mu_r^{X,N, n}\big) -b^n_{\kappa_n(r)} \big(X_{\kappa_n(r)}^{i,N, n}, \mu_{\kappa_n(r)}^{X,N,n}\big)  \big| \Big\} dr   \big| \mathcal{M}(\kappa_n(s), s) \big|
 \end{align*}
 and the application of Young's inequality  and H\"older's inequality yields, 
 \begin{align*}
 T_{22A} \leq & K  E \big\{\big|X_s^{i,N}  - X_{s}^{i,N,n}\big|^{p-2}  +\big|X_{\kappa_n(s)}^{i,N}  - X_{{\kappa_n(s)}}^{i,N,n}\big|^{p-2}  \big\}^{p/(p-2)}  
 \\
 & + K n^{-p/2+1}E\int_{\kappa_n(s)}^s \Big\{ \big(1+ \big|X_r^{i,N}\big|+\big|X_r^{i,N,n}\big|\big)^{\rho/2}\big|X_r^{i,N}-X_r^{i,N,n}\big| 
\\
&+\mathcal{W}_2 \big( \mu_r^{X,N},  \mu_r^{X,N, n} \big)+  \big| b_r\big(X_r^{i,N, n}, \mu_r^{X,N, n}\big) -b^n_{\kappa_n(r)} \big(X_{\kappa_n(r)}^{i,N, n}, \mu_{\kappa_n(r)}^{X,N,n}\big)  \big| \Big\}^{p/2} dr \big|  \mathcal{M}(\kappa_n(s), s) \big|^{p/2}
\\
\leq & K \sup_{0\leq r\leq s}E \big|X_r^{i,N}  - X_{r}^{i,N,n}\big|^{p}  
\\
& + K n^{-p/2+1}E\int_{\kappa_n(s)}^s  \big(1+ \big|X_r^{i,N}\big|+\big|X_r^{i,N,n}\big|\big)^{(p\rho)/4}\big|X_r^{i,N}-X_r^{i,N,n}\big|^{p/2} dr \big| \mathcal{M}(\kappa_n(s), s)\big|^{p/2}
\\
&+ K n^{-p/2+1}E\int_{\kappa_n(s)}^s  \mathcal{W}_2^{p/2} \big( \mu_r^{X,N},  \mu_r^{X,N, n} \big) dr \big|  \mathcal{M}(\kappa_n(s), s)\big|^{p/2}
\\
&+ K n^{-p/2+1}E\int_{\kappa_n(s)}^s   \big| b_r\big(X_r^{i,N, n}, \mu_r^{X,N, n}\big) -b^n_{\kappa_n(r)} \big(X_{\kappa_n(r)}^{i,N, n}, \mu_{\kappa_n(r)}^{X,N,n}\big)  \big|^{p/2} dr  \big|  \mathcal{M}(\kappa_n(s), s) \big|^{p/2}
\\
\leq & K \sup_{0\leq r\leq s}E \big|X_r^{i,N}  - X_{r}^{i,N,n}\big|^{p}  
\\
& + K n^{-p/2+1}\int_{\kappa_n(s)}^s  \big\{E\big|X_r^{i,N}-X_r^{i,N,n}\big|^{p}\big\}^{1/2} \Big\{E\big(1+ \big|X_r^{i,N}\big|+\big|X_r^{i,N,n}\big|\big)^{(p\rho)/2} \big| \mathcal{M}(\kappa_n(s), s) \big|^{p}\Big\}^{1/2} dr
\\
&+ K n^{-p/2+1}\int_{\kappa_n(s)}^s \big\{E \mathcal{W}_2^{p} \big( \mu_r^{X,N},  \mu_r^{X,N, n} \big)  \big\}^{1/2} \Big\{E\big| \mathcal{M}(\kappa_n(s), s) \big|^{p}\Big\}^{1/2} dr
\\
&+ K n^{-p/2+1}\int_{\kappa_n(s)}^s  \big\{ E\big| b_r\big(X_r^{i,N, n}, \mu_r^{X,N, n}\big) -b^n_{\kappa_n(r)} \big(X_{\kappa_n(r)}^{i,N, n}, \mu_{\kappa_n(r)}^{X,N,n}\big)  \big|^{p} \big\}^{1/2} \Big\{E\big| \mathcal{M}(\kappa_n(s), s)\big|^{p} \big\}^{1/2} dr
\end{align*}
for any $s\in[0,T]$, $i\in\{1,\ldots, N\}$ and $n,N\in\mathbb{N}$.   On using estimates from  equation \eqref{eq:mar1:rate}, Lemmas \ref{lem:mb:milstein} and \ref{lem:b-b}, 
\begin{align*}
 T_{22A} \leq & K \sup_{0\leq r\leq s}E \big|X_r^{i,N}  - X_{r}^{i,N,n}\big|^{p} 
  \\
& + K n^{-p/2+1}\int_{\kappa_n(s)}^s  \big\{E\big|X_r^{i,N}-X_r^{i,N,n}\big|^{p}\big\}^{1/2} \Big\{E\big(1+ \big|X_r^{i,N}\big|+\big|X_r^{i,N,n}\big|\big)^{p\rho} E \big| \mathcal{M}(\kappa_n(s), s) \big|^{2p}\Big\}^{1/4} dr
\\
&+ K n^{-3p/4+1}\int_{\kappa_n(s)}^s \big\{E \mathcal{W}_2^{p} \big( \mu_r^{X,N},  \mu_r^{X,N, n} \big)  \big\}^{1/2}  dr + K n^{-p}
\end{align*}
and thus the application of estimates in equation \eqref{eq:w2:particle} and Young's inequality yields, 
\begin{align} \label{eq:T22A}
T_{22A} \leq & K \sup_{i\in \{1,\ldots N\}}\sup_{0\leq r\leq s}E \big|X_r^{i,N}  - X_{r}^{i,N,n}\big|^{p}  + K n^{-p}, 
\end{align}
for any $s\in[0,T]$, and $n,N\in\mathbb{N}$.

For estimating $T_{22B}$, one notices that, 
\begin{align*}
T_{22B} :=&K  E \big\{\big|X_s^{i,N}  - X_{s}^{i,N,n}\big|^{p-2}  +\big|X_{\kappa_n(s)}^{i,N}  - X_{{\kappa_n(s)}}^{i,N,n}\big|^{p-2}  \big\} 
\\
& \qquad \times \Big|\int_{\kappa_n(s)}^s \big\{ \sigma_r\big(X_r^{i,N}, \mu_r^{X,N}\big)- \tilde{\sigma}^n_{\kappa_n(r)} \big(r,X_{\kappa_n(r)}^{i,N,n}, \mu_{\kappa_n(r)}^{X,N,n}\big) \big\} dW_r^{i}  \Big| \big|  \mathcal{M}(\kappa_n(s), s) \big|
\\
\leq & K  E \big\{\big|X_s^{i,N}  - X_{s}^{i,N,n}\big|^{p-2}  +\big|X_{\kappa_n(s)}^{i,N}  - X_{{\kappa_n(s)}}^{i,N,n}\big|^{p-2}  \big\} 
\\
& \qquad \times \Big\{ \Big|\int_{\kappa_n(s)}^s \big\{ \sigma_r\big(X_r^{i,N}, \mu_r^{X,N}\big)- \sigma_{r} \big(X_{r}^{i,N,n}, \mu_{r}^{X,N,n}\big) 
\\
& \qquad \qquad + \sigma_r\big(X_r^{i,N,n}, \mu_r^{X,N,n}\big)- \tilde{\sigma}^n_{\kappa_n(r)} \big(r,X_{\kappa_n(r)}^{i,N,n}, \mu_{\kappa_n(r)}^{X,N,n}\big) \big\} dW_r^{i}  \Big| \Big\}\big|  \mathcal{M}(\kappa_n(s), s) \big|
\\
\leq &  K  E \big\{\big|X_s^{i,N}  - X_{s}^{i,N,n}\big|^{p-2}  +\big|X_{\kappa_n(s)}^{i,N}  - X_{{\kappa_n(s)}}^{i,N,n}\big|^{p-2}  \big\} 
\\
& \qquad \times  \Big|\int_{\kappa_n(s)}^s \big\{ \sigma_r\big(X_r^{i,N}, \mu_r^{X,N}\big)- \sigma_{r} \big(X_{r}^{i,N,n}, \mu_{r}^{X,N,n}\big)\big\} dW_r^{i}  \Big| \big|  \mathcal{M}(\kappa_n(s), s) \big|
\\
& +  K  E \big\{\big|X_s^{i,N}  - X_{s}^{i,N,n}\big|^{p-2}  +\big|X_{\kappa_n(s)}^{i,N}  - X_{{\kappa_n(s)}}^{i,N,n}\big|^{p-2}  \big\} 
\\
& \qquad \times  \Big|\int_{\kappa_n(s)}^s \big\{ \sigma_r\big(X_r^{i,N, n}, \mu_r^{X,N, n}\big)- \tilde{\sigma}^n_{\kappa_n(r)} \big(r,X_{\kappa_n(r)}^{i,N,n}, \mu_{\kappa_n(r)}^{X,N,n}\big) \big\} dW_r^{i}  \Big| \big|  \mathcal{M}(\kappa_n(s), s) \big|
\end{align*}
which on using H\"older's inequality and Young's inequality yields, 
\begin{align*}
T_{22B} \leq & K  E \big\{\big|X_s^{i,N}  - X_{s}^{i,N,n}\big|^{p-2}  +\big|X_{\kappa_n(s)}^{i,N}  - X_{{\kappa_n(s)}}^{i,N,n}\big|^{p-2}  \big\}^{p/(p-2)} 
\\
& + K  E\Big|\int_{\kappa_n(s)}^s \big\{ \sigma_r\big(X_r^{i,N}, \mu_r^{X,N}\big)- \sigma_{r} \big(X_{r}^{i,N,n}, \mu_{r}^{X,N,n}\big)\big\} dW_r^{i}  \Big|^{p/2} \big|  \mathcal{M}(\kappa_n(s), s) \big|^{p/2}
\\
& + K E \Big|\int_{\kappa_n(s)}^s \big\{ \sigma_r\big(X_r^{i,N, n}, \mu_r^{X,N, n}\big)- \tilde{\sigma}^n_{\kappa_n(r)} \big(r,X_{\kappa_n(r)}^{i,N,n}, \mu_{\kappa_n(r)}^{X,N,n}\big) \big\} dW_r^{i}  \Big|^{p/2} \big|  \mathcal{M}(\kappa_n(s), s) \big|^{p/2}
\\
 \leq & K  \sup_{0\leq r \leq s}E \big|X_r^{i,N}  - X_{r}^{i,N,n}\big|^{p} 
 \\
 & + K  \Big\{E\Big|\int_{\kappa_n(s)}^s \big\{ \sigma_r\big(X_r^{i,N}, \mu_r^{X,N}\big)- \sigma_{r} \big(X_{r}^{i,N,n}, \mu_{r}^{X,N,n}\big)\big\} dW_r^{i}  \Big|^{3p/4}\Big\}^{2/3}\big\{ E\big|  \mathcal{M}(\kappa_n(s), s) \big|^{3p/2}\big\}^{1/3}
 \\
 & + K \Big\{E \Big|\int_{\kappa_n(s)}^s \big\{ \sigma_r\big(X_r^{i,N, n}, \mu_r^{X,N, n}\big)- \tilde{\sigma}^n_{\kappa_n(r)} \big(r,X_{\kappa_n(r)}^{i,N,n}, \mu_{\kappa_n(r)}^{X,N,n}\big) \big\} dW_r^{i}  \Big|^{p} \Big\}^{1/2}\big\{E\big|  \mathcal{M}(\kappa_n(s), s) \big|^{p}\big\}^{1/2}
\end{align*}
and then one applies the estimates in equation \eqref{eq:mar1:rate}, Remark \ref{rem:poly:lipschitz} and Lemma \ref{lem:sigma-sigma:milstein} to obtain, 
\begin{align*}
T_{22B} \leq & K  \sup_{0\leq r \leq s}E \big|X_r^{i,N}  - X_{r}^{i,N,n}\big|^{p} 
  \\
 & + K  n^{-p/4}\Big\{ n^{-3p/8+1}E\int_{\kappa_n(s)}^s \big| \sigma_r\big(X_r^{i,N}, \mu_r^{X,N}\big)- \sigma_{r} \big(X_{r}^{i,N,n}, \mu_{r}^{X,N,n}\big)\big|^{3p/4} dr  \Big\}^{2/3}
 \\
 & + Kn^{-p/4} \Big\{ n^{-p/2+1}E \int_{\kappa_n(s)}^s \big| \sigma_r\big(X_r^{i,N, n}, \mu_r^{X,N,n }\big)- \tilde{\sigma}^n_{\kappa_n(r)} \big(r,X_{\kappa_n(r)}^{i,N,n}, \mu_{\kappa_n(r)}^{X,N,n}\big) \big|^p dr \Big\}^{1/2}
 \\
\leq & K  \sup_{0\leq r \leq s}E \big|X_r^{i,N}  - X_{r}^{i,N,n}\big|^{p} +K n^{-p}
\\
& + K  n^{-p/4}\Big\{ n^{-3p/8+1}\int_{\kappa_n(s)}^s \big\{ E\big( 1+ \big|X_r^{i,N}\big| + \big|X_{r}^{i,N,n}\big|\big)^{3\rho p/16}\big|X_r^{i,N}-X_{r}^{i,N,n}\big|^{3p/4}
\\
&\qquad \qquad \qquad +E\mathcal{W}_2^{3p/4} \big( \mu_r^{X,N}, \mu_{r}^{X,N,n}\big) \big\} dr  \Big\}^{2/3}
 \\
\leq & K  \sup_{0\leq r \leq s}E \big|X_r^{i,N}  - X_{r}^{i,N,n}\big|^{p} +K n^{-p}
\\
& + K  n^{-p/4}\Big\{ n^{-3p/8+1}\int_{\kappa_n(s)}^s\big\{ \big( E\big( 1+ \big|X_r^{i,N}\big| + \big|X_{r}^{i,N,n}\big|\big)^{3\rho p/4} \big)^{1/4}\big( E\big|X_r^{i,N}-X_{r}^{i,N,n}\big|^{p}\big)^{3/4}
\\
&\qquad \qquad \qquad +\big(E\mathcal{W}_2^{p} \big( \mu_r^{X,N}, \mu_{r}^{X,N,n}\big)\big)^{3/4} \big\} dr  \Big\}^{2/3}
\end{align*}
for any $s\in[0,T]$, $i\in\{1,\ldots, N\}$ and $n,N\in\mathbb{N}$. Further, one uses Lemma \ref{lem:mb:milstein}, equation \eqref{eq:w2:particle} and Young's inequality to obtain, 
\begin{align}
T_{22B} \leq &  K  \sup_{0\leq r \leq s}E \big|X_r^{i,N}  - X_{r}^{i,N,n}\big|^{p} +K n^{-p} \notag
\\
& + K  n^{-p/2} \Big\{\Big( \sup_{0\leq r\leq s}E\big|X_r^{i,N}-X_{r}^{i,N,n}\big|^{p}\Big)^{3/4} + \Big( \sup_{0\leq r\leq s}\frac{1}{N}\sum_{j=1}^N  E\big|X_r^{j,N}-X_{r}^{j,N,n}\big|^{p} \Big)^{3/4}\Big\}^{2/3} \notag
\\
 \leq & K  \sup_{i\in\{1,\ldots,N\}}\sup_{0\leq r \leq s}E \big|X_r^{i,N}  - X_{r}^{i,N,n}\big|^{p} +K n^{-p} \label{eq:T22B}
\end{align}
for any $s\in[0,T]$ and $n,N\in\mathbb{N}$.  

By adapting arguments similar to the one used in the estimation of $T_{22B}$, one obtains,
\begin{align}
T_{22C} : = & K  E \big\{\big|X_s^{i,N}  - X_{s}^{i,N,n}\big|^{p-2}  +\big|X_{\kappa_n(s)}^{i,N}  - X_{{\kappa_n(s)}}^{i,N,n}\big|^{p-2}  \big\} \notag
\\
& \times \Big|\int_{\kappa_n(s)}^s \big\{ \sigma^0_r\big(X_r^{i,N}, \mu_r^{X,N}\big) - \tilde{\sigma}^{0,n}_{\kappa_n(r)} \big(r,X_{\kappa_n(r)}^{i,N,n}, \mu_{\kappa_n(r)}^{X,N,n}\big) \big\} dW_r^{0} \Big| \,  \big|  \mathcal{M}(\kappa_n(s), s) \big| \notag
\\
 \leq & K  \sup_{i\in\{1,\ldots,N\}}\sup_{0\leq r \leq s}E \big|X_r^{i,N}  - X_{r}^{i,N,n}\big|^{p} +K n^{-p} \label{eq:T22C}
\end{align}
for any $s\in[0,T]$ and $n,N\in\mathbb{N}$.  

Hence, on combining estimates obtained in equations  \eqref{eq:T22A}, \eqref{eq:T22B} and \eqref{eq:T22C} in equation \eqref{eq:T22A+T22B+T22C}, one obtains
\begin{align}
T_{22} \leq K  \sup_{i\in\{1,\ldots,N\}}\sup_{0\leq r \leq s}E \big|X_r^{i,N}  - X_{r}^{i,N,n}\big|^{p} +K n^{-p} \label{eq:T22}
\end{align}
for any $s\in[0,T]$ and $n,N\in\mathbb{N}$. 

By using methods similar to the one used in estimating $T_{22}$, one also obtains, 
\begin{align}
T_{23}:= & E\big|X_s^{i,N}  - X_{s}^{i,N,n}\big|^{p-2}   \sum_{k=1}^d \big(X_s^{(k),i,N}  - X_{s}^{(k), i,N,n}\big)  \notag
\\
& \qquad \times \partial_x b^{(k)}_{\kappa_n(s)}\big(X_{\kappa_n(s)}^{i,N,n},\mu_{\kappa_n(s)}^{X,N,n}\big)  \int_{\kappa_n(s)}^s \tilde{\sigma}^{0,n}_{\kappa_n(r)} \big(r,X_{\kappa_n(r)}^{i,N,n}, \mu_{\kappa_n(r)}^{X,N,n}\big) dW_r^{0} \notag
\\
 \leq & K  \sup_{i\in\{1,\ldots,N\}}\sup_{0\leq r \leq s}E \big|X_r^{i,N}  - X_{r}^{i,N,n}\big|^{p} +K n^{-p} \label{eq:T23}
\end{align}
for any $s\in[0,T]$ and $n,N\in\mathbb{N}$.  

Thus, merging  the estimates in equations \eqref{eq:T21}, \eqref{eq:T22} and \eqref{eq:T23} yields, 
\begin{align}
T_{2} \leq K  \sup_{i\in\{1,\ldots,N\}}\sup_{0\leq r \leq s}E \big|X_r^{i,N}  - X_{r}^{i,N,n}\big|^{p} +K n^{-p} \label{eq:T2}
\end{align}
for any $s\in[0,T]$ and $n,N\in\mathbb{N}$.

For estimating $T_3$, use equation \eqref{eq:milstein} to get, 
 \begin{align*}
T_3 := & E\big|X_s^{i,N}  - X_{s}^{i,N,n}\big|^{p-2}   \sum_{k=1}^d \big(X_s^{(k),i,N}  - X_{s}^{(k), i,N,n}\big)
\\
& \qquad \times \frac{1}{N} \sum_{j=1}^N  \partial_{\mu}b^{(k)}_{\kappa_n(s)}\big(X_{\kappa_n(s)}^{i,N,n}, \mu_{\kappa_n(s)}^{X,N,n}, X_{\kappa_n(s)}^{j,N,n}\big) \big(X_{s}^{j,N,n}- X_{\kappa_n(s)}^{j,N,n} \big) 
\\
=& E\big|X_s^{i,N}  - X_{s}^{i,N,n}\big|^{p-2}   \sum_{k=1}^d \big(X_s^{(k),i,N}  - X_{s}^{(k), i,N,n}\big)\frac{1}{N} \sum_{j=1}^N  \partial_{\mu}b^{(k)}_{\kappa_n(s)}\big(X_{\kappa_n(s)}^{i,N,n}, \mu_{\kappa_n(s)}^{X,N,n}, X_{\kappa_n(s)}^{j,N,n}\big)
\\
& \qquad \Big\{\int_{\kappa_n(s)}^s  b^n_{\kappa_n(r)} \big(X_{\kappa_n(r)}^{j,N, n}, \mu_{\kappa_n(r)}^{X,N,n}\big) dr +   \int_{\kappa_n(s)}^s \tilde{\sigma}^n_{\kappa_n(r)} \big(r,X_{\kappa_n(r)}^{j,N,n}, \mu_{\kappa_n(r)}^{X,N,n}\big) dW_r^{j} \notag
\\
& \qquad \qquad +  \int_{\kappa_n(s)}^s \tilde{\sigma}^{0,n}_{\kappa_n(r)} \big(r,X_{\kappa_n(r)}^{j,N,n}, \mu_{\kappa_n(r)}^{X,N,n}\big) dW_r^{0}    \Big\}=: T_{31}+T_{32}+T_{33}
\end{align*}
for any $s\in[0,T]$, $i\in\{1,\ldots, N\}$ and $n,N\in\mathbb{N}$.

For estimating $T_{31}$, one uses the Cauchy-Schwarz inequality, Remarks \ref{rem:poly:growth}, \ref{rem:der:growth:poly} and H\"older's inequality to obtain, 
\begin{align*}
T_{31}:= & E\big|X_s^{i,N}  - X_{s}^{i,N,n}\big|^{p-2}   \sum_{k=1}^d \big(X_s^{(k),i,N}  - X_{s}^{(k), i,N,n}\big) 
\\
& \qquad \times \frac{1}{N} \sum_{j=1}^N  \partial_{\mu}b^{(k)}_{\kappa_n(s)}\big(X_{\kappa_n(s)}^{i,N,n}, \mu_{\kappa_n(s)}^{X,N,n}, X_{\kappa_n(s)}^{j,N,n}\big)  \int_{\kappa_n(s)}^s  b^n_{\kappa_n(r)} \big(X_{\kappa_n(r)}^{j,N, n}, \mu_{\kappa_n(r)}^{X,N,n}\big) dr 
\\
\leq & E\big|X_s^{i,N}  - X_{s}^{i,N,n}\big|^{p-2}   \sum_{k=1}^d \big|X_s^{(k),i,N}  - X_{s}^{(k), i,N,n}\big|   
\\
& \qquad \times \frac{1}{N} \sum_{j=1}^N  \big| \partial_{\mu}b^{(k)}_{\kappa_n(s)}\big(X_{\kappa_n(s)}^{i,N,n}, \mu_{\kappa_n(s)}^{X,N,n}, X_{\kappa_n(s)}^{j,N,n}\big)\big| \int_{\kappa_n(s)}^s  \big|b^n_{\kappa_n(r)} \big(X_{\kappa_n(r)}^{j,N, n}, \mu_{\kappa_n(r)}^{X,N,n}\big)\big|  dr 
\\
\leq & K n^{-1} E\big|X_s^{i,N}  - X_{s}^{i,N,n}\big|^{p-1}    \Big\{ \big(1+\big|X_{\kappa_n(s)}^{j,N, n}\big|\big)^{\rho/2+1}   + \mathcal{W}_2\big(\mu_{\kappa_n(r)}^{X,N,n}, \delta_0 \big)  \Big\}
\end{align*}
which on using  Young's inequality, equation \eqref{eq:w2:mb} and Lemma \ref{lem:mb:milstein} yields, 
\begin{align}
T_{31} \leq  K \sup_{i\in\{1,\ldots, N\}}\sup_{0\leq r \leq s} E\big|X_s^{i,N}  - X_{s}^{i,N,n}\big|^{p} + K n^{-p}  \label{eq:T31}
\end{align}
for any $s\in[0,T]$ and $n,N\in\mathbb{N}$.

For estimating $T_{32}$, let us define, for $k=1,\ldots, d$, 
\begin{align}
\mathcal{M}^{N,(k)}(\kappa_n(s), s):= \frac{1}{N} \sum_{j=1}^N  \partial_{\mu}b^{(k)}_{\kappa_n(s)}\big(X_{\kappa_n(s)}^{i,N,n}, \mu_{\kappa_n(s)}^{X,N,n}, X_{\kappa_n(s)}^{j,N,n}\big) \int_{\kappa_n(s)}^s \tilde{\sigma}^n_{\kappa_n(r)} \big(r,X_{\kappa_n(r)}^{j,N,n}, \mu_{\kappa_n(r)}^{X,N,n}\big) dW_r^{j} \label{eq:mar3}
\end{align}
and observe that due to Burkholder--Davis--Gundy inequality, H\"older's inequality and Remark \ref{rem:der:growth:poly},  
\begin{align}
E\big|\mathcal{M}^{N, (k)} &  (\kappa_n(s), s)\big|^q \leq K  E \frac{1}{N} \sum_{j=1}^N \Big| \int_{\kappa_n(s)}^s  \partial_{\mu}b^{(k)}_{\kappa_n(s)}\big(X_{\kappa_n(s)}^{i,N,n}, \mu_{\kappa_n(s)}^{X,N,n}, X_{\kappa_n(s)}^{j,N,n}\big)   \tilde{\sigma}^n_{\kappa_n(r)} \big(r,X_{\kappa_n(r)}^{j,N,n}, \mu_{\kappa_n(r)}^{X,N,n}\big) dW_r^{j} \Big|^q \notag
\\
&\leq K  n^{-q/2+1} E \frac{1}{N} \sum_{j=1}^N  \int_{\kappa_n(s)}^s  \big|\partial_{\mu}b^{(k)}_{\kappa_n(s)}\big(X_{\kappa_n(s)}^{i,N,n}, \mu_{\kappa_n(s)}^{X,N,n}, X_{\kappa_n(s)}^{j,N,n}\big)\big|^q  \big| \tilde{\sigma}^n_{\kappa_n(r)} \big(r,X_{\kappa_n(r)}^{j,N,n}, \mu_{\kappa_n(r)}^{X,N,n}\big)\big|^q dr   \notag
\\
&\leq K  n^{-q/2+1}  \frac{1}{N} \sum_{j=1}^N  \int_{\kappa_n(s)}^s E \big| \tilde{\sigma}^n_{\kappa_n(r)} \big(r,X_{\kappa_n(r)}^{j,N,n}, \mu_{\kappa_n(r)}^{X,N,n}\big)\big|^{q} dr   \notag
\end{align}
which on using Corollary \ref{cor:sigma:rate} yields, 
\begin{align}
E\big|\mathcal{M}^{N}   (\kappa_n(s), s)\big|^q  \leq K n^{-q/2} \label{eq:mar2}
\end{align}
for any $q\leq p_0/(\rho/2+1)$. Thus, using the notation defined in equation \eqref{eq:mar3}, 
\begin{align*}
T_{32}  : = & E\big|X_s^{i,N}  - X_{s}^{i,N,n}\big|^{p-2}   \sum_{k=1}^d \big(X_s^{(k),i,N}  - X_{s}^{(k), i,N,n}\big)
\\
& \qquad \times \frac{1}{N} \sum_{j=1}^N  \partial_{\mu}b^{(k)}_{\kappa_n(s)}\big(X_{\kappa_n(s)}^{i,N,n}, \mu_{\kappa_n(s)}^{X,N,n}, X_{\kappa_n(s)}^{j,N,n}\big) \int_{\kappa_n(s)}^s \tilde{\sigma}^n_{\kappa_n(r)} \big(r,X_{\kappa_n(r)}^{j,N,n}, \mu_{\kappa_n(r)}^{X,N,n}\big) dW_r^{j}
\\
= & E\big\{\big|X_s^{i,N}  - X_{s}^{i,N,n}\big|^{p-2}   \sum_{k=1}^d \big(X_s^{i,N}  - X_{s}^{i,N,n}\big)
\\
& \qquad -\big|X_{\kappa_n(s)}^{i,N}  - X_{\kappa_n(s)}^{i,N,n}\big|^{p-2}   \sum_{k=1}^d \big(X_{\kappa_n(s)}^{i,N}  - X_{\kappa_n(s)}^{i,N,n}\big)\big\} \mathcal{M}^{N}(\kappa_n(s), s)
\\
& + E\big|X_{\kappa_n(s)}^{i,N}  - X_{\kappa_n(s)}^{i,N,n}\big|^{p-2}   \big(X_{\kappa_n(s)}^{i,N}  - X_{\kappa_n(s)}^{i,N,n}\big)  \mathcal{M}^{N}(\kappa_n(s), s)
\end{align*}
and then notice that the third term in the above equation vanishes because in view of Lemma \ref{lem:mb:milstein}, $ \mathcal{M}^{N, (k)}(\kappa_n(s), s)$ is a martingale. Thus,  equation \eqref{eq:mvt} yields, 
\begin{align*}
T_{32}\leq & K  E \big\{\big|X_s^{i,N}  - X_{s}^{i,N,n}\big|^{p-2}  +\big|X_{\kappa_n(s)}^{i,N}  - X_{{\kappa_n(s)}}^{i,N,n}\big|^{p-2}  \big\} \big|X_s^{i,N} - X_{\kappa_n(s)}^{i,N} -\big( X_{s}^{i,N,n}   - X_{{\kappa_n(s)}}^{i,N,n}\big)\big|
\\
&  \times \big| \mathcal{M}^N(\kappa_n(s), s) \big|
\end{align*}
for any $s\in[0,T]$, $i\in\{1,\ldots, N\}$ and $n,N\in\mathbb{N}$. Now, replacing $\mathcal{M}(\kappa_n(s),s)$ by $\mathcal{M}^N(\kappa_n(s),s)$ in equation \eqref{eq:T22A+T22B+T22C} and in what follows along with the estimates in equation \eqref{eq:mar2}, one can obtain the following estimates, 
\begin{align}
T_{32} \leq K  \sup_{i\in\{1,\ldots,N\}}\sup_{0\leq r \leq s}E \big|X_r^{i,N}  - X_{r}^{i,N,n}\big|^{p} +K n^{-p} \label{eq:T32}
\end{align}
for any $s\in[0,T]$ and $n,N\in\mathbb{N}$. 

By using methods similar to the one used in estimating $T_{32}$, one also obtains, 
\begin{align}
T_{33}  : = & E\big|X_s^{i,N}  - X_{s}^{i,N,n}\big|^{p-2}   \sum_{k=1}^d \big(X_s^{(k),i,N}  - X_{s}^{(k), i,N,n}\big) \notag
\\
& \qquad \times \frac{1}{N} \sum_{j=1}^N  \partial_{\mu}b^{(k)}_{\kappa_n(s)}\big(X_{\kappa_n(s)}^{i,N,n}, \mu_{\kappa_n(s)}^{X,N,n}, X_{\kappa_n(s)}^{j,N,n}\big) \int_{\kappa_n(s)}^s \tilde{\sigma}^{0,n}_{\kappa_n(r)} \big(r,X_{\kappa_n(r)}^{j,N,n}, \mu_{\kappa_n(r)}^{X,N,n}\big) dW_r^{0} \notag
\\
 \leq K & \sup_{i\in\{1,\ldots,N\}}\sup_{0\leq r \leq s}E \big|X_r^{i,N}  - X_{r}^{i,N,n}\big|^{p} +K n^{-p} \label{eq:T33}
\end{align}
for any $s\in[0,T]$ and $n,N\in\mathbb{N}$.  

Hence, on combining the estimates from equations \eqref{eq:T31}, \eqref{eq:T32} and \eqref{eq:T33}, one obtains
\begin{align}
T_{3} \leq K  \sup_{i\in\{1,\ldots,N\}}\sup_{0\leq r \leq s}E \big|X_r^{i,N}  - X_{r}^{i,N,n}\big|^{p} +K n^{-p} \label{eq:T3}
\end{align}
for any $s\in[0,T]$ and $n,N\in\mathbb{N}$. The proof of equation \eqref{eq:rate:b:inter} is completed by substituting estimates from equations \eqref{eq:T1}, \eqref{eq:T2} and \eqref{eq:T3} in equation \eqref{eq:T1+T2+T3}. 

In order to complete the proof, we consider 
\begin{align*} 
E \big|X_s^{i,N} & - X_{s}^{i,N,n}\big|^{p-2} \big(X_s^{i,N}  - X_{s}^{i,N,n}\big)  \big(b_{s}\big(X_s^{i,N,n},\mu_s^{X,N,n}\big)- b_{\kappa_n(s)}^n \big(X_{\kappa_n(s)}^{i,N,n},\mu_{\kappa_n(s)}^{X,N,n}\big)\big)
\\
=  & E \big|X_s^{i,N}  - X_{s}^{i,N,n}\big|^{p-2} \big(X_s^{i,N}   - X_{s}^{i,N,n}\big)  \big(b_{s}\big(X_s^{i,N,n},\mu_s^{X,N,n}\big)- b_{\kappa_n(s)}\big(X_{s}^{i,N,n},\mu_{s}^{X,N,n}\big)
\\
& +  b_{\kappa_n(s)}\big(X_{s}^{i,N,n},\mu_{s}^{X,N,n}\big)-b_{\kappa_n(s)}\big(X_{\kappa_n(s)}^{i,N,n},\mu_{\kappa_n(s)}^{X,N,n}\big) 
\\
& +b_{\kappa_n(s)}\big(X_{\kappa_n(s)}^{i,N,n},\mu_{\kappa_n(s)}^{X,N,n}\big) - b_{\kappa_n(s)}^n\big(X_{\kappa_n(s)}^{i,N,n},\mu_{\kappa_n(s)}^{X,N,n}\big) \big)
\end{align*}
which on using Young's inequality,   equation \eqref{eq:Taming:milstein}, Assumption \ref{as:lipschitz} and equation \eqref{eq:rate:b:inter}   yields, 
\begin{align*}
E \big|X_s^{i,N} & - X_{s}^{i,N,n}\big|^{p-2} \big(X_s^{i,N}  - X_{s}^{i,N,n}\big)  \big(b_{s}\big(X_s^{i,N,n},\mu_s^{X,N,n}\big)- b_{\kappa_n(s)}^n \big(X_{\kappa_n(s)}^{i,N,n},\mu_{\kappa_n(s)}^{X,N,n}\big)\big)
\\
 \leq & K E \big|X_s^{i,N}  - X_{s}^{i,N,n}\big|^{p} + E\big| b_{s}\big(X_s^{i,N,n},\mu_s^{X,N,n}\big)- b_{\kappa_n(s)}\big(X_{s}^{i,N,n},\mu_{s}^{X,N,n}\big) \big|^p
\\
& + E \big|X_s^{i,N}  - X_{s}^{i,N,n}\big|^{p-2} \big(X_s^{i,N}   - X_{s}^{i,N,n}\big) \big(b_{\kappa_n(s)}\big(X_{s}^{i,N,n},\mu_{s}^{X,N,n}\big)-b_{\kappa_n(s)}\big(X_{\kappa_n(s)}^{i,N,n},\mu_{\kappa_n(s)}^{X,N,n}\big) \big)
\\
& + E\big| b_{\kappa_n(s)}\big(X_{\kappa_n(s)}^{i,N,n},\mu_{\kappa_n(s)}^{X,N,n}\big) - b_{\kappa_n(s)}^n\big(X_{\kappa_n(s)}^{i,N,n},\mu_{\kappa_n(s)}^{X,N,n}\big) \big|^p
\\
\leq & K n^{-p} +  E\Big| b_{\kappa_n(s)}\big(X_{\kappa_n(s)}^{i,N,n},\mu_{\kappa_n(s)}^{X,N,n}\big) - \frac{b_{\kappa_n(s)}\big(X_{\kappa_n(s)}^{i,N,n},\mu_{\kappa_n(s)}^{X,N,n}\big)}{1+n^{-1} \big| X_{\kappa_n(s)}^{i,N,n} \big|^\rho} \Big|^p
\end{align*}
for any $s\in[0,T]$, $i\in\{1,\ldots, N\}$ and $n,N\in\mathbb{N}$. The proof is completed by using Remark \ref{rem:poly:growth}, equation \eqref{eq:w2:mb} and Lemma \ref{lem:mb:milstein}. 
\end{proof}
The following is the main result of this section. 
\begin{thm}[\textbf{Rate of Convergence}] \label{thm:rate:milstein}
Let Assumptions \ref{as:x0}, \ref{as:coercivity}, \ref{as:monotonicity:rate}, \ref{as:polynomial:Lipschitz}, \ref{as:lipschitz}, \ref{as:der:x:poly:lip} and \ref{as:der:mea:poly:lip} be satisfied.  Then, the explicit Milstein-type scheme \eqref{eq:milstein} converges to the true solution of the interacting particle system \eqref{eq:interacting} associated with McKean--Vlasov SDE \eqref{eq:MVSDE} in strong sense with the $L^p$ rate of convergence equal to $1$ i.e.,  
\begin{align*}
\sup_{i\in\{1,\ldots, N\}} E\sup_{t\in[0,T]} |X_t^{i,N}-X_t^{i,N,n}|^p \leq K n^{-p},
\end{align*} 
for any $p< \min\{p_1, p_0/(2\rho+4)\}$, where the constant $K>0$ does not depend on $n, N\in\mathbb{N}$. 
\end{thm}
\begin{proof}
The proof follows by applying arguments used in the proof of Theorem \ref{thm:rate:euler}. Indeed, one replaces equation \eqref{eq:true-euler} by the following equation, 
\begin{align*}
X_t^{i,N} & - X_t^{i,N, n} =  \int_0^t \big( b_s\big(X_s^{i,N}, \mu_s^{X,N}\big) - b_{\kappa_n(s)}^n \big(X_{\kappa_n(s)}^{i,N, n},  \mu_{\kappa_n(s)}^{X,N,n}\big) \big) ds  
\\
&+ \int_0^t \big(\sigma_s\big(X_s^{i,N}, \mu_s^{X,N}\big) -\tilde{\sigma}_{\kappa_n(s)}^n\big(s, X_{\kappa_n(s)}^{i,N,n}, \mu_{\kappa_n(s)}^{X,N,n}\big) \big) dW_s^{i} 
\\
&+ \int_0^t \big(\sigma^0_s\big(X_s^{i,N}, \mu_s^{X,N}\big) - \tilde{\sigma}^{0,n}_{\kappa_n(s)}\big(s, X_{\kappa_n(s)}^{i,N,n}, \mu_{\kappa_n(s)}^{X,N,n}\big)  \big)dW_s^{0} 
\end{align*}
and thus equation \eqref{eq:ito:est:euler} is replaced by, 
\begin{align*}
 E\big|&X_t^{i,N}  - X_t^{i,N, n} \big|^p \leq K E\int_0^t \big|X_s^{i,N}  - X_s^{i,N, n} \big|^{p}  ds 
\\
& + K E\int_0^t \big|X_s^{i,N}  - X_s^{i,N, n} \big|^{p-2}  \big(X_s^{i,N}  - X_s^{i,N, n} \big) \big(b_s \big(X_s^{i,N, n},  \mu_s^{X,N, n}\big) - b_{\kappa_n(s)}^n \big(X_{\kappa_n(s)}^{i,N, n},  \mu_{\kappa_n(s)}^{X,N,n}\big) \big) ds  
\\
&+K E \int_0^t   \big|\sigma_s\big(X_s^{i,N, n}, \mu_s^{X,N, n}\big) -\tilde{\sigma}_{\kappa_n(s)}^n\big(X_{\kappa_n(s)}^{i,N,n}, \mu_{\kappa_n(s)}^{X,N,n}\big) \big|^p ds 
\\
&+ K E\int_0^t    \big| \sigma^0_s\big(X_s^{i,N, n}, \mu_s^{X,N, n}\big)  -\tilde{\sigma}^{0,n}_{\kappa_n(s)}\big(X_{\kappa_n(s)}^{i,N,n}, \mu_{\kappa_n(s)}^{X,N,n}\big)  \big|^p ds 
\end{align*}
for any $t\in[0,T]$, $i\in\{1,\ldots,N\}$ and $n, N\in\mathbb{N}$. The application of Lemmas \ref{lem:sigma-sigma:milstein} and  \ref{lem:b-b:x} yields, 
\begin{align*}
\sup_{i\in\{1,\ldots,N\}}\sup_{0\leq r \leq t}  E\big|X_r^{i,N}  - X_r^{i,N, n} \big|^p \leq K \int_0^t \sup_{i\in\{1,\ldots,N\}}\sup_{0\leq r \leq s}E \big|X_r^{i,N}  - X_r^{i,N, n} \big|^{p}  ds  + K n^{-p},
\end{align*}
for any $t\in[0,T]$, $i\in\{1,\ldots,N\}$ and $n, N\in\mathbb{N}$. The use of Gronwall's inequality yields, 
\begin{align*}
\sup_{i\in\{1,\ldots,N\}}\sup_{0\leq t \leq T}   E\big|X_t^{i,N}  - X_t^{i,N, n} \big|^p \leq K n^{-p},
\end{align*}
for any $p< \min\{p_1, p_0/(2\rho+4)\}$, $t\in[0,T]$, $i\in\{1,\ldots,N\}$ and $n, N\in\mathbb{N}$. Then, the result follows by  Lemma \ref{lem:gk}.
\end{proof}

\section{Numerical Results}\label{Sec:4}

The aim of this section is to demonstrate the practical performance of the schemes proposed in this article. 
To approximate the law $\mathcal{L}_{X_{t_n}}$ (or the conditional law $\mathcal{L}^{1}_{X_{t_n}}$) at each time-step $t_n$ contained in a uniform time-grid on $[0,T]$, we use a standard particle method with $N$ particles for each realisation of $W^0$. For our numerical experiments, we used $N=10^3$ to approximate the conditional law (e.g., a conditional expectation) and perform $100$ independent outer simulations over $\Omega^{0}$, in order to  estimate the $L^p$-error on the product space. This resembles in our last example the estimation of nested expectations, as e.g.\ in \cite{bujok2015} or the survey paper \cite{giles} and references given therein, and suggests further research concerning efficient multi-level Monte Carlo methods for estimation with conditional laws.

Since the exact solution of the examples considered below is not known, we determined the strong convergence rates, in terms of time-steps, by comparing two numerical solutions (at time $T=1$) obtained from simulations based on a fine and coarse time-grid, respectively. 
To obtain a coupling between these two levels, the same Brownian motions are used for both. 
In order to test the strong convergence in $h$, we thus compute the $L^p$-error, for $p=2,4,6$, denoted by $\text{RMSE}$, 
\begin{equation*}
\text{RMSE}:= \left( \frac{1}{N} \sum_{i=1}^{N} \left|X^{i,N,l}_T - X^{i,N,l-1}_T \right|^p \right)^{1/p},
\end{equation*}
where $X^{i,N,l}_T$ denotes the numerical solution of $X$ at time $T$ computed with $N$ particles and $2^l$ time-steps, where $l \geq 1$.

This section numerically illustrates the convergence of the tamed Euler and Milstein schemes for interacting particle systems, with and without common noise.

To demonstrate numerically the performance of our proposed tamed Euler scheme, we present the following McKean--Vlasov equation, which is a mean-field version of the well-known $3/2$-model that is often used for pricing VIX options and modelling certain stochastic volatility processes: \\ \\
\textbf{Example 1. (Mean-field $3/2$ Stochastic Volatility Model).}   
Consider the $2$-dimensional McKean--Vlasov SDE
\begin{align*}
X_t = X_0 + \int_{0}^{t} \left( \lambda X_s(\mu - |X_s|) + E X_s \right) ds + \int_{0}^{t} \xi |X_s|^{3/2} dW_s,
\end{align*}
where we choose $X_0= [1,1]^T$, $\lambda=2.5$, $\mu=1$ and 
\[
   \xi=
  \left[ {\begin{array}{cc}
   2/\sqrt{10} & 1/\sqrt{10} \\
   1/\sqrt{10} & 2/\sqrt{10} \\
  \end{array} } \right].
\]
The above model without the mean-field has been studied numerically in \cite{sabanis2016}.
Fig.\ \ref{fig:SuperlinearDiffusion1a} depicts the strong convergence rate for this example and we observe a rate of order $1/2$. 

To illustrate the convergence behaviour of our proposed tamed Milstein scheme, we consider the following $1$-dimensional example which has been studied numerically in \cite{ beyn2017, kumar2019}.  \\ \\ 
\textbf{Example 2. (Mean-field Stochastic Double Well Dynamics).} Consider the $1$-dimensional McKean--Vlasov SDE
\begin{align*}
X_t = X_0 + \int_{0}^{t} \left( X_s(1-X_s^2) + E X_s \right) ds + \int_{0}^{t} \sigma (1 -X_s^2) dW_s,
\end{align*}
with $X_0=1$ and $\sigma = 0.3$. We recall that the $L$-derivative terms appearing in the Milstein scheme (\ref{eq:milstein}) are not implemented, as these terms are expected to be close to zero for a large number of particles. However, for a small number of particles (e.g., $5$ or $10$), they have to be considered to obtain a convergence rate of order $1$; see \cite{bao2020} for a more rigorous discussion on the numerical role of the $L$-derivative terms (in the case of zero common noise).
\begin{figure}[!h]
\begin{subfigure}[b]{0.48\textwidth}
\includegraphics[width=\textwidth]{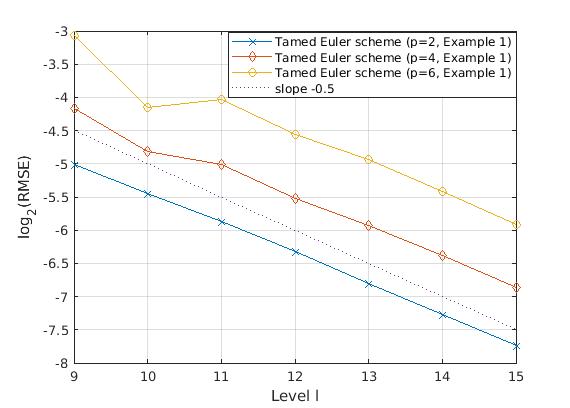}
\caption{}
\label{fig:SuperlinearDiffusion1a}
\end{subfigure}
\begin{subfigure}[b]{0.48\textwidth}
\includegraphics[width=\textwidth]{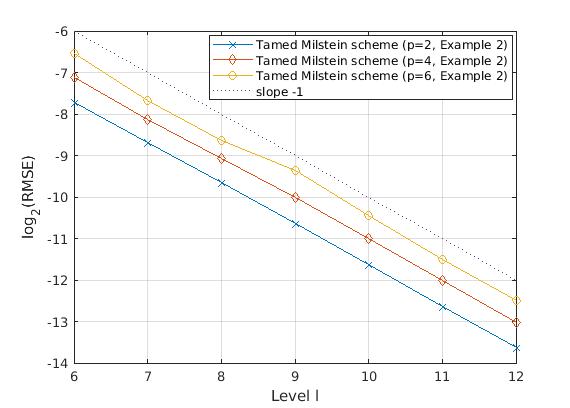}
\caption{}
\label{fig:SuperlinearDiffusion2}
\end{subfigure}
\caption{Left: Strong convergence of the tamed Euler scheme for Example 1.
Right: Strong convergence of the tamed Milstein scheme for Example 2.}
\end{figure}
Fig.\ \ref{fig:SuperlinearDiffusion2} and Fig.\ \ref{fig:SuperlinearDiffusion3} reveal the expected strong convergence rate of order 1. \\ \\
\textbf{Example 3. (Mean-field Stochastic Double Well Dynamics with Common Noise).} Consider the $1$-dimensional McKean--Vlasov SDE
\begin{align*}
X_t = X_0 + \int_{0}^{t} \left( X_s(1-X_s^2) + E^1 X_s \right) ds + \int_{0}^{t} \sigma (1 -X_s^2) dW_s + \int_{0}^{t} \sigma (1 -X_s^2) dW_s^{0},
\end{align*}
with $X_0=1$ and $\sigma = 0.1$. This example is a slight modification of Example 2 and involves additionally a common noise term. We remark that in this case iterated stochastic integrals (the Brownian motion $W$ integrated against $W^{0}$ and the other way around) appear, but due to the antisymmetry property of the L\'{e}vy area these terms will cancel.
\begin{figure}[!h]
\centering
\includegraphics[width=0.48\textwidth]{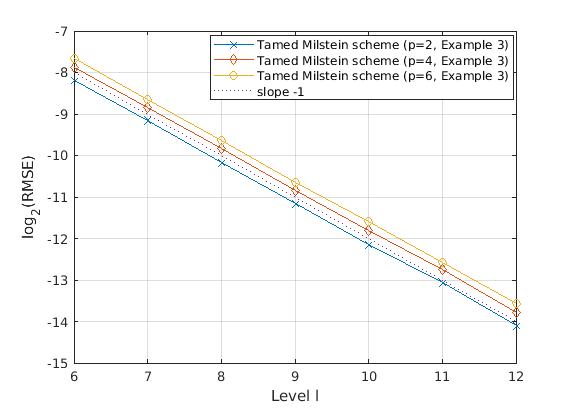}
\caption{Strong convergence of the tamed Milstein scheme for Example 3.}
\label{fig:SuperlinearDiffusion3}
\end{figure}

\section*{acknowledgement}
We would like to thank David \v{S}i\v{s}ka and Lukasz Szpruch, both from University of Edinburgh, for allowing us to use the above mentioned Lemma~\ref{lem:DL} from their working paper which played a crucial role in the proof of Theorem \ref{thm:eu}. 


\subsection*{Authors' addresses}\hfill\\

\noindent Chaman Kumar, Department of Mathematics, Indian Institute of Technology Roorkee,  Roorkee, 247 667, India.  \\{\tt 
chaman.kumar@ma.iitr.ac.in}\\

\noindent Neelima, Department of Mathematics, Ramjas College, University of Delhi, Delhi, 110 007, India. \\{\tt 
neelima\_maths@ramjas.du.ac.in} \\

\noindent Christoph Reisinger, Mathematical Institute, University of Oxford.
Andrew Wiles Building, Woodstock Road, Oxford, OX2 6GG, UK. \\{\tt christoph.reisinger@maths.ox.ac.uk}\\

\noindent Wolfgang Stockinger, Mathematical Institute, University of Oxford.
Andrew Wiles Building, Woodstock Road, Oxford, OX2 6GG, UK. \\{\tt wolfgang.stockinger@maths.ox.ac.uk}\\


\begin{thebibliography}{99}
\bibitem{ambrosio2008}
\newblock L.\ Ambrosio, N.\ Gigli  and G.\ Savar\'e, 
\newblock Gradient Flows in Metric Spaces and in the Space of Probability Measures, 
\newblock \emph{Birkh\"aser Verlag}, Basel, 2nd.\ edition, 2008. 

\bibitem{bao2020}
 \newblock J.\ Bao, C.\ Reisinger, P.\ Ren and W.\ Stockinger,
 \newblock First order convergence of Milstein schemes for McKean equations and interacting particle systems, 
 \newblock \emph{arXiv:2004.03325}, 2020. 
 
\bibitem{BBP} 
\newblock M.\ Bauer, T.\ Meyer-Brandis and F.\ Proske, 
\newblock Strong solutions of mean-field stochastic differential equations with irregular drift, 
\newblock \emph{Electronic Journal of Probability}, vol.\ 23(132), 2018.   
 
 
\bibitem{beyn2017}
 \newblock W.-J.\ Beyn, E.\ Isaak, R.\ Kruse, 
  \newblock Stochastic C-stability and B-consistency of explicit and implicit Milstein-type schemes, 
   \newblock \emph{J of Scientific Computing}, \textbf{70(3)}, (2017), 1042--1077. 

\bibitem{bujok2015}
 \newblock K.\ Bujok, B.\ Hambly, C.\ Reisinger, 
  \newblock Multilevel simulation of functionals of Bernoulli random variables with application to basket credit derivatives, 
   \newblock \emph{Methodology and Computing in Applied Probability}, \textbf{17(3)}, (2015), 579--604. 

\bibitem{cardaliaguet2013}
 \newblock P.\ Cardaliaguet, 
  \newblock Notes on Mean-field games, from P.-L. Lions' lectures at Coll\'ege de France, 
   \newblock \url{https://www.ceremade.dauphine.fr/~cardaliaguet/MFG20130420.pdf}, 2013.
   
 
  \bibitem{carmona2018a}
 \newblock R.\ Carmona and F.\ Delarue, 
 \newblock Probabilistic theory of mean field games with applications I: Mean-field FBSDEs, control, and games, 
 \newblock \emph{Springer International Publishing}, Switzerland, 2018.
 
 \bibitem{carmona2018b}
 \newblock R.\ Carmona and F.\ Delarue, 
 \newblock Probabilistic theory of mean field games with applications II: Mean field games with common noise and master equations,
 \newblock \emph{Springer International Publishing}, Switzerland, 2018.
   
 \bibitem{giles}
 \newblock M.\ Giles,
 \newblock Multilevel Monte Carlo methods,
 \newblock \emph{Acta Numerica}, 2018.   
   
 \bibitem{goard}
 \newblock J.\ Goard and M.\ Mazur, 
 \newblock Stochastic volatility models and the pricing of VIX options,
 \newblock \emph{Math. Finance}, \textbf{23}, pp.\ 439--458, 2016.


\bibitem{gyongy1980}
\newblock  I.\ Gy\"ongy and N.\ V.\ Krylov, 
\newblock On stochastic equations with respect to semimartingales I,
\newblock  \emph{Stochastics}, \textbf{4}, pp.\ 1--21, 1980. 


\bibitem{gyongy2003}   
\newblock  I.\ Gy\"ongy and N.\ Krylov, 
\newblock  On the rate of convergence of splitting-up approximations for SPDEs,
\newblock \emph{In: Stochastic Inequalities and Applications. Progress in Probability}, vol.\ 56, Birkhäuser, Basel, pp.\ 301--321. 2003.   
    

%
%
%
%
\bibitem{hammersley2019}
\newblock W.\ Hammersley, D.\ \v{S}i\v{s}ka and L.\ Szpruch, 
\newblock Weak existence and uniqueness for McKean--Vlasov SDEs with common noise,
\newblock \emph{arXiv:1908.00955}, 2020. 

\bibitem{HSS}
\newblock W.\ Hammersley, D.\ \v{S}i\v{s}ka and L.\ Szpruch,
\newblock McKean--Vlasov SDE under measure dependent Lyapunov conditions,
\newblock \emph{arXiv:1802.03974v1}, 2018.

\bibitem{hutzenthaler2012}
     \newblock  M.\ Hutzenthaler, A.\  Jentzen and P.\ E.\ Kloeden, 
     \newblock Strong convergence of an explicit numerical method for SDEs with nonglobally Lipschitz continuous coefficients,
     \newblock \emph{The Annals of Applied Probability}, \textbf{22(4)}, pp.\ 1611--1641, 2012.


\bibitem{kumar2020}
\newblock  C.\ Kumar and Neelima, 
\newblock On explicit Milstein-type schemes for McKean--Vlasov Stochastic Differential Equations with super-linear drift coefficient,
\newblock \emph{arXiv:2004.01266}, 2020. 


\bibitem{kumar2017b}
     \newblock  C.\ Kumar and S.\ Sabanis, 
     \newblock On explicit approximations for L\'evy driven SDEs with super-linear diffusion coefficients,
     \newblock \emph{Electronic Journal of Probability}, \textbf{22(73)}, pp.\ 1--19, 2017.

 \bibitem{kumar2019}
     \newblock  C.\ Kumar and S.\ Sabanis, 
     \newblock On Milstein approximations with varying coefficients: the case of super-linear diffusion coefficients, 
     \newblock  \emph{BIT Numerical Mathematics}, \textbf{59(4)}, pp.\ 929--96, 2019.

\bibitem{ledger2019}     
  \newblock S.\ J.\ Ledger and A.\ S\o jmark, 
  \newblock  At the mercy of the common noise: Blow-ups in a conditional McKean--Vlasov Problem,
  \newblock \emph{arXiv:1807.05126}, 2018.      
 
\bibitem{mckean}
\newblock H.\ McKean, 
\newblock A class of Markov processes associated with nonlinear parabolic equations, 
\newblock \emph{Proceedings of the National Academy of Sciences of the USA}, vol.\ 56(6), pp.\ 1907-1911, 1966. 

\bibitem{MVA}
\newblock Y.\ S.\ Mishura, A. Y.\ Veretennikov, 
\newblock Existence and uniqueness theorems for solutions of McKean--Vlasov stochastic equations,
\newblock \emph{arXiv:1603.02212}, 2018.  

\bibitem{pham2016}
    \newblock H.\ Pham, 
    \newblock Linear quadratic optimal control of conditional McKean--Vlasov equation with random coefficients and applications,
    \newblock \emph{Probability, Uncertainty and Quantitative Risk}, vol. 1:7, 2016.
%
%
   
%
    
     \bibitem{reis2019a}
     \newblock G. dos Reis, S.\ Engelhardt and G.\ Smith,
     \newblock Simulation of McKean--Vlasov SDEs with super linear growth, 
     \newblock  \emph{arXiv:1808.05530v2}, 2019.  
    
 \bibitem{reis2019b}
     \newblock G.\ dos Reis, W.\ Salkeld and J.\ Tugaut,
     \newblock Freidlin-Wentzell LDP in path space for McKeanVlasov equations and the functional iterated logarithm law, 
     \newblock  \emph{Ann. Appl. Probab.}, \textbf{29(3)}, pp.\ 1487--1540, 2019.    
     
\bibitem{sabanis2013}
    \newblock  S.\ Sabanis, 
    \newblock A note on tamed Euler approximations,
    \newblock \emph{Electronic Communications in Probability}, \textbf{18}, pp.\ 1-10, 2013.
%
\bibitem{sabanis2016}
     \newblock S.\ Sabanis, 
     \newblock Euler approximations with varying coefficients: the case of superlinearly growing diffusion coefficients,
     \newblock \emph{The Annals of Applied Probability}, \textbf{26(4)}, pp.\ 2083--2105, 2016.
    
\bibitem{siska2020}
 \newblock D.\ \v{S}i\v{s}ka and L.\ Szpruch,
 \newblock Gradient flows for regularized control problems,
 \newblock Working paper, 2020.

\bibitem{AS}
\newblock  A.-S.\ Sznitman, 
\newblock  Topics in Propagation of Chaos,
\newblock \emph{\'Ecole d'\'{e}t\'{e} de probabilit\'{e}s de Saint-Flour XIX - 1989}, vol.\ 1464 of Lecture notes in Mathematics, Springer-Verlag, 1991.
        
     
\bibitem{ullner2018}
\newblock E.\ Ullner, A.\ Politi and A.\ Torcini, 
\newblock Ubiquity of collective irregular dynamics in balanced networks of spiking neurons Chaos,
\newblock \emph{An Interdisciplinary Journal of Nonlinear Science}, \textbf{28(8)}, 2018. 

\bibitem{villani2009}
\newblock C.\ Villani, 
\newblock Optimal Transport: Old and New, 
\newblock \emph{Springer--Verlag}, Berlin, 2009.

%
%


\end{thebibliography}
\end{document}